\documentclass[12pt,oneside,a4paper]{amsart}
\usepackage[british]{babel}

\usepackage{ upgreek }
\usepackage{graphicx}
\usepackage{amscd}
\usepackage{amsmath}
\usepackage{amsthm}
\usepackage{amsfonts}
\usepackage{latexsym}
\usepackage{amssymb}
\usepackage{mathrsfs}
\usepackage{comment}
\excludecomment{mycomment}
\usepackage{color}
\usepackage[usenames,dvipsnames,table,xcdraw]{xcolor}
\usepackage{pdfsync}
\usepackage{bbm}
\usepackage{dsfont}
\usepackage[colorlinks=true]{hyperref}
\usepackage{booktabs}
\usepackage{float}
\usepackage{wrapfig}
\usepackage{caption}
\usepackage{subcaption}
\usepackage{longtable}
\usepackage{listings}
\usepackage[T1]{fontenc}
\usepackage[scaled]{beramono}
\usepackage{tikz}
\usepackage{pgffor}
\usepackage[all]{xy}
\usepackage{bbold}
\usepackage{enumitem}
\definecolor{trp}{rgb}{1,1,1}

\definecolor{red}{rgb}{1,0,.2}

\usepackage{amsthm}
\theoremstyle{plain}
\newtheorem{theorem}{Theorem}[section]

\newtheorem{claim}[theorem]{Claim}
\newtheorem{fact}[theorem]{Fact}

\newtheorem{corollary}[theorem]{Corollary}

\newtheorem{definition}[theorem]{Definition}
\newtheorem{example}[theorem]{Example}

\newtheorem{lemma}[theorem]{Lemma}

\newtheorem{prop}[theorem]{Proposition}
\newtheorem{proposition}[theorem]{Proposition}

\newtheorem{remark}[theorem]{Remark}

\numberwithin{equation}{section}

\def \N {\mathbb N}
\def \Z {\mathbb Z}
\def \Q {\mathbb Q}
\def \R {\mathbb R}

\def \ind{1\!\!1}

\newcommand*{\arabicdec}[1]{\the\numexpr\value{#1}\relax}

\graphicspath{{figures/}}

\usepackage[foot]{amsaddr}

\usepackage{anysize}
\papersize{24.5cm}{16.2006cm}

\marginsize{1.5cm}{1cm}{0cm}{0cm}

\usepackage{caption}

\definecolor{blue}{rgb}{0,0,1}

\definecolor{red}{rgb}{1,0,.7}

\title[Projections of the random Menger Sponge]{Projections of the random Menger Sponge}

\author{K\'aroly Simon$^{1,2,3}$}
\address{$^1$Department of Stochastics, Institute of Mathematics, Budapest University of Technology and Economics, M\H{u}egyetem rkp. 3., H-1111 Budapest, Hungary\newline \indent $^2$ Alfréd Rényi Institute of Mathematics – Eötvös Loránd Research Network, Re\'altanoda u. 13-
15., 1053 Budapest, Hungary}
\email{simonk@math.bme.hu}

\author{Vilma Orgov\'anyi$^{1,3}$}
\address{$^3$ MTA-BME Stochastics Research Group, Műegyetem rkp.
3., H-1111 Budapest, Hungary} 
\email{orgovanyi.vilma@gmail.com}

\date{\today}
\begin{document}
\maketitle
\begin{abstract}
	Using a similar random process to the one which yields the fractal percolation sets, starting from the deterministic Menger sponge we get the random Menger sponge. We examine its orthogonal projections from the point of Hausdorff dimension, Lebesgue measure and existence of interior points. 
	
	We obtain these results as special cases of our theorems stated for 
random self-similar IFSs.
These 
 are obatained by a random process similar to the fractal percolation, applied for the cylinder sets of  a deterministic self-similar IFS, as in \cite{falconer2014exact}.
In this paper the associated deterministic IFS on the line is of the special form 
$\mathcal{S}=\left\{\frac{1}{L}x+t_i  \right\} _{i=1}^{m}$,  where $L\in\mathbb{N}$, $L\geq 2$ and $t_i\in\mathbb{Q}$.
\end{abstract}
\tableofcontents
\thispagestyle{empty}

\section{Introduction}\label{y19}

\subsection{Brief summary}\label{u68}
The Mandelbrot percolation Cantor set is a two-parameter family of random sets on $\mathbb{R}^d$.
Namely, fix the parameters $K\geq 2$ and $p\in(0,1)$. In the first step of the construction
we partition the unit cube $[0,1]^d$ into 
axes-parallel cubes of side length $1/K$.  Each of these cubes are retained   with probability $p$ and discarded with  
probability $1-p$ independently. This step is repeated independently in each of the retained cubes 
ad infinitum or until no retained cubes left. The random set we end up with is the Mandelbrot percolation set. 

Falconer and Jin introduced a generalization of the  Mandelbrot percolation Cantor sets \cite{falconer2014exact}. 
In this paper we consider a special case of 
Falconer and Jin's construction. 

Namely, we apply the random process, which results  in the 
Mandelbrot percolation set, on the subsequent level cylinders in the construction of a self-similar set. 

More precisely, to understand the nature of the 
 above-mentioned random process, we consider a (deterministic) $M$-array tree $\mathcal{T}$.    That is every node of $\mathcal{T}$ has exactly $M$
children. (In the construction of the Mandelbrot percolation set $M=K^d$.)
We assign a random label (from $\{0,1\}$) to each of these nodes. The label of the root $\emptyset $ is equal to $1$ and  the random label of all other nodes are independent $\text{Bernoulli}(p)$
random variables. A level $n$ node is retained if all of its ancestors are labelled with $1$. 
In the mandelbrot percolation example, 
every retained level-$n$ node naturally corresponds to a retained level $n$ cube.  
An infinite path starting from the root is retained if all the nodes of the path are labelled with $1$. It may happen that no infinite paths are retained. This event is called extinction. The set of retained level $n$ nodes is denoted by $\mathcal{E}_n$ for an $n\in \mathbb{N}\cup \{\infty  \}$.
In the case of the Mandelbrot percolation, every element of $\mathcal{E}_{\infty  }$
naturally correspond to a point of the Mandelbrot percolation Cantor set. 

A generalization of the Mandelbrot percolation sets 
can be obtained if we consider a self-similar IFS (Iterated Function System)
$\mathcal{F}:=
\left\{ f_i \right\} _{i=1}^{M}$ on $\mathbb{R}^d$
and, we retain the elements of the attractor of $\mathcal{F}$ having a symbolic representation from 
$\mathcal{E}_{\infty  }$. We call these random sets 
\texttt{coin tossing self-similar sets} since we decide if a cylinder set is retained or not as a result of subsequent coin-tossing. For more detailed notations see Definition \ref{y36}.

If the self-similar IFS  is the one whose attractor is the so-called Menger sponge then this random construction yields the random Menger sponge.  More precisely, first we define 
the (deterministic) Menger sponge, which is the natural three-dimensional generalization of the well-known Sierpi\'nski-carpet. 
\begin{figure}[ht]
    \centering
    \includegraphics[width=\textwidth]{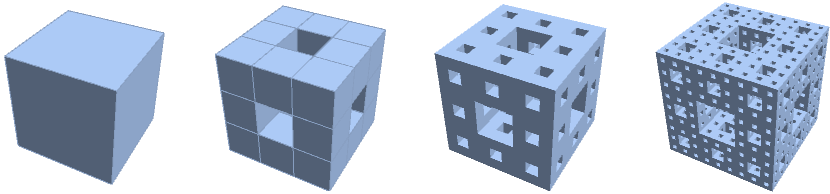}
    \caption{The zeroth, first, second and third level approximation of the Menger sponge.}
    \label{y31}
\end{figure}
We partition the unit cube $[0,1]^3$ into 27 cubes of side length $1/3$. For each of the $6$ 
faces of $[0,1]^3$ there is a central cube.  We remove these $6$ cubes as well as  the cube which contains the center of $[0,1]^3$. We retain the remaining 20 cubes. Then we repeat this process in each of the retained cubes ad infinitum. The set which remains after infinitely many steps is the Menger sponge $\mathcal{M}$.  See Figure \ref{y31} for the first four approximations of $\mathcal{M}$. 

To construct a random Menger sponge we need a parameter $p\in(0,1)$. Instead of retaining all of the $20$ cubes we retain each with probability $p$ 
and discard it with probability $1-p$ independently. Then we repeat this random construction in each retained cube ad infinitum or until there is no retained cubes left since it is possible that in finitely many steps we have no cubes retained. 
It follows from  the theory of Branching Processes that 
whenever $p>\frac{1}{20}$ we end up with a non-empty Cantor set 
with positive probability. This random set is called random Menger sponge corresponding to probability $p$ and it is denoted by $\mathcal{M}_p$. For the precise definition see Example \ref{y29}.
We study the structure of projections of $\mathcal{M}_p$ to lines trough the origin of the form
$\text{proj}_{(a,b,c) }(x,y,z):=ax+by+cz$. We say that $\text{proj}_{(a,b,c) }$ is rational if
$a,b,c\in \mathbb{Q}$. It follows from Mattila's Theorem \cite{matt_75}
that for Lebesgue almost all $(a,b,c)$ the projection of $\mathcal{M}_p$
\begin{enumerate}
    \item has Hausdorff dimension $\min\{1, \text{dim}_H(\mathcal{M}_p)\}$,
    \item if $\text{dim}_H(\mathcal{M}_p)>1$ then $\mathcal{L}eb_1(\text{proj}_{(a,b,c)}\mathcal{M}_p)>0$.
\end{enumerate}
Here we prove results of similar nature but ones which are much more precise in this special case.

\subsection{Pictural explanation of our results about $\mathcal{M}_p$}\label{u60}

Our results regarding the random Menger sponge  
(Theorems \ref{y52}-\ref{y54}) are summarized on Figure \ref{y10}.
Many of these are special cases of 
theorems that we prove in this paper for 
more general coin tossing self-similar IFSs.
A number of results presented on Figure \ref{y10},
are about the $(1,1,1)$-projection of $\mathcal{M}_p$.
We remark that corresponding results can be proved for rational projections satisfying some additional conditions. These conditions can be easily checked for the $(1,1,1)$-projection and in this case, we can determine the parameter intervals where these additional conditions hold.  

\subsubsection{Explanation of Figure \ref{y10}}
Under the baseline there is information about the $\text{proj}_{(1,1,1)}$-projected set. Information about the random Menger sponge itself are visualized  by colored rectangles arrond the $p$-axis, and lastly the information about the $\text{proj}_{(1,0,0)}$-projection (i.e. the projection to the coordinate axis) are placed above the baseline. 

\begin{figure}[h!]
    \centering
    \includegraphics[width=\textwidth]{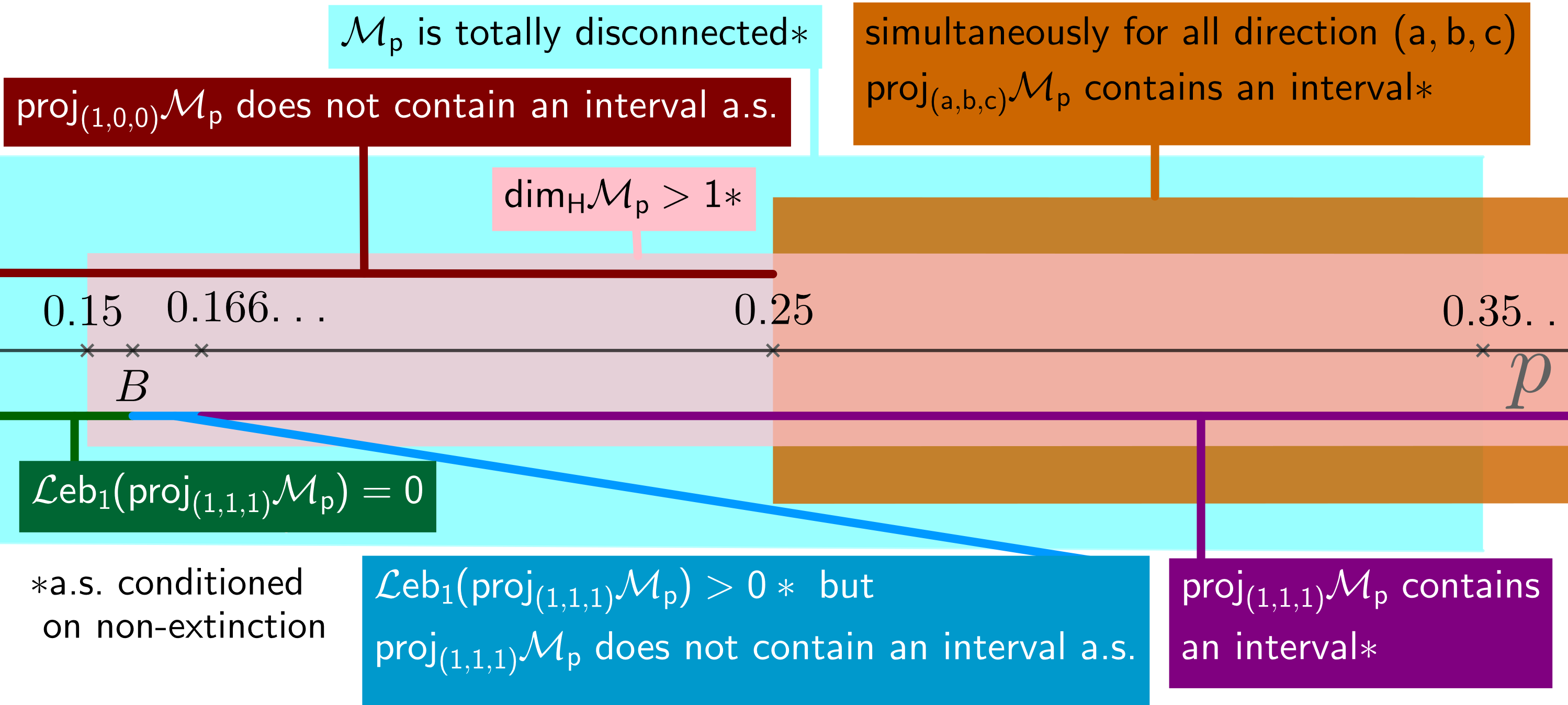}
    \caption{The parameter intervals for $p$ in the case of the random Menger sponge.}
    \label{y10}
\end{figure}
In what follows we present some further explanation to this figure. First consider the $\text{proj}_{(1,1,1)}$-projection of $\mathcal{M}_p$.
It follows from Theorem \ref{y54} that there exists a  $B \in (0.15,0.1514\dots)$ such that:

\begin{enumerate}[label=\emph{\alph*})]

     \item  For $p \in$ \color{OliveGreen}$\left(0.15, B\right)$
     \color{black} the  Lebesgue-measure of
 $\text{proj}_{(1,1,1)}(\mathcal{M}_{p})$ is
     zero almost surely, by Theorem \ref{y56}, despite the fact that $\mathcal{M}_p$ has Hausdorff 
     dimension greater than 1 almost surely conditioned on non-extinction by Fact \ref{y58}.

     \item For $p \in$\color{Aquamarine}$\left(B, 0.1666\dots  \right)$\color{black}, \  
      conditioned on non-extinction,
     $\text{proj}_{(1,1,1)}(\mathcal{M}_p)$
     is a set of positive Lebesgue measure
     (Theorem \ref{y54})
     which contains no interior points (Theorem \ref{y55}) almost surely.
     \item For $p\in $  \color{Fuchsia}$ \left(0.166\dots  ,1 \right]$\color{black},\ 
     conditioned on non-extinction,
     $\text{proj}_{(1,1,1)}(\mathcal{M}_{p})$ contains an interval almost surely by Theorem \ref{y55}.
 \end{enumerate}

 About the other projections of the random Menger sponge $\mathcal{M}_{p}$  we can say the following:
 \begin{enumerate}[label=\emph{\alph*}),resume]
     \item Let  $p\in $ \color{Orange}$\left(0.25,1\right]$\color{black}. 
     Then simultaneously for all directions $(a,b,c)$ the projection contains an interval almost surely conditioned on non-extinction by Theorem \ref{y53}. 
 \end{enumerate}
This is interesting when $p<0.35\dots$ since in this case the Menger sponge itself is totally disconnected
by Theorem \ref{y52}.
 Also  the bound $0.25$ is sharp because for $p<0.25$ the $\text{proj}_{(1,0,0)}$-projection does not contain an interval almost surely, by Theorem \ref{y22}.
As we already mentioned the results above are special cases of our more general theorems, which are presented in Section \ref{u70}. To state them we need to introduce our basic notations.
\subsection{Notations}\label{a99}
 Before we give the precise definition of the coin tossing self-similar sets, first we define the deterministic self-similar sets in $\R^d$.
Given a self-similar  IFS $\mathcal{F}$ on $\mathbb{R}^d$ 
\begin{equation}
\label{y34}
\mathcal{F}:=\left\{ 
    f_i(x):=r_iQ_ix+t_i
 \right\}_{i=0}^{M-1}
 , \, f_i:\mathbb{R}^d\to\mathbb{R}^d,\ 
 r_i\in(0,1),\, Q_i \in O(d),
\, t_i\in\mathbb{R}^d.
\end{equation}
We use the 
short hand notations
\begin{equation}
\label{y33}
f_{i_1,\dots,i_n}:=f_{i_1}\circ \cdots\circ f_{i_n},\; r_{i_1,\dots,i_n}:=r_{i_1} \dots r_{i_n},\;
[M]:=\left\{ 0,\dots  ,M-1 \right\}.
\end{equation}
It is easy to see that we can choose 
\begin{equation}
\label{y35}
B\subset \R^d\text{, compact such that} \quad f_i(B) \subset B \quad \text{for all } i \in [M].
\end{equation}

Then the union of all $n$-cylinders $\bigcup \limits_{i_1,\dots  ,i_n}f_{i_1,\dots,i_n}(B)$ form a nested sequence of compact sets. Their intersection is the attractor
$$
\Lambda :=\bigcap\limits_{n=1}^{\infty   }
\bigcup \limits_{i_1,\dots  ,i_n}f_{i_1,\dots,i_n}(B).
$$

The definition of $\Lambda$ does not depend on the choice of $B$ as long as $B$ satisfies \eqref{y35}.

\begin{definition}[Coin tossing self-similar sets]\label{y36}
Let $\mathcal{F}:=\left\{f_i\right\}_{i=0}^{M-1}$ be a (deterministic) self-similar IFS on $\mathbb{R}^d$ as it was defined in \eqref{y34} and let $p\in (0,1)$. The corresponding coin tossing self-similar set $\Lambda_{\mathcal{F}}(p)$
is defined as follows: In the first step for every $k\in[M]$ we toss (independently) a bias coin which lands on head with probability $p$. The random subset $X_1 \subset [M]$ consists of those $k \in [M]$ for which the  coin tossing resulted in head. Assume that we have already constructed $X_n \subset [M]^n$. Then for every node $\mathbf{i}\in X_n$ we define (independently of everything) the random set $X_1^{\mathbf{i}} \subset [M]$ which has the same distribution as $X_1$.
 The set of the
 offsprings of $\mathbf{i}$ is defined by
 $O(\mathbf{i})=\left\{\mathbf{i}k\in[M]^{n+1}:k\in X_1^{\mathbf{i}} \right\}$, where  $\mathbf{i}k=i_1, \dots, i_n,k$ if $\mathbf{i}=i_1, \dots ,i_n$.
Finally, we form $X_{n+1} = \bigcup_{\mathbf{i}\in X_n}O(\mathbf{i}) \subset [M]^{n+1}$. Then the coin tossing self-similar set is defined by
 \begin{equation}\label{y32}
 \Lambda_{\mathcal{F}}(p):=\bigcap\limits_{n=1}^{\infty }\bigcup\limits_{\mathbf{i}\in X_n}f_{\mathbf{i}}(B),
 \end{equation}   
where $B$ is chosen as in \eqref{y35}. 
\end{definition}
We do not assume that the Open Set Condition (OSC) (see \cite{falconer2004fractal}  )
holds for $\mathcal{F}$. However, we mention the following theorem.
\begin{theorem}[Falconer \cite{falconer1986random}, Mauldin-Williams \cite{mauldin1986random}]\label{u62}
Let $\mathcal{F}$ be a deterministic self-similar IFS, as in Definition \ref{y36}, which satisfies the OSC.
Then alost surely, conditioned on non-extinction,
\begin{equation}
\label{u65}
\dim_{\rm H}  \Lambda_{\mathcal{F}}(p)
=\dim_{\rm B}  \Lambda_{\mathcal{F}}(p)
=s, \text{ where }
\sum_{i=0}^{M-1 }pr _{i}^{s }=1,
\end{equation}
where $r_i$ is the contraction ratio of the similarity mapping $f_i$.
\end{theorem}
Motivated by this formula we introduce the similarity dimension of a coin tossing self-similar set $\Lambda_{\mathcal{F}}(p)$:
\begin{equation}
\label{u37}
\dim_{\rm Sim}  \Lambda_{\mathcal{F}}(p):=s, \text{ where }
\sum_{i=0}^{M-1 }pr _{i}^{s }=1.
\end{equation}

In this paper we only consider homogeneous coin tossing self-similar IFSs, which means that all contraction ratios $r_i$ are equal to the same $r\in (0,1)$. In this homogeneous case, formula \eqref{u65} simplifies to
\begin{equation}
\label{u61}
\left( \text{ OSC } \& \ 
r_i\equiv r,\  \forall i
\right) \Longrightarrow
s=\dim_{\rm H}  \Lambda_{\mathcal{F}}(p)
=\dim_{\rm B}  \Lambda_{\mathcal{F}}(p)=
\frac{\log (Mp)}{-\log r},
\end{equation}
almost surely, conditioned on non-extinction.

A more detailed definition of a coin tossing self-similar set, which describes the ambient probability space can be found in Section \ref{u69}.
\begin{example}[(Homogeneous) Mandelbrot percolation]\label{y30}
The homogeneous Mandelbrot percolation on $\R^d$ with parameters $(K,p)$ can be obtained as a special case of the construction defined above by choosing $M=K^d$ and
$$\mathcal{F}:=\left\{f_i(x)=
\frac{1}{K}x+t_i\right\}_{i=0}^{M-1},$$ where $\left\{t_i\right\}_{i=0}^{M-1}$ is an enumeration of the left bottom corners of the $K$-mesh cubes contained in $[0,1]^d$.  
\end{example}
\begin{example}[Random Menger sponge]\label{y29}
The (deterministic) Menger sponge is the attractor (see Figure \ref{y31}) of the following self-similar IFS in $\R^3$:
\begin{equation}\label{u89}
\mathcal{F}:
=\left\{f_{i}(\underline{x})=\frac{1}{3}(\underline{x})+t_i\right\}_{i=0}^{19},
\end{equation}
where $\left\{t_i\right\}_{i=0}^{19}$ is an enumeration of the following set
\begin{multline}\label{y59}
\left\{0,\frac{1}{3},\frac{2}{3}\right\}^3 \setminus \left\{\left(\frac{1}{3},\frac{1}{3},0\right), \left(\frac{1}{3},0,\frac{1}{3}\right), 
\left(0,\frac{1}{3},\frac{1}{3}\right), \left(\frac{1}{3},\frac{1}{3},\frac{2}{3}\right),\right .\\
\left .  \left(\frac{1}{3},\frac{2}{3},\frac{1}{3}\right),\left(\frac{2}{3},\frac{1}{3},\frac{1}{3}\right), \left(\frac{1}{3},\frac{1}{3},\frac{1}{3}\right)\right\}.
\end{multline}
We obtain the random Menger sponge by applying the random construction introduced in Definition \ref{y36} for the deterministic IFS above. As we have already mentioned we denote the random Menger sponge with parameter $p$ with $\mathcal{M}_p$.

\end{example}
If we project the random Menger sponge to straight lines, then the resulting random set is a coin tossing self-similar set on the line.
\begin{example}[Coin tossing integer self-similar sets on the line]\label{y28} 
We obtain the coin tossing integer self-similar sets on the line by applying the random construction introduced in Definition \ref{y36} for the following deterministic IFS: 
\begin{equation}
\mathcal{F}:=\left\{ 
    f_i(x):=\frac{1}{L}x+t_i
 \right\}_{i=0}^{M-1}
 , \, f_i:\mathbb{R}\to\mathbb{R},\ 
 L\in\N \setminus \left\{0,1\right\},
\, t_i\in \mathcal{N},
\end{equation}
where $\mathcal{N}\subset \R$ is a lattice.

\end{example}
\begin{remark}\label{u88}
Without loss of generality
(see \cite[Section 1.3.3]{SBS_book})
we may assume, that 
\begin{equation}\label{y27}
    \mathcal{N}:=\N,\quad 
    0=t_0\leq t_1\leq t_2\leq\dots \leq t_{M-1},\quad
    \text{ and } \quad t_{M-1}|L-1.
\end{equation}
 
\end{remark}
\begin{remark}\label{y15}
In the deterministic case usually we require that the elements of the IFS $\mathcal{F}$ are different. However, in the random case we allow repetition among the functions of $\mathcal{F}$. This is reasonable since even if $f_i=f_j$ we randomize them differently. For example, if $\mathcal{F}:= \left\{ \frac{1}{3}x,\frac{1}{3}x,\,\frac{1}{3}x+\frac{2}{3}\right\}$ and $\mathcal{S}:=\left\{ \frac{1}{3}x,\,\frac{1}{3}x+\frac{2}{3}\right\}$, then the coin tossing self-similar sets $\Lambda_{\mathcal{F}}(p)$, $\Lambda_{{\mathcal{S}}}(p)$ are different.  
\end{remark}

\begin{definition}\label{y14}
Given the deterministic IFS $\mathcal{F}:=\left\{f_i\right\}_{i=0}^{M-1}$ we select all distinct elements of $\mathcal{F}$ and form a new IFS from these functions, we denote it by $\mathcal{S}:=\left\{S_{i}\right\}_{i=0}^{m-1} \subset \mathcal{F}$. That is for every $i \in [M]$ there exist a unique $j \in [m]$ such that $f_i=S_j$. For every $j \in [m]$ let 
\begin{equation}\label{u80}
    n_j:= \#\left\{f_{i}\in \mathcal{F}: f_{i}=S_{j}\right\}
\end{equation}
 and we define the probability vector
 \begin{equation}
      \underline{q}:=(q_1, \dots, q_m)\text{, where } q_j:=\frac{n_{j}}{M}.
 \end{equation}
\end{definition}

For a $k\geq 2$ we write 
$\Sigma ^{(k)}:=[k]^{\mathbb{N}}$, where we remind the reader that we defined 
$[k]:=\left\{ 0,1,\dots  ,k-1 \right\}$.

We introduce the natural (probability) measure 
$\mu :=\underline{q}^{\N}$
on $\Sigma ^{(m)}$.
 We define the natural projection $\Pi^{(m)}:\Sigma^{(m)}\to \R$
\begin{equation}\label{y13}
\Pi^{(m)}({\tt i}):= \lim_{n\to \infty}S_{i_1, \dots, i_n}(0),
\end{equation}
for ${\tt i}=i_1,\dots,i_n\dots \in \Sigma^{(m)}$. The push forward of the natural measure $\mu$ is denoted by $\nu$,
\begin{equation}\label{y12}
    \nu:=\Pi^{(m)}_{*}\mu.
\end{equation}

Following Ruiz \cite[Section 3.1.1]{ruiz2009dimension} we define the L-adic intervals:
\begin{equation}\label{u87}
    \mathcal{D}_{k}:=\left\{\left[(i-1)\cdot L^{-k}, i\cdot L^{-k}\right]: i \in \Z\right\} , \quad \text{for }k\in \left\{-1, 0,\dots\right\}.
\end{equation}
Since $\nu$ is compactly supported there exist finitely many intervals, called basic types, 
\begin{multline}\label{u67}
    J^{0}, \dots, J^{N-1} \in \mathcal{D}_{-1}\text{ such that } \text{spt}(\nu)\subset \cup_{i \in [ N]  }J^{i}\text{ and }\\
    \nu(J^{i})>0\text{ for every }i \in [N],
\end{multline}
here we assume that the intervals $J^{0}, \dots, J^{N-1}$ are arranged in an increasing order.
It is clear from definition that the interval spanned by the attractor is 
$I:=\left[ 0,L\frac{ t_{m-1}}{L-1} \right]$.
It follows from our assumption  $L-1|t_{m-1}$ that the right endpoint of $J^{N-1}$ coincides with  the right endpoint of $I$ .
 For every $k\in [N] $ the interval $J^k$
   subdivides into $L^n$ intervals from $\mathcal{D}_{n-1}$ (of length $L^{-(n-1)}$) which are denoted by
   $J _{\underline{ a}}^{k }=J^{k}_{a_1,\dots,  a_{n}}$ for  $\underline{a}=(a_1,\dots  ,a_{n})\in [L]^{n}$.
We define the $N\times N$ matrices $\left\{A_a\right\}_{a=0}^{L-1}$:
\begin{equation}
\label{y89}
A_a(\ell, k)
:=
\left\{
\begin{array}{ll}
n_i
,&
\hbox{if $\exists i \quad S_i(J^k)=J^{\ell}_a$;}
\\
0
,&
\hbox{otherwise.}
\end{array}
\right.
\end{equation}
%
Note that we index the rows and columns of the matrices $A_{a}$ from $0$ to $N-1$.
For $\underline{a}=\left(a_1, \dots, a_n\right) \in [ L] ^n$,
\begin{equation}\label{u86}
 A_{\underline{a}}:=A_{a_1}\cdots A_{a_n}.
\end{equation}
The $j$-th column sum ($CS$) of the $a$-th matrix is denoted by  
\begin{equation}\label{y23}
  CS_{a,j}:=\sum_{i=0}^{N-1} A_{a}(i,j).
\end{equation}
For the rest of this section let $\Lambda_{\mathcal{F}}(p)$ denote the coin tossing integer self-similar set on the line, defined in Example \ref{y28} with parameters $\mathcal{F}$ and $p$.
\subsection{Results}\label{u70}
In the rest of this Section we deal with coin tossing integer self-similar IFSs $\mathcal{F}$ on the line introduced in Example \ref{y28} including Remark \ref{u88}. That is for the rest of this Section we always assume that
\begin{multline}
\label{u66}
\mathcal{F}:=\left\{ 
    f_i(x):=\frac{1}{L}x+t_i
 \right\}_{i=0}^{M-1}
 ,
\\
L\in\N \setminus \left\{0,1\right\},
\, t_i\in \mathbb{N}, 0=t_0\leq \cdots\leq t_{M-1},\ 
L-1|t_{M-1}.
\end{multline}
Moreover, throughout this Section we use the notation introduced in Section \ref{a99}. The new results of this paper are as follows:

\begin{theorem}\label{y26}
  The coin tossing integer self-similar set $\Lambda_{\mathcal{F}}(p)$ contains an interval, almost surely conditioned on non-extinction, if the following two conditions hold
\begin{enumerate}
    \item $p\cdot CS_{a,U}>1$ for all $a \in [L]$ and $U \in [N]$. That is $p$ is larger than the reciprocal of every column sum of every matrix. 
    \item There exists a $\underline{b}\in  [L]^*$ and $U \in [N]$ such that $A_{\underline{b}}(U,V)>0$ for all $V \in [N]$. That is there exist a product ($A_{\underline{b}}$, $\underline{b} \in [L]^*$) of the matrices with a strictly positive row.
\end{enumerate}

\end{theorem}

\begin{theorem}\label{y21}
    The coin tossing integer self-similar set $\Lambda_{\mathcal{F}}(p)$ does not contain any intervals,
almost surely, if 
 there exists an $a\in [L]$ such that the spectral radius of the matrix $p \cdot A_a$ is smaller than 1. 
\end{theorem}

\begin{theorem}\label{y25}
The coin tossing integer self-similar set $\Lambda_{\mathcal{F}}(p)$ has positive Lebesgue measure, almost surely conditioned on non-extinction, if the following two conditions hold:
\begin{enumerate}
    \item $p^L\cdot(\prod_{a=0}^{L-1} CS_{a,U})>1$ for all $U \in [N]$.  That is for every $U \in [N]$ column  index we consider the geometric mean of the $U$-th column sums of the matrices $\left\{A_{a}\right\}_{a=0}^{L-1}$ and we denote it by $g_{U}$. Our assumption is that $p>\max_{U \in [N]}\frac{1}{g_{U}}$.
    \item For every $b \in  [L]$ there exists an $U \in [N]$ such that $A_{b}(U,V)>0$ for all $V \in [N]$. That is, every matrix $A_{b}$ ($b \in [L]$) has a positive row.
\end{enumerate}
\end{theorem}
\begin{theorem}\label{y24}
There exists a $B>\frac{L}{M}$ such that for $p\in \left(0,B \right)$ the Lebesgue measure of the
coin tossing integer self-similar set
$\Lambda_{\mathcal{F}}(p)$ is almost surely 0. 
\end{theorem}
The relevance of the bound $\frac{L}{M}$ is that 
the similarity dimension $\dim_{\rm Sim}\Lambda_{\mathcal{F}}(p)>1  $ if and only if $p>\frac{L}{M}$.
Thus, for $p\in\left( \frac{L}{M},B \right)$
the similarity dimension of $\Lambda_{\mathcal{F}}(p)$
is greater than $1$ but its Lebesgue measure is zero almost surely.

In the following we state theorems about the random Menger sponge (defined in Example \ref{y29}) and its rational projections (see \eqref{y56}). As we mentioned earlier, part of them are consequences of the three theorems above, but some are special to the random Menger sponge $\mathcal{M}_p$. It is clear that the IFS which defines the Menger sponge (see Example \ref{y29}) 
satisfies the OSC. So, we can apply formula 
  \eqref{u61}, with the substitution 
  $r=1/3$ and $M=20$. This yields that
\begin{fact}\label{y58}
For almost all realizations conditioned on non-extinction we have
$\text{dim}_{H}\mathcal{M}_p=\frac{\log 20\cdot p}{\log 3}$, which is greater than $1$ if and only if 
$p>\frac{3}{20}=0.15$.
\end{fact}
Now we state the theorems regarding the random Menger sponge. The following theorems are visualized on Figure \ref{y10}. The contents of these theorems have already been mentioned in Section \ref{u60}. Here we just give a more precise formulation.

Recall, that  $\text{proj}_{(a,b,c)}(x,y,z): \R^3\to \R$ denote the scalar multiple of the orthogonal projection to the vector $(a,b,c)$,
\begin{equation}\label{y57}
\text{proj}_{(a,b,c)}(x,y,z):= ax+by+cz,
\end{equation} 
and the projection is rational if $a,b,c \in \Q$.

\begin{theorem}\label{y52}
If $p<\frac{1}{\sqrt{8}}=0.35\dots$, then $\mathcal{M}_{p}$ is totally disconnected almost surely conditioned on non-extinction.
\end{theorem}
\begin{theorem}\label{y53}
For any projection $(\text{proj}_{(a,b,c)})$ if $p>0.25$ then $\text{proj}_{(a,b,c)}(\mathcal{M}_{p})$ contains an interval almost surely conditioned on non-extinction.
\end{theorem}
\begin{theorem}\label{y56}
    For every rational projection $\text{proj}_{(a,b,c)}$
    there exists a $p'=p'(a,b,c)>\frac{3}{20}$ such that for every $p<p'$: $\mathcal{L}eb_1(\text{proj}_{(a,b,c)}(\mathcal{M}_p))=0$ almost surely.
    \end{theorem}
\begin{theorem}\label{y22}
If $p<0.25$ then the projection of the random Menger sponge onto the $x$-axis ($\text{proj}_{(1,0,0)}(\mathcal{M}_{p})$) does not contains an interval almost surely.
\end{theorem}
\begin{theorem}\label{y55}
If $p>\frac{1}{6}=0.1666\dots$, then $\text{proj}_{(1,1,1)}(\mathcal{M}_p)$ contains an interval almost surely conditioned on non-extinction, and if $p<\frac{1}{6}$, then $\text{proj}_{(a,b,c)}(\mathcal{M}_p)$ does not contain an interval almost surely.
\end{theorem}
\begin{theorem}\label{y54}
If $p>(8\cdot 6\cdot 6)^{-\frac{1}{3}}=0.1514\dots$, then $\mathcal{L}eb_1(\text{proj}_{(1,1,1)}(\mathcal{M}_p))>0$ almost surely conditioned on non-extinction.
\end{theorem}






\subsubsection{The organization of the rest of the paper}
Firstly in Section \ref{y18} we prove the theorems regarding the random Menger sponge (Theorem \ref{y52}-\ref{y54}).
Theorem \ref{y56}-\ref{y54} are special cases of Theorem \ref{y26}-\ref{y24}, hence we prove them assuming those more general theorems. For Theorem \ref{y52} we use a standard branching process type argument. To prove Theorem \ref{y53} we use the result of Simon and Vágó \cite{SV_2014}. 

In section \ref{y20} we present the proof of the general theorems, Theorem \ref{y26}-\ref{y24}.
We start with proving Theorem \ref{y26} using a similar argument as in \cite{SD2005}, secondly we prove Theorem \ref{y21} by a standard argument, then we prove Theorem \ref{y25} in a similar way as in \cite{SMS2009}. Finally, we prove Theorem \ref{y24} by firstly considering the deterministic attractor in a similar way as in  \cite{barany2014dimension} in a similar setup as in \cite{ruiz2009dimension}. Using this deterministic result we can prove Theorem \ref{y25}.

In the Appendix we present another example, namely two of the projections of random Sierpi\'nski carpet. We briefly summarize the related results, and add some new by applying our theorems. We then compare the results regarding the random Menger sponge and the random Sierpi\'nski carpet.

\section{Menger sponge}\label{y18}
 We consider the projections
$\text{proj}_{(1,0,0)}$ and  $\text{proj}_{(1,1,1)}$ of $\mathcal{M}_p$. Using that
all rational projections of  $\mathcal{M}_p$
(after proper re-scaling)
can be considered as coin tossing integer self-similar sets,
 we can use the general theorems (Theorem \ref{y26}-\ref{y24}) for the proofs of Theorems \ref{y56}-\ref{y54}.

\begin{example}[The $\text{proj}_{(1,0,0)}$ projection of the random Menger sponge]\label{u50}
We project the random Menger sponge to the x-axis, hence the corresponding IFS is the following (see Definition \ref{y14}):
$S_i(x)= \frac{1}{3}x+t_i$, $t_i \in \left\{0,3,6\right\}$, with $\left(n_0,n_1,n_2\right)=\left(8,4,8\right)$, and basic types (see \eqref{u67}) $J^0=\left[0,3\right]$, $J^1=\left[3,6\right]$, $J^2=\left[6,9\right]$. Hence, the corresponding matrices (see \eqref{y89}) are,
\begin{equation}
B_0=
\begin{bmatrix}
8 & 0 & 0\\
4 & 0 & 0 \\
8 & 0 & 0
\end{bmatrix}, \quad
B_1=
\begin{bmatrix}
0 & 8 & 0\\
0 & 4 & 0 \\
0 & 8 & 0
\end{bmatrix}, \quad
B_2=
\begin{bmatrix}
0 & 0 & 8\\
0 & 0 & 4 \\
0 & 0 & 8
\end{bmatrix}.
\end{equation}
\end{example}

\begin{figure}[h!]
    \centering
  \includegraphics[width=0.8\linewidth]{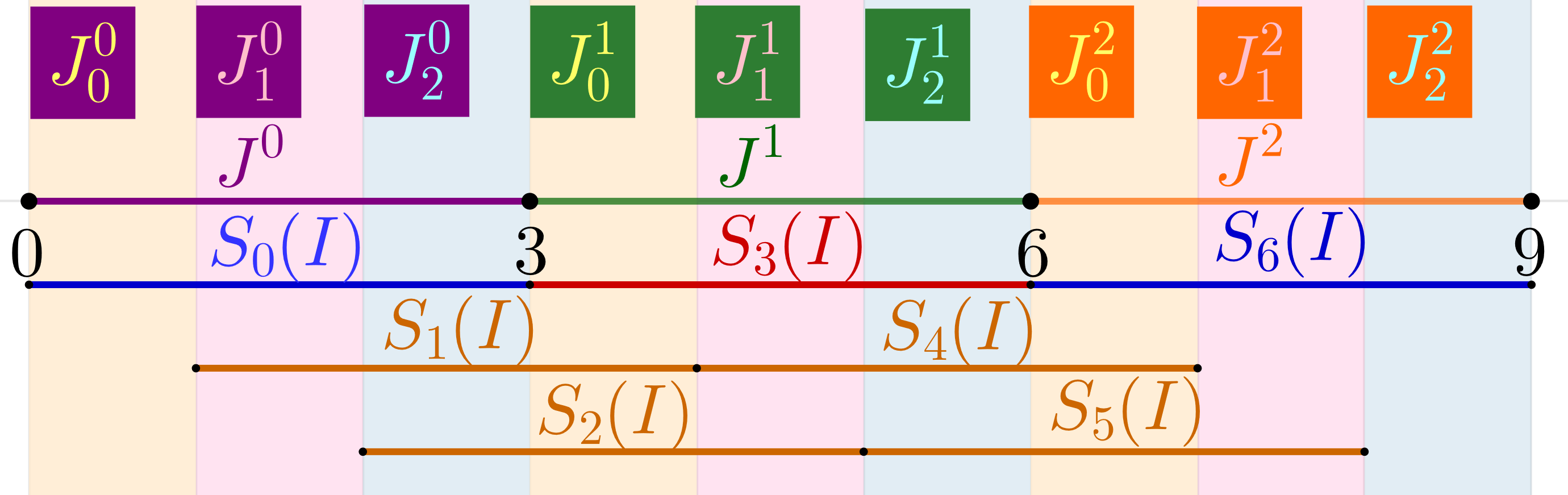}
  \caption{Illustration of Example \ref{y87}}\label{y46}
\end{figure}

\begin{example}[The proj$_{(1,1,1)}$ projection of $\mathcal{M}_p$]\label{y87}
$S_i(x)= \frac{1}{3}x+i$, $i \in [7]$. $\left(n_0,\dots, n_6\right)=\left(1,3,3,6,3,3,1\right)$. $J^0=[0,3]$, $J^1=[3,6]$, $J^2=[6,9]$. For an illustration see Figure \ref{y90}.
\begin{equation}      
A_0=
\begin{bmatrix} 
1 & 0 & 0\\
6 & 3 & 3 \\
1 & 3 & 3
\end{bmatrix}, \quad
A_1=
\begin{bmatrix}
3 & 1 & 0\\
3 & 6 & 3 \\
0 & 1 & 3
\end{bmatrix}, \quad
A_2=
\begin{bmatrix}
3 & 3 & 1\\
3 & 3 & 6 \\
0 & 0 & 1
\end{bmatrix}
\end{equation}
\end{example}
\begin{proof}[Proof of Theorem \ref{y22}]
The statement follows from \ref{y21}, as the Theorem is about the $\text{proj}_{(1,0,0)}$-projection (Example \ref{u50}). An easy calculation yields that the spectral radius of $p\cdot B_1$ is $4\cdot p$ which is less than 1 for $p<0.25$.
\end{proof}

\begin{proof}[Proof of Theorem \ref{y55}]
    Here we use what we wrote in Example \ref{y87}.
 The first part follows from Theorem \ref{y26}.
 The second part is an immediate consequence of  Theorem \ref{y21}. This is so, because  the spectral radius of $p \cdot A_0$ in Example \ref{y87} is $6p$, which is less than $1$ whenever $p<\frac{1}{6}$.
\end{proof}

\vspace{-0.5cm}
\begin{proof}[Proof of Theorem \ref{y54}]
The theorem is an easy consequence of Theorem \ref{y25} using again Example \ref{y87}.
\end{proof}

\vspace{-0.5cm}
\begin{proof}[Proof of Theorem \ref{y56}]
The proof is straightforward using Theorem \ref{y24}.
\end{proof}

\vspace{-1cm}
\subsection{Proof of Theorem \ref{y52}}
 \begin{lemma}\label{v95}
 We consider two independent copies of $\mathcal{M}_p$, and we translate 
one of them by $(1,0,0)$. We call them 
$\mathcal{M}_p^1$ and  $\mathcal{M}_p^2$. 
We denote their intersection by $\mathcal{C}$. That is
 $\mathcal{C}=\mathcal{M}_p^1 \cap \mathcal{M}_p^2$
as in Figure \ref{u84}. 
\begin{figure}[h!]
    \centering
    \includegraphics[width=0.5\textwidth]{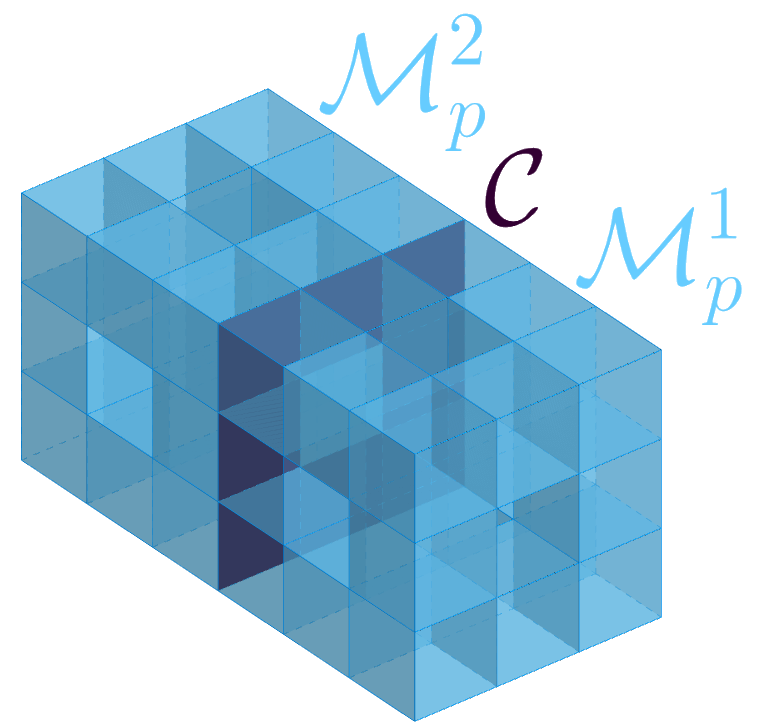}
    \caption{Explanation of notations of Lemma \ref{v95}.}
    \label{u84}
\end{figure}
If $p>\frac{1}{\sqrt{8}}=0.35\dots  $ then there  cannot exist a path from $\mathcal{M}_p^1$
to $\mathcal{M}_p^2$ through $\mathcal{C}$, almost surely.
\end{lemma}
\begin{proof}
Assume that we have a path from $\mathcal{M}_{p}^{1}$ to $\mathcal{M}_{p}^{2}$ through $\mathcal{C}$ with positive probability. Then for every  $n$, with positive probability, there exist a pair of retained level $n$ cubes which are on opposite  sides of $\mathcal{C}$ and share a common face (contained in $\mathcal{C}$). Observe that the number of such level $n$ pairs -- denote it by $\varepsilon_n$ -- forms a Galton-Watson branching process with offspring-distribution $Bin(8, p^2)$. It follows from the Theory of Branching Processes (see \cite[page 7]{athreya-ney}), that the process dies out almost surely in finite steps whenever 
\[
8\cdot p^2=\mathbb{E}(\varepsilon_n)<1,
\]
i.e. when $p<\frac{1}{\sqrt{8}}$. Now if the process dies out in finite steps almost surely, then there cannot exists a connection between $\mathcal{M}_{p}^{1}$ and $\mathcal{M}_{p}^{2}$ through $\mathcal{C}$.
\end{proof}

\begin{proof}[Proof of theorem \ref{y52}]
Fix $p_0<\frac{1}{\sqrt{8}}$. Assume that the random Menger sponge $\mathcal{M}_{p_0}$ contains a connected component with positive probability. Then with positive probability there exists an  $n$ such that with positive probability we can find two level $n$ retained adjacent cube that are connected through their common face say $\mathcal{D}_1$ and $\mathcal{D}_2$. By statistical self-similarity, conditioned on $\mathcal{D}_1$ and $\mathcal{D}_2$ are two retained level $n$ cubes, they are independent, rescaled copies of the random Menger sponge $\mathcal{M}_p$. Hence, by Lemma \ref{v95} with probability one, $\mathcal{D}_1$ and $\mathcal{D}_2$ can not be connected through their common wall, which is a contradiction.
\end{proof}

\subsection{The dual nature of $\mathcal{M}_p$}
All until now we have always looked at $\mathcal{M}_p$
as an example of a coin tossing self-similar set. 
However, to prove Theorem \ref{y53}, we need another interpretation of $\mathcal{M}_p$ as a 
 three-dimensional inhomogeneous Mandelbrot percolation set. The random set 
what we call inhomogeneous Mandelbrot percolation
in this paper simply named  Mandelbrot percolation 
in \cite{SR2015}, \cite{SV_2014}. 
 The motivation for this different interpretation of $\mathcal{M}_p$ is that we would like to use the projection theorems of the 
 paper \cite{SV_2014}, which are stated for the 
 inhomogeneous Mandelbrot percolation sets.
To construct $\mathcal{M}_p$ as an inhomogeneous Mandelbrot percolation set, we divide the unit cube 
$Q:=[0,1]^3$ into $27$ axes parallel cubes of side length $1/3$. 
$$
\left\{\mathfrak{K}_{ \left(i,j,k\right)}     \right\} _{i,j,k=0}^{2 }, \text{ where }
\mathfrak{K}_{ \left(i,j,k\right)}:=
\left[\frac{i}{3},\frac{i}{3}+\frac{1}{3}  \right]
\times
\left[\frac{j}{3},\frac{j}{3}+\frac{1}{3}  \right]
\times
\left[\frac{k}{3},\frac{k}{3}+\frac{1}{3}  \right].
$$
Similarly, we define the level $n$ cubes 
$$
\mathcal{\mathfrak{K}}_n:=
\left\{ 
    \mathfrak{K}_{(\mathbf{i},\mathbf{j},\mathbf{k})}:=
(\overline{\mathbf{i}},\overline{\mathbf{j}},\overline{\mathbf{k}})+\left[ 0,3^{-n} \right]^3
 \right\}_{\mathbf{i},\mathbf{j},\mathbf{k}\in\left\{ 0,1,2 \right\}^n},
$$
where $\mathbf{i}:=(i_1,\dots  ,i_n)$ and
$\overline{\mathbf{i}}:=\sum_{\ell =1}^{n}i_{\ell }3^{-\ell }$

Let 
\begin{multline}\label{u58}
    \widetilde{\mathcal{A}}:=\{(1,1,0), (1,0,1), (0,1,1), (1,1,2), (1,2,1), (2,1,1), (1,1,1)\}.
\end{multline}
Note that these indices corresponds to the cubes which 
are deleted when we construct the first approximation of the deterministic Menger sponge. 
For each $(i,j,k)\in \left\{ 0,1,2 \right\}^3$ we set  a probability $p_{(i,j,k)}\in [0,1]$ as follows:
\begin{equation}
\label{u51}
p_{(i,j,k)}:=
\left\{
\begin{array}{ll}
0
,&
\hbox{if $(i,j,k)\in \widetilde{\mathcal{A}} $;}
\\
p
,&
\hbox{otherwise.}
\end{array}
\right.
\end{equation}
We retain the cube $\mathfrak{K}_{ \left(i,j,k\right)}$ with probability $p_{(i,j,k)}$ independently. Assume that we have already constructed the level $n$ retained cubes. That is we are given the random set $\mathcal{E}_n\subset \left\{ 0,1,2 \right\}^{n}\times \left\{ 0,1,2 \right\}^{n}\times \left\{ 0,1,2 \right\}^{n}$ such that after $n$ steps we have retained the cubes 
$\left\{ \mathfrak{K}_{(\mathbf{i},\mathbf{j},\mathbf{k})} \right\} _{(\mathbf{i},\mathbf{j},\mathbf{k})\in \mathcal{E}_n}$. For every $(\mathbf{i},\mathbf{j},\mathbf{k})\in \mathcal{E}_n$ we retain the cube
$\mathfrak{K}_{(\mathbf{i}i_{n+1},\mathbf{j}j_{n+1},\mathbf{k}k_{n+1})}$ with probability $p_{(i_{n+1},j_{n+1},k_{n+1})}$ independently of everything. Let $E_n:=\bigcup\limits_{(\mathbf{i},\mathbf{j},\mathbf{k})\in \mathcal{E}_n}
\mathfrak{K}_{(\mathbf{i},\mathbf{j},\mathbf{k})} 
$. Then $\mathcal{M}_p=\bigcap\limits _{n=1}^{\infty  }
E_n$.

\begin{proposition}\label{w99}
    For any plane $S(a,b,c)$:
    \begin{equation}\label{w97}
        \frac{\mathcal{L}eb_2(Q \setminus M_1\cap S(a,b,c))}{\mathcal{L}eb_2(Q\cap S(a,b,c))}\leq \frac{5}{9}.
    \end{equation}
    where $\mathcal{L}eb_2$ denotes the two dimensional Lebesgue measure, $M_1$ denotes the first approximation of the deterministic Menger sponge. Finally, $S(a,b,c)$ denote the plane 
\begin{equation}\label{v98}
    S(a,b,c):=\left\{(x,y,z)\in \R^3: ax+by+c=z\right\}.
\end{equation}
    \end{proposition}

\subsection{Proof of Theorem \ref{y53} assuming Proposition \ref{w99}}
We use some of the results of  Simon and Vágó \cite{SV_2014} on projections of inhomogeneous Mandelbrot percolations. In \cite{SV_2014} Simon and Vágó introduces conditions under which the projections of the inhomogeneous Mandelbrot percolation contains an interval almost surely conditioned on non-extinction. In what follows we show that for $p>0.25$ these conditions are satisfied. 
 Let
$\Pi_{(a,b)}$ denote the projection along the planes $S(a,b,.)$ to the $z$-coordinate axis, namely
\begin{equation}\label{v97}
    \Pi_{(a,b)}(x,y,z):= z-ax-by.
\end{equation}
In particular,
 $\Pi_{(a,b)}(S(a,b,t))=t$. Observe that whenever $\widetilde{c}\ne 0$ we have
 \begin{equation}
 \label{w56}
 \frac{1}{\widetilde{c}}\text{proj}_{(\widetilde{a},\widetilde{b},\widetilde{c})}\equiv
 \Pi _{\left( -\frac{\widetilde{a}}{\widetilde{c}} ,-\frac{\widetilde{b}}{\widetilde{c}}\right)}.
 \end{equation}
 We remark that because of the symmetries of the Menger sponge without loss of generality we can restrict our attention to those projections
 $\text{proj}_{\widetilde{a},\widetilde{b},\widetilde{c}}$ for which $\widetilde{c}\ne 0$.
 These projections can be identified by 
 projections along planes 
in the form of \eqref{v98}.

In particular, the following implication holds:
\begin{equation}
\label{w55}
\text{int}(\Pi _{(a,b)}(\mathcal{M}_p))\ne \emptyset, \quad \forall (a,b)\quad
\Longrightarrow\quad
\text{int}(\text{proj}_{\widetilde{a},\widetilde{b},\widetilde{c}}(\mathcal{M}_p)),
\quad
\forall 
(\widetilde{a},\widetilde{b},\widetilde{c}),
\end{equation}
where $\text{int}$ stands for the interior.
Clearly,  the inequality in \eqref{w97}  is equivalent to
\begin{equation}\label{w98}
    \frac{\mathcal{L}eb_2(M_1\cap S(a,b,c))}{\mathcal{L}eb_2(Q\cap S(a,b,c))}\geq \frac{4}{9},
    \end{equation}
\begin{proposition}\label{w59}
If the inequality in \eqref{w98} holds for all $(a,b,c)$ then 
almost surely, for all $(a,b)$, the projection
$\Pi _{(a,b)}(\mathcal{M}_p)$ contains an interval for all $p>0.25$.
\end{proposition}

\begin{proof}
    Suggested by the formula on the top of page 177
of the paper \cite{SV_2014} we consider
 the function, 
\begin{equation}\label{v99}
    f(t):=\mathcal{L}eb_2(S(a,b,t)\cap Q).
\end{equation}
A combination of \cite[Theorem 1.2 and Proposition 2.3]{SV_2014} applied to the projection  $\Pi_{(a,b)}$ and function $f$
 yields the assertion of the Proposition.
 More precisely, if we replace 
 $proj_\alpha $ in
 \cite[Theorem 1.2]{SV_2014} by $\Pi _{(a,b)}$
 then we obtain that
 the assertion of 
 Proposition \ref{w59}  holds if a certain condition which is named as Condition $A(a,b)$ 
 (see \cite[Definition 2.1]{SV_2014})
 holds for all $(a,b)$. 
 On the other hand, 
 \cite[Proposition 2.3]{SV_2014} asserts
 in our case
 that another condition which is called 
 Condition $B(a,b)$  (see \cite[Definition 2.2]{SV_2014})  implies that Condition $A(a,b)$ holds. It is immediate from the definitions that 
 \eqref{w98} implies that Condition $B(a,b)$ holds
 with $f$ defined in \eqref{v99} and $\Pi _{(a,b)}$.
\end{proof}

\begin{proof}[Proof of Theorem \ref{y53} assuming Proposition \ref{w99}]
   The assertion of Theorem  \ref{y53}
   immediately follows from Proposition \ref{w59}
   and the implication in \eqref{w55}.
\end{proof}

\subsection{Proof of Proposition \ref{w99}}\label{w48}
We often use the following simple observation.
\begin{fact}\label{w66}
Observe that $Q\setminus M_1$ remains unchanged if we permutate the coordinate axes and if we reflect to any of the planes $x=\frac{1}{2}$, $y=\frac{1}{2}$ or $z=\frac{1}{2}$.
\end{fact}
Hence, it is easy to see that without loss of generality we may assume, that 
\begin{equation}\label{w67}
   0 \leq a \leq b \leq 1. 
\end{equation}
Under this assumption, the plane $S(a,b,c)$ intersects the unit cube $[0,1]^3$
if and only if 
\begin{equation}
\label{w54}
-(a+b)\leq c\leq 1.
\end{equation}
\begin{figure}[h!]
    \centering
    \includegraphics[width=0.5\textwidth]{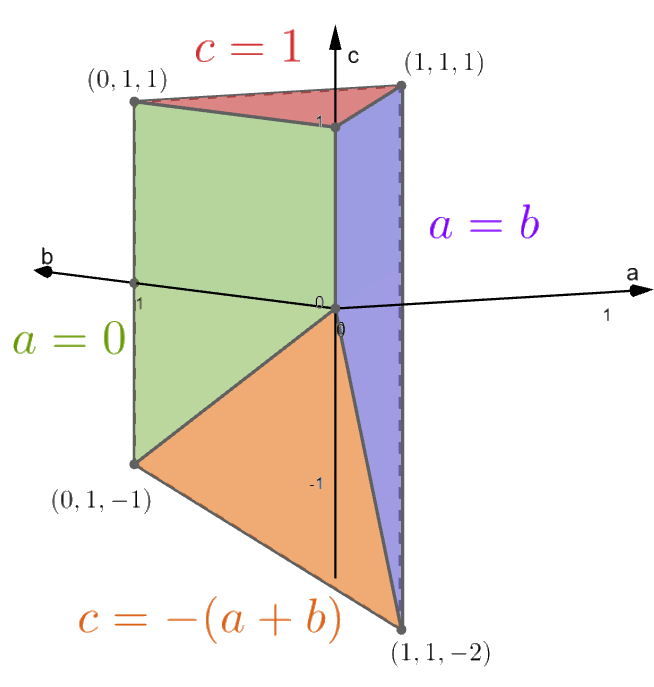}
    \caption{Illustration of the set $\mathfrak{D}$ \eqref{w53}.}
    \label{u36}
\end{figure}
So, we can confine ourselves to the region (see Figure \ref{u36})
\begin{equation}
\label{w53}
\mathfrak{D}:=
\left\{ (a,b,c):\quad
0\leq a\leq b\leq 1,\quad 
-(a+b)\leq c\leq 1
\right\}.
\end{equation}
We partition $\mathfrak{D}=
\mathfrak{D}_{\text{CA}}
\cup
\mathfrak{D}_{\text{WM}}$, where
\begin{multline}
\label{w52}
\mathfrak{D}_{\text{WM}}:=
\left\{ (a,b,c)\in \mathfrak{D}:
a\in \left[\frac{1}{3},1\right], b \in [a,1],\right.
\\
\left.
c\in\left[\frac{2}{3}-(a+b), 0]\cup [\max\{0, 1-(a+b)\}, \frac{1}{3}\right]\right\},
\end{multline}
and
\begin{equation}
\label{w51}
\mathfrak{D}_{\text{CA}}:=\mathfrak{D}
\setminus 
\mathfrak{D}_{\text{WM}}.
\end{equation}
According to this partition, 
the proof is divided into two major parts. We verify \eqref{w97}
\begin{itemize}
\item for $(a,b,c)\in \mathfrak{D}_{\text{CA}}$
with case analysis in Section \ref{w69}.
\item For $(a,b,c)\in \mathfrak{D}_{\text{WM}}$
we do the following:
We introduce a function $\widetilde{H}(a,b,c)$
on $\mathfrak{D}$ in 
\eqref{w49}, and we reformulate the inequality in \eqref{w97} as  $\widetilde{H}|_{\mathfrak{D}_{\text{WM}}}\geq 0$. To verify this, in Section 
\ref{w60} we estimate the Lipschitz constant of 
$\widetilde{H}|_{\mathfrak{D}_{\text{WM}}}$.
In Section   \ref{w50} we consider a sufficiently dense grid  $G\subset  \mathfrak{D}_{\text{WM}}$.
Then  using Wolfram Mathematica we calculate
$C:=\min\widetilde{H}|_{G}$. We verify that   $C>0$ is so large that 
because of the Lipschitz property of  $\widetilde{H}$,
we get that 
$\widetilde{H}|_{\mathfrak{D}_{\text{WM}}}>0$.
\end{itemize}

First we introduce the notation used in the rest of Section \ref{w48}.

\subsubsection{Notations}
We write $\text{proj}_{x,y}$ for the  orthogonal projection to the $(x,y) $-coordinate plane, that is
\begin{equation}\label{w96}
    \text{proj}_{x,y}:\R^3\to\R^2,\quad \text{proj}_{x,y}((u,v,w)):= (u,v).
\end{equation}
The area of the unit cube intersected with the plane $S(a,b,c)$ and the area of the $\text{proj}_{x,y}$-projection respectively denoted by:
\begin{equation}\label{w95}
    F(a,b,c):=\mathcal{L}eb_2(Q\cap S(a,b,c)), \quad \widetilde{F}(a,b,c):=\mathcal{L}eb_2(\text{proj}_{x,y}(Q\cap S(a,b,c))).
\end{equation}
It is easy to see that 
\begin{equation}\label{w90}
\widetilde{F}(a,b,c)\sqrt{1+a^2+b^2}=F(a,b,c).
\end{equation}
Let 
\begin{multline}\label{w94}
    \mathcal{A}:=\left\{\left(0,\frac{1}{3}, \frac{1}{3}\right),\left(\frac{1}{3},0,\frac{1}{3}\right),\left(\frac{1}{3},\frac{1}{3},0\right),\left(\frac{1}{3},\frac{1}{3},\frac{1}{3}\right),\left(\frac{2}{3},\frac{1}{3},\frac{1}{3}\right), \right. \\ \left. \left(\frac{1}{3},\frac{2}{3},\frac{1}{3}\right), 
    \left(\frac{1}{3},\frac{1}{3},\frac{2}{3}\right)\right\}
\end{multline}
denote the lower left corners
(the closest point of the cube to the origin)
of the level-1 cubes missing from the first approximation of the Menger sponge.
Also when we want to calculate the area of a level 1 square with lower left corner $(u,v,w)$ intersected with the plane $S(a,b,c)$ we renormalize the plane with respect to the cube and calculate the area of the renormalized intersection and divide it by 9 (see \eqref{w92}). Renormalization means that we transform a level 1 cube into the unit cube. In case of a level one cube $(u,v,w)+[0,\frac{1}{3}]^3$ this transformation is:
\begin{equation}
    (x,y,z)\to 3(x-u, y-v, z-w).
\end{equation}

The renormalization of the plane $S(a,b,c)$ with respect to the level $1$ cube with lower left corner $(u,v,w)$
is $S(a,b,c')$, where
$c':=3(au+bv+c-w)$. This is why we introduce
\begin{equation}\label{w93}
    g(a,b,c,u,v,w):=(a,b,3(au+bv+c-w)).
\end{equation}
Observe that for a level 1 cube $\widehat{Q}:=(u,v,w)+\left[0,\frac{1}{3}\right]^3$:
\begin{equation}\label{w92}
    \mathcal{L}eb_2(\widehat{Q}\cap S(a,b,c))= \frac{1}{9}F(g(a,b,c,u,v,w)).
\end{equation}
Hence, to verify Proposition \ref{w99}  we need to show that for any $(a,b,c)$:
\begin{multline}\label{w91}
H(a,b,c):=\frac{5}{9}\mathcal{L}eb_2(Q\cap S(a,b,c))-\mathcal{L}eb_2(Q \setminus M_1\cap S(a,b,c))\\=
\frac{5}{9}F(a,b,c)-\frac{1}{9}\sum_{(u,v,w) \in\mathcal{A}}F(g(a,b,c,u,v,w))\geq 0.
\end{multline}
Let 
\begin{equation}\label{w49}
   \widetilde{H}(a,b,c):=\frac{5}{9}\widetilde{F}(a,b,c)-\frac{1}{9}\sum_{(u,v,w)\in\mathcal{A}}\widetilde{F}(g(a,b,c,u,v,w)).
\end{equation}
By \eqref{w90} $H(a,b,c)=\frac{1}{\sqrt{1+a^2+b^2}}\widetilde{H}(a,b,c)$.
In this way we have proved that
\begin{fact}\label{w47}
    The inequality in \eqref{w97} holds if 
\begin{equation}\label{w68}
    \widetilde{H}(a,b,c)\geq 0, \, \text{ for any }0\leq a\leq b \leq 1\text{ and any }c.
\end{equation}
\end{fact}
Let also 
\begin{equation}\label{w58}
    \ell_i(x):=-\frac{a}{b}x+\frac{\frac{i}{3}-c}{b}, \quad i \in \left\{0,1,2,3\right\}
\end{equation}
denote the line which is the $\text{proj}_{x,y}$-projection of $$\left\{\left(x,y,\frac{i}{3}\right):x\in \R, y\in \R\right\}\cap S(a,b,c).$$
That is $y=\ell_i(x)$ is the $\frac{i}{3}$-level set of $S(a,b,c)$. Note that \eqref{w67} implies that the slope of $\ell_i(x)$ is between $-1$ and $0$.

For any $H\subseteq [0,1]^2$ the \texttt{ bad part of $H$}
is $H\cap \text{proj}_{x,y}(S(a,b,c)\cap Q\setminus M_1)$
 and the \texttt{good part of $H$} is   $H\cap \text{proj}_{x,y}(S(a,b,c)\cap M_1)$. The bad part of $[0,1]^2$ is going to be denoted by dark color on the figures below.
Let $Q_{i,j}:=\left(\frac{i}{3}, \frac{j}{3}\right)+\left[0,\frac{1}{3}\right]^2$, $i,j \in \left\{0,1,2\right\}$ as shown in the Figure \textsc{\ref{w86}}.
We call these squares \texttt{level-$1$ squares}.
By the construction of the Menger sponge, the squares in the corners can not have bad parts. The bad parts of $Q_{1,1}$ is between the lines $\ell_{0}(x)$ and $\ell_{3}(x)$, and the bad parts of $Q_{1,0} , Q_{2,1},Q_{0,1}, Q_{1,2}$ are between the lines $\ell_{1}(x)$ and $\ell_2(x)$.

First we show by a detailed case analysis (Section \ref{w69}) that for some values of $a,b$ and $c$ the assertion \eqref{w68} holds. Then for the uncovered values of $a,b,c$ we create a grid to calculate the value of $\widetilde{H}$ on the grid and estimate the value of the function out of the grid points, using the upper bound (Section \ref{w60}) on the Lipschitz constant for the function $\widetilde{H}$ .

\subsubsection{Case analysis}\label{w69}
Recall that we always assume that $0 \leq a \leq b \leq 1$. Note that if $a=b=0$ and $c\in (\frac{1}{3}, \frac{2}{3})$ then $\widetilde{H}(a,b,c)=0$. 

\begin{fact}\label{w89}
If $c>0$ and $a+b+c\leq 1$, then $\widetilde{H}(a,b,c)\geq0$.
\end{fact}
\begin{proof}
Whenever $c>0$ and $a+b+c\leq 1$ by elementary geometry we have $\widetilde{F}(a,b,c)=1$. Hence, $\widetilde{H}(a,b,c)$ is the smallest when all of $Q_{1,0} , Q_{2,1},Q_{0,1}, Q_{1,2}$ is between the lines $\ell_{1}(x)$ and $\ell_2(x)$. 
In that case $\widetilde{H}(a,b,c)=0$. For visual explanation see Figures \textsc{\ref{w85}} and \textsc{\ref{w84}}.
\end{proof}
\begin{figure}
\minipage{0.25\textwidth}
  \includegraphics[width=\linewidth]{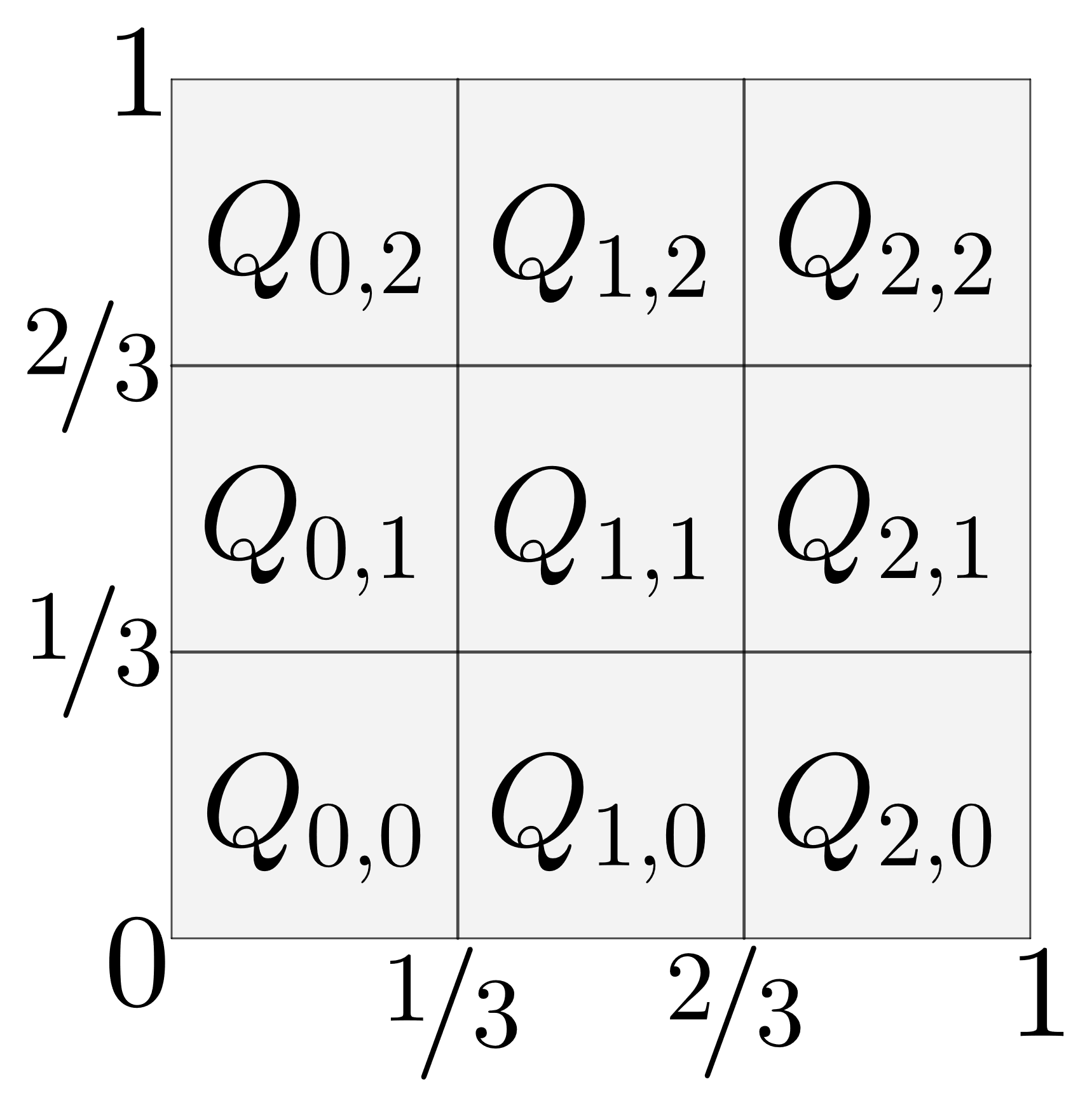}
  \subcaption{}\label{w86}
\endminipage\hfill
\minipage{0.25\textwidth}
  \includegraphics[width=\linewidth]{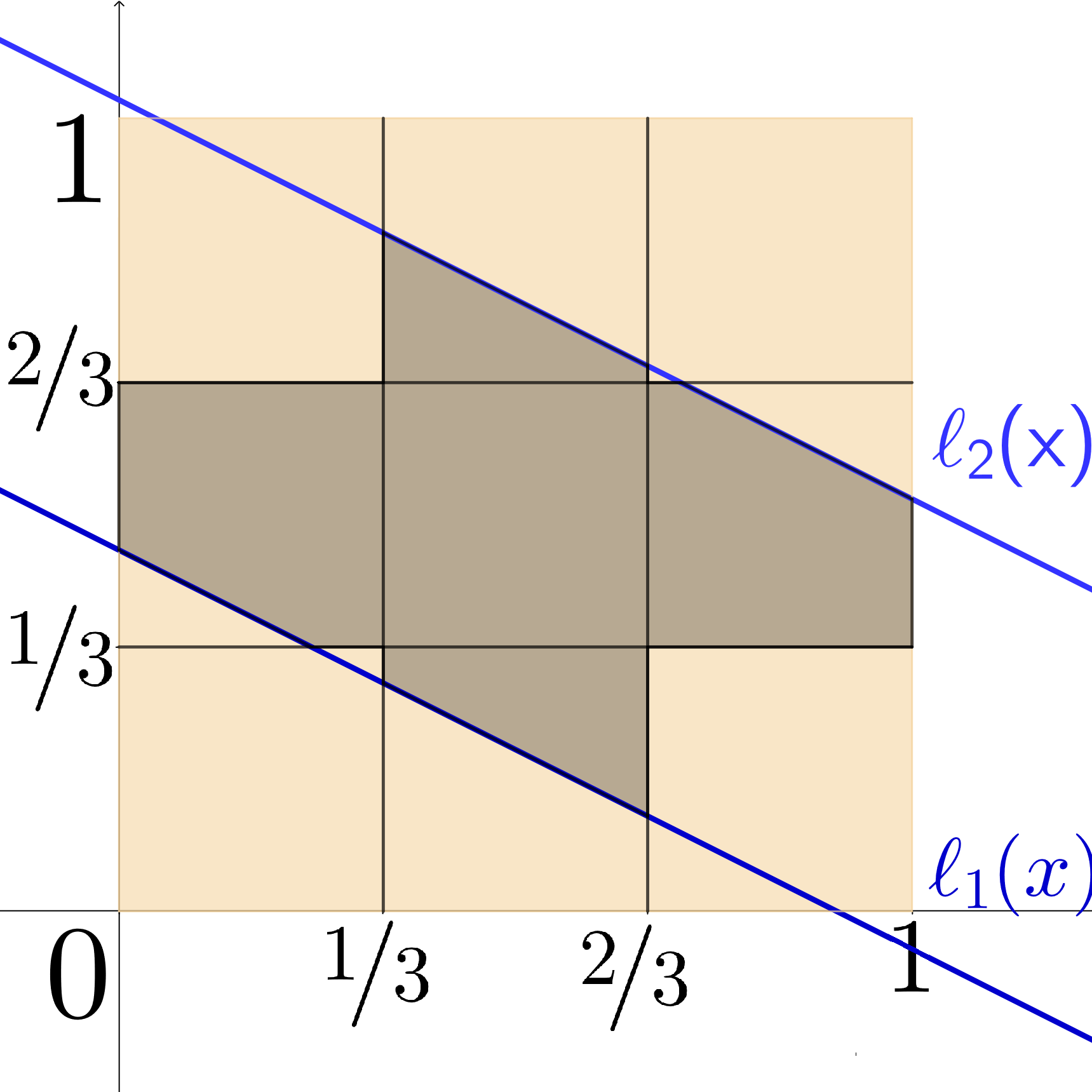}
  \subcaption{}\label{w85}
\endminipage\hfill
\minipage{0.25\textwidth}
  \includegraphics[width=\linewidth]{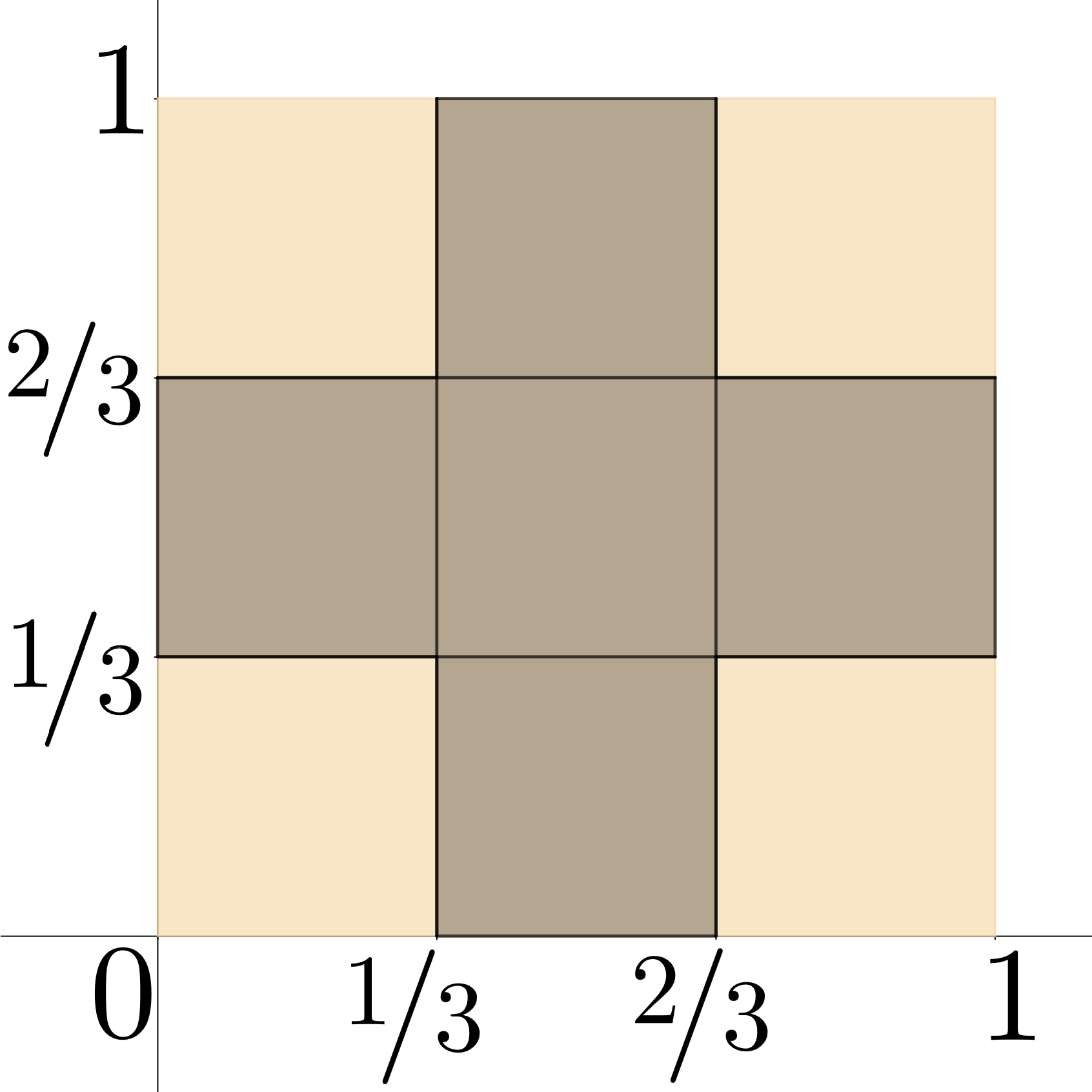}
  \subcaption{}\label{w84}
\endminipage\hfill
\minipage{0.25\textwidth}
  \includegraphics[width=\linewidth]{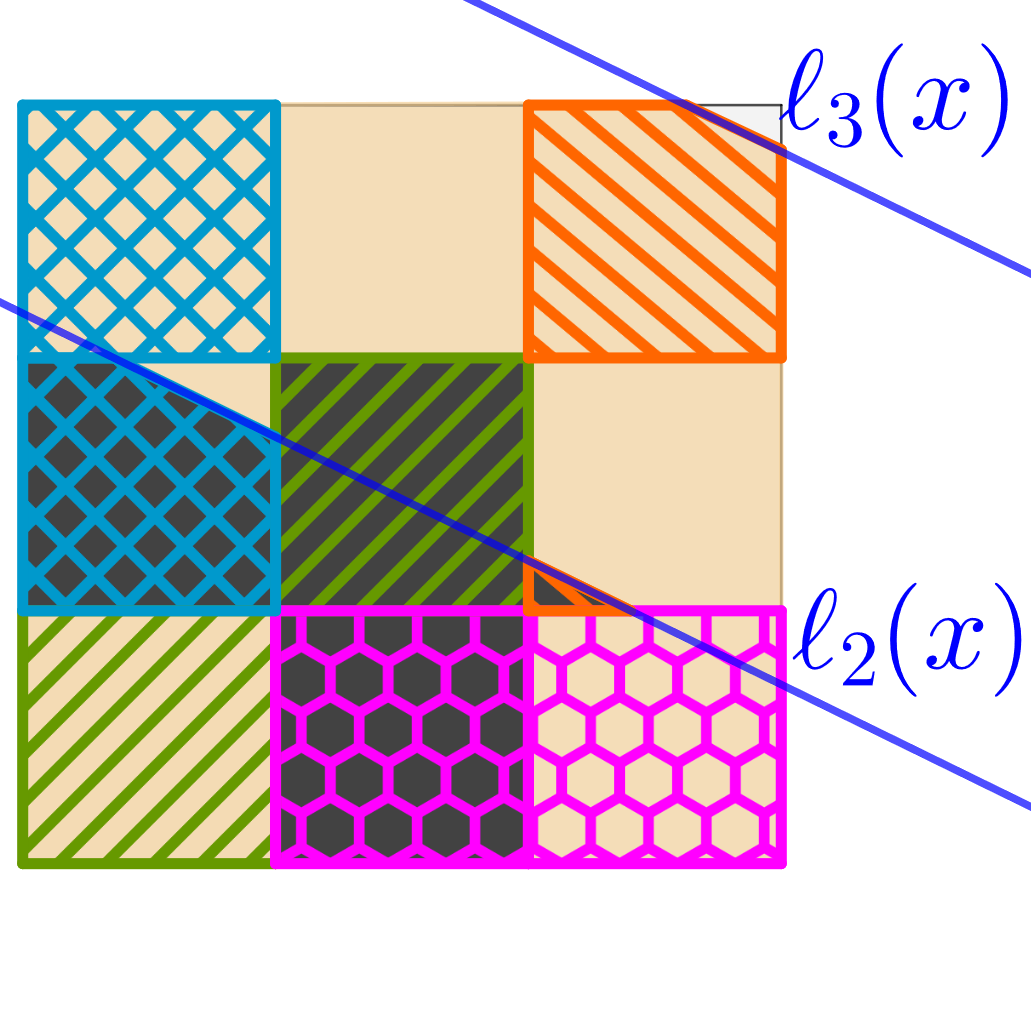}
   \subcaption{}\label{w83}
\endminipage\hfill
\caption{Figures for Fact \ref{w89} and \ref{w87}. In the last three figures the lighter color denotes the good and the darker the bad part of the unit square respectively.}
\end{figure}
\begin{fact}\label{w87}
If $\frac{1}{3}\leq c\leq 1$, then $\widetilde{H}(a,b,c)\geq0$.
\end{fact}
\begin{proof}
Whenever $a+b+c\leq 1$ we are done by Fact \ref{w89}. Hence, we can assume, that $a+b+c>1$. In this case $\ell_0([0,1]),\ell_1([0,1])\leq 0$. 
This follows from $c\geq\frac{1}{3}$ and 
\eqref{w58}.
For any such $S(a,b,c)$ plane if we fix the line $\ell_3(x)$ it is easy to see that the worst case scenario is when $c=\frac{1}{3}$, since for a fixed $\ell_3(x)$ the function $\widetilde{F}(a,b,c)$ remains the same but the area of the bad part grows as we decrease c. 
Hence, we may assume, that $c=\frac{1}{3}$ and we know that $\ell_3(1)<1$ (because  $a+b+c>1$) hence $a+b>\frac{2}{3}$.
So, it completes the proof of Fact \ref{w87}
if we verify that the following Claim holds: 
\begin{claim}\label{w57}
If  $a+b+c>1$ and $c=\frac{1}{3}$ then 
 for every level-$1$ square
$Q_{i,j}$ we can find another level-$1$ square 
$Q_{i',j'}$ such that the area of the good part of 
$Q_{i',j'}$ is greater than or equal to the 
area of the bad part of $Q_{i,j}$. Moreover, distinct $(i',j')$ correspond to distinct $(i,j)$.    
\end{claim}
Namely, this Claim implies that the area of the bad part of the unit square is smaller than the area of its good part and 
hence, 
 $\widetilde{H}(a,b,c)>0$.

The proof of Claim \ref{w57} will be given below, but first we remark that this proof if  illustrated 
 on Figure \textsc{\ref{w83}}. The light background regions on Figure \textsc{\ref{w83}} refer to the good parts and the dark background regions refer to the bad parts. Figure \textsc{\ref{w83}} indicates that the  bad region of every level-$1$ square  is compensated by an identically patterned good region of a corresponding level $1$ square having at least as large area as the corresponding bad region.
 
\begin{proof}[Proof of Claim \ref{w57}]\ 
    \begin{itemize}
    \item $Q_{1,2}$ does not have a bad part in this case, because $a+b>\frac{2}{3}$ and $b>\frac{1}{3}$ hence $a+2b>1$ thus $\ell_2(\frac{1}{3})=\frac{1-a}{3b}<\frac{2}{3}$. All of the lines $\ell_{i}(x)$ has negative slope, hence this shows that all of the points of $\ell_2(x)$ is under the square $Q_{1,2}$.
    \item The bad part of $Q_{1,1}$ is compensated by the good part of $Q_{0,0}$, since $\ell_3(0) \geq \ell_3\left(\frac{1}{3}\right)-\frac{1}{3}$ and hence the line segment $\ell_3([0,\frac{1}{3}])$ is higher relative to the line $y=0$ than the segment $\ell_3(\left[\frac{1}{3}, \frac{2}{3}\right])$ to the line $y=\frac{1}{3}$. This part is illustrated by green lines on Figure \textsc{\ref{w83}}.
    \item The area of the bad part of $Q_{1,0}$ is smaller than the area of the good part of $Q_{2,0}$ since $\ell_{3}(\frac{2}{3}) \geq \ell_{2}(\frac{1}{3})$ and $\ell_2(x)$ and $\ell_3(x)$ has the same slope. This part is illustrated by pink honeycombs on Figure \textsc{\ref{w83}}.
    \item The area of the bad part of $Q_{2,1}$ is at most as much as the area of the good part of $Q_{2,2}$. This follows from the fact that $\ell_3(\frac{2}{3})-\frac{2}{3} \geq \ell_2(\frac{1}{3})-\frac{1}{3}$. This part is illustrated by orange lines on Figure \textsc{\ref{w83}}.
    \item The area of the bad part of $Q_{0,1}$ is at most as much as the area of the good part of $Q_{0,2}$. This follows from the fact that $\ell_3(0)-\frac{2}{3} \geq \ell_2(0)-\frac{1}{3}$. This part is illustrated by blue lines (crossed) on Figure \textsc{\ref{w83}}.
\end{itemize}
\end{proof}
As we mentioned above, Claim \ref{w57} implies that the assertion of Fact \ref{w87} holds. 
\end{proof}

\begin{fact}\label{w65}
If $a+b+c\leq\frac{2}{3}$, then $\widetilde{H}(a,b,c)\geq 0$.
\end{fact}
\begin{proof}
This is immediate from Facts \ref{w66} and \ref{w87}.
\end{proof}
\begin{corollary}\label{w79}
If $a+b\leq\frac{2}{3}$, then $\widetilde{H}(a,b,c)\geq 0$.
\end{corollary}
\begin{proof}
Whenever 
\begin{description}
    \item [$0<c\leq\frac{1}{3}$] we are done by Fact \ref{w89};
    \item [$\frac{1}{3}\leq c\leq 1$], then 
    \begin{description}
        \item [$a+b<1-c$] we are done by Fact \ref{w89},
        \item[$a+b \geq 1-c$] we are done by Fact \ref{w87};
    \end{description}
    \item [$c\leq 0$] we are done by Fact \ref{w65}.
\end{description}

\end{proof}
\begin{fact}\label{w72}
If $a<\frac{1}{3}$, then $\widetilde{H}(a,b,c)\geq 0$.
\end{fact}
\begin{proof}
Let $h_i:=\ell_i(0)$ and $\widehat{h}_i:=\ell_i(1)$. Since $\text{dist}(h_i, h_{i+1})=\frac{1}{3b}$ and $b\leq 1$ it follows, that $\text{dist}(h_i, h_{i+1})\geq \frac{1}{3}$, hence the following cases are possible:
\begin{enumerate}[label*=\Roman*.]
    \item  $h_0\in \left[0,\frac{1}{3}\right]$ 
    \begin{enumerate}[label*=\alph*.]
        \item $h_1 \in \left[\frac{1}{3}, \frac{2}{3}\right)$
        \begin{enumerate}[label*=\arabic*.]
            \item  \label{w77} $h_2 \in \left[\frac{2}{3}, 1\right)$
            \item \label{u32} $h_2 \geq 1$
        \end{enumerate}
        \item \label{u31} $h_1 \in \left[\frac{2}{3},1\right)$ and $h_2 \geq 1$
        \item \label{u30} $h_1 \geq1$ and $h_2 > 1$
    \end{enumerate}
    \item \label{u29} $h_0 \in \left[\frac{1}{3}, \frac{2}{3}\right)$
    \begin{enumerate}[label*=\alph*.]
        \item \label{u28} $h_1\in \left[\frac{2}{3}, 1\right)$ and $h_2\geq 1$
        \item \label{w76} $h_1 \geq 1$ and $h_2>1$
    \end{enumerate}
    \item \label{u27} $h_0 \in \left[\frac{2}{3}, 1\right)$ and $h_1 \geq 1$ and $h_2>1$
    \item \label{u26} $h_0\geq 1$
    \item \label{w78} $h_0 < 0$.
\end{enumerate}
The proof of 
Case \ref{w78} follows from Fact \ref{w66}, 
Fact \ref{w89}  and from Cases I-IV. We mainly use two methods for the proof.

One is to estimate the area in the case when $a=0$ i.e. the slope of the lines $\ell_i$ equals zero, then we upper bound the growth of the area of the bad part and lower bound the shrinkage of the area of the good part to get a "worst case scenario" fraction of the bad and the total area. We use this method in case of  \ref{w77}-\ref{u31}.

The other way is to prove that the good part of the unit square is at least as big as the bad part. We do this one-by-one, for each $Q_{i,j}$ that has bad part. We apply this method in case of \ref{u30}-\ref{u26}.

\textbf{Case \ref{w77}}: In this case $h_0\in \left[0,\frac{1}{3}\right], h_1 \in \left[\frac{1}{3}, \frac{2}{3}\right), h_2 \in \left[\frac{2}{3}, 1\right)$ and $h_3\geq 1$. 
We distinguish two cases: when $a=0$ or $a>0$.
\begin{description}
    \item[$a=0$]    Then
  $\ell_i(x)=h_i$ and $\widetilde{F}(a,b,c)=1-h_0$. $Q_{1,1}$ has only bad part since $h_0\leq \frac{1}{3}$ and $h_1\geq 1$, $Q_{1,0}$ has no bad parts because $h_1\geq \frac{1}{3}$. The area of the bad part of $Q_{0,1}$ and $Q_{2,1}$ together is $\frac{2}{3}\left(\frac{2}{3}-h_1\right)$ and the area of the bad part of $Q_{1,2}$ is $\frac{1}{3}\left(h_2-\frac{2}{3}\right)$. So, the bad area in $[0,1]^2$ is $1/3$ of the total area which is $\widetilde{F}(a,b,c)=1-h_0$.
    \item[$a>0$] In this case we distinguish  two further sub-cases. Namely, the first is when $\ell_3(1)\geq1$ (see Figure \textsc{\ref{w75}}) and the second one is when $\ell_3(1)<1$ (see Figure \textsc{\ref{w74}}). In both sub-cases the area of the bad part increases at most with the area of the triangle $(0,h_1), (1,h_1), (1,\widehat{h_1})$ intersected with $Q_{0,1}$ and $Q_{1,0}$ and $Q_{2,1}$ (see the light hatched area  on Figure \textsc{\ref{w74}}). This area is smaller than $\frac{a}{3b}$. 
    \begin{description}
        \item[$\ell _3(1)\geq 1$] Then the total area 
        (the area of the good and the bad part together)
        does not decrease. Rewriting $h_2=h_0+\frac{2}{3b}$ and $h_1=h_0+\frac{1}{3b}$ gives that $\widetilde{H}(a,b,c)\geq \frac{2}{9}(1-h_0)-\frac{a}{3b}$. It follows from  $h_2=\frac{\frac{2}{3}-c}{b}<1$, that $b+c>\frac{2}{3}$ and from $\frac{2}{3b}=h_2-h_0 \leq 1$, that $b>\frac{2}{3}$. Multiplying $\frac{2}{9}(1-h_0)-\frac{a}{3b}$ by $9b>0$ gives $2b+c-3a>\frac{1}{3}>0$. Thus,  $\widetilde{H}(a,b,c)>0$.
        \item[$\ell _3(1)< 1$] Then the total area  decreases. An easy calculation shows that the total area is $1-\frac{c^2}{2ab}-\frac{(a+b+c-1)^2}{2ab}$. Observe that $\ell_3(1)=-\frac{a}{b}+\frac{1-c}{b}<1$ implies that $a+b+c>1$. Using this and that $b>\frac{2}{3}$ a substantial but elementary calculation shows that in this case $\widetilde{H}(a,b,c)>0$.
      \end{description}
  \end{description}

\textbf{Case \ref{u32}}: In this case $\widehat{h}_3>1$ since $\widehat{h}_{3}>h_2>1$. The proof is very similar to the one of \ref{w77}.
\begin{description}
    \item[$a=0$] In this case the total area $\widetilde{F}(a,b,c)$ is $1-h_0$. From $h_1\geq \frac{1}{3}$ it follows that $Q_{1,0}$ does not have a bad part. Hence the bad parts and their areas are as follows:
    \begin{itemize}
        \item $Q_{1,1}$ and $Q_{1,2}$ consists only of bad parts, hence their bad area together is $\frac{2}{9}$,
        \item $Q_{0,1}$ and $Q_{2,1}$ have  bad parts over $h_1$, the are of their bad parts together is $\left(\frac{2}{3}-h_1\right)\cdot \frac{2}{3}$
    \end{itemize}
    Altogether the bad part has area $$ \frac{2}{3}\left(1-h_0\right)-\frac{2}{9b}.$$
    $2 b(1-h_0)<2$, hence $6b(1-h_0)<2+4b(1-h_0)$ thus $\frac{2}{3}-\frac{2}{9b(1-h_0)}<\frac{4}{9}$, hence the fraction of the bad area and the total is smaller than $\frac{4}{9}<\frac{5}{9}$.
    \item[$a>0$] When $a>0$, it is clear that the area bad area grows with at most $\frac{a}{3b}$ similarly to the previous case. The total area does not decrease, hence we can use the previous calculation to show that in this case the assertion holds.
\end{description}

\textbf{Case \ref{u31}}:
\begin{description}
\item[$a=0$] The total area is again $(1-h_0)$. It follows from $h_1>\frac{2}{3}$, that only $Q_{1,2}$ and $Q_{1,1}$ can have bad parts. It is easy to see, that $Q_{1,1}$ has only bad parts, and that the bad part of $Q_{1,2}$ has area $(1-h_1)\frac{1}{3}=\frac{1}{3}\left(1-h_0-\frac{1}{3b}\right)$. An easy calculation shows that in this case the fraction of the bad and the total area is smaller than $\frac{4}{9}$, hence we are done.
\item[$a>0$]Similarly to the earlier cases we consider the growth of the area if we increase $a$. The additional bad part is the intersection of $Q_{0,1}\cup Q_{2,1}\cup Q_{1,2}$ with the triangle which we get by intersecting the unit square with the area between $\ell_{1}$ and the line with height $h_1$. The area of this triangle is smaller than $\frac{a}{3b}$. The assertion follows if we rewrite the inequality $\frac{1}{9(1-h_0)}+\frac{1}{3}+\frac{3a-1}{9b(1-h_0)}<\frac{5}{9}$, getting 
$1+\frac{3a-1}{b}<2(1-h_0)$, which trivially holds, since $3a-1<0$, hence the left hand side is smaller than 1, and the right hand side is greater than 1 by $1-h_0>\frac{2}{3}$
\end{description}

\textbf{Case \ref{u30}}:
This case is almost identical to \ref{w76}, only in that case the total area might be smaller because $h_0$ is greater. For this reason instead of proving this statement we present the proof for Case \ref{w76} later because that might seem more complicated.

\textbf{Case \ref{u28}}
Observe that from $\ell_{1}(1)\geq h_0\geq \frac{1}{3}$ it follows that $Q_{1,0}$ does not have bad part.
\begin{description}
    \item[$Q_{1,2}$] The bad parts of it are compensated by the good parts of $Q_{0,2}$. This is because $Q_{0,2}$ only has good parts, because $\ell_{3}(x)\cap [0,1]>1$ and $\ell_{0}(0)\leq \frac{2}{3}$.
    \item[$Q_{1,1}$] $Q_{2,2}$ only has good parts, this can be verified in the same way as in the case of $Q_{0,2}$. Thus the bad part of $Q_{1,1}$ is compensated by $Q_{2,2}$.
    \item[$Q_{0,1}$]Since $b\leq 1$, $\ell_1(x)-\ell_0(x)=\frac{1}{3b}\geq  \frac{1}{3}$. Consequently the area of $Q_{0,0}$ above the line $\ell_{0}(x)$ (i.e. the good part of $Q_{0,0}$) is greater than the part of $Q_{0,1}$ above $\ell_{1}(x)$ (i.e. the bad part of $Q_{0,1}$). It follows, that the area of the bad part of $Q_{0,1}$ is smaller than the area of the good part of $Q_{0,0}$.  
    \item[$Q_{2,1}$] By almost exactly the same reasoning as in the previous case the area of the bad part of $Q_{2,1}$ is smaller than the good part of $Q_{2,1}$.
\end{description}
\textbf{Case \ref{w76}}: In this case $h_0\in \left[\frac{1}{3}, \frac{2}{3}\right), h_1\geq 1, h_2, h_3 >1$. Note that $a<\frac{1}{3}$ yields $\ell_3([0,1])>h_1\geq 1$ and since $h_0<\frac{2}{3}$ the squares $Q_{0,1}$ and $Q_{2,2}$ has only good parts. It follows from $a<\frac{1}{3}$ that $\ell_1([0,1])>h_0\geq \frac{1}{3}$ and $h_1\geq 1$. Potentially only $Q_{1,1}, Q_{1,2}$ and $Q_{2,1}$ have bad parts. For the illustration of the following pairing see Figure \textsc{\ref{w73}}, where the bad parts are denoted by the darker, the good parts are denoted by the lighter color, the pairs has the same pattern and color.
\begin{description}
    \item[$Q_{1,1}$] Since the area of the bad part of $Q_{1,1}$ is at most $\frac{1}{9}$, and the area of $Q_{2,2}$ is $\frac{1}{9}$ the area of the bad part of $Q_{1,1}$ is compensated by the area of the good part of $Q_{2,2}$.
    \item[$Q_{1,2}$] The same holds for the bad area of $Q_{1,2}$ and the good area of $Q_{0,2}$.
    \item [$Q_{2,1}$] $Q_{2,1}$ has bad part only in the case, when $\widehat{h_1}<\frac{2}{3}$. The good part of $Q_{2,0}$ which is above $\ell_0(x)$ and the bad part of $Q_{2,1}$ which is above $\ell_1(x)$ and since $\ell_1(x)$ and $\ell_0(x)$ has the same slope, it is enough to show that $\frac{1}{3}-\widehat{h}_0>\frac{2}{3}-\widehat{h}_1$. Since $\widehat{h}_1= \widehat{h}_0+\frac{1}{3b}$ the inequality reduces to $1>b$, which is true by our assumptions, hence the area of the good part of $Q_{2,0}$ is greater than the area of the bad part of $Q_{2,1}$.
\end{description}

\textbf{Case \ref{u27}}
It is clear that only $Q_{1,1}$
 and $Q_{1,2}$ can have bad parts, since $\ell_{1}\cap[0,1]\geq h_0\geq \frac{2}{3}$.
 \begin{description}
     \item[$Q_{1,1}$] Whenever $Q_{1,1}$ has bad parts, $\ell_{0}(\frac{2}{3})\leq\frac{2}{3}$, then the whole $Q_{2,2}$ is good. Thus it compensates for the bad parts of $Q_{1,1}$.
     \item[$Q_{1,2}$] Since $\frac{1-a}{3b}>0$ it is easy to see that the area of the good part of $Q_{0,2}$ is greater than the area of the bad part of $Q_{1,2}$.
 \end{description}
\textbf{Case \ref{u26}}
In this case only $Q_{1,1}$ can have bad part, and if it happens then the whole $Q_{2,2}$ is good, hence the bad area of $Q_{1,1}$ is compensated.

\begin{figure}[h!]
\minipage{0.33\textwidth}
  \includegraphics[width=\linewidth]{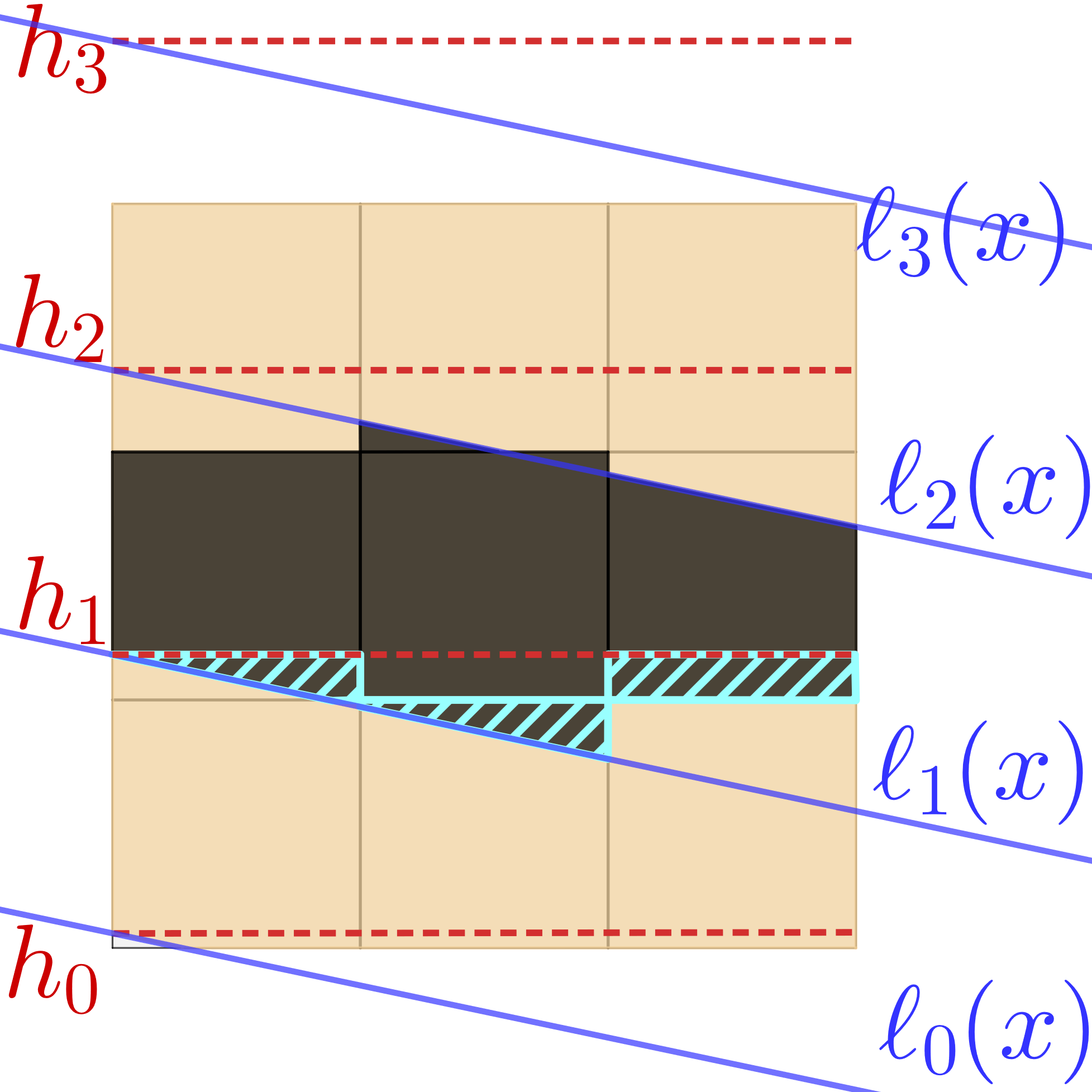}
  \subcaption{}\label{w75}
\endminipage\hfill
\minipage{0.33\textwidth}
  \includegraphics[width=\linewidth]{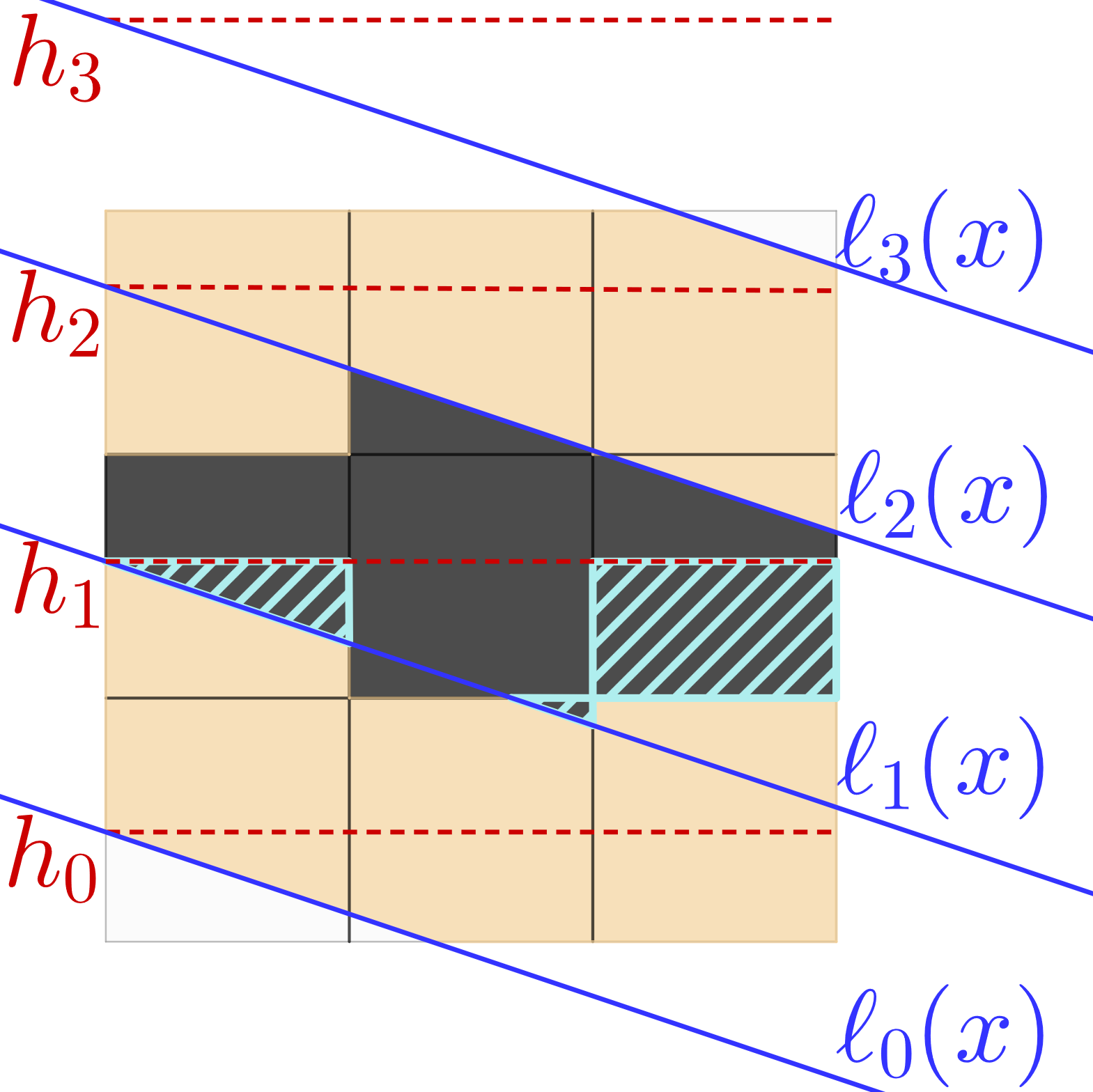}
  \subcaption{}\label{w74}
\endminipage\hfill
\minipage{0.33\textwidth}
  \includegraphics[width=\linewidth]{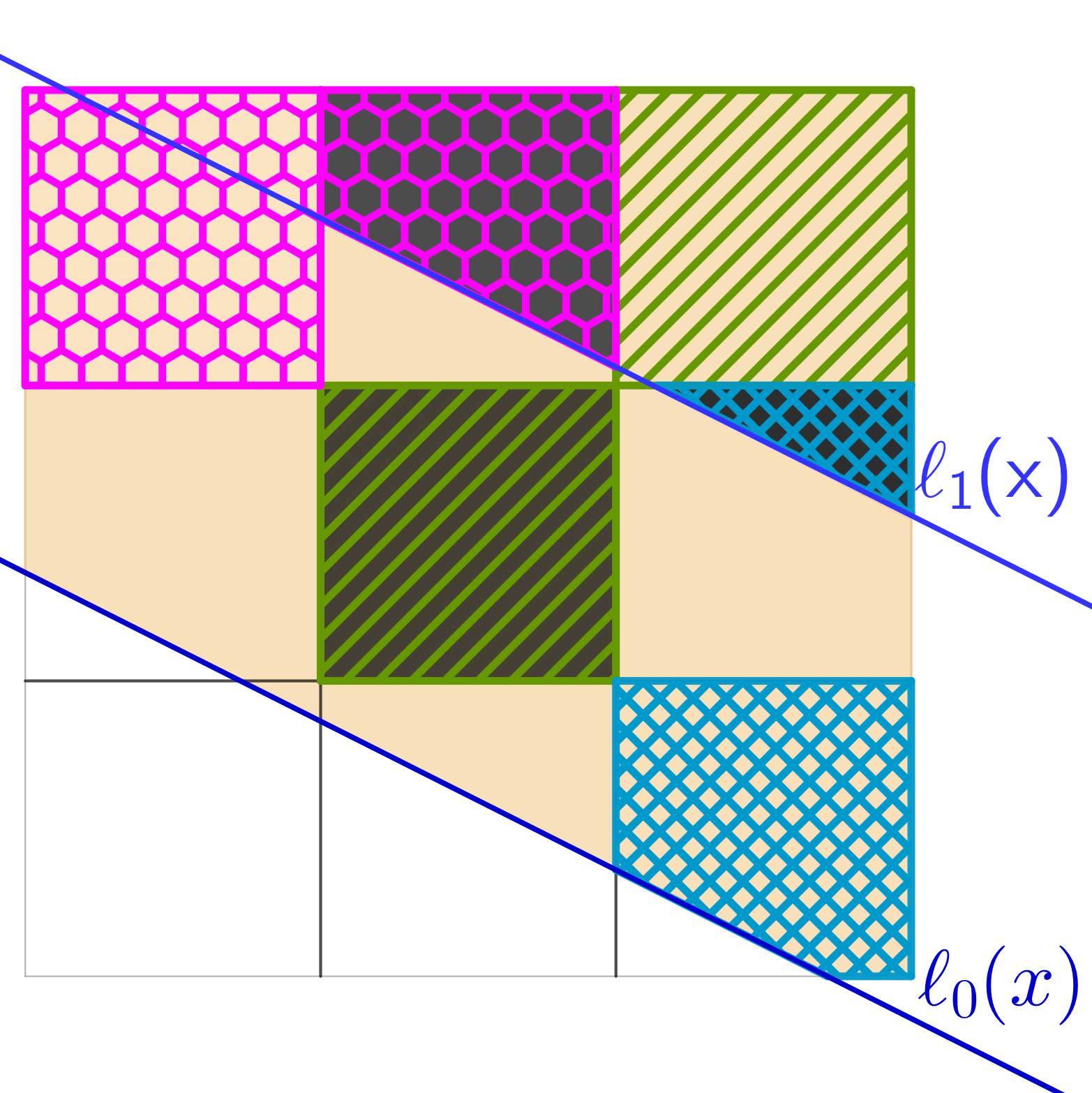}
  \subcaption{}\label{w73}
\endminipage\hfill
\caption{Figures for Fact \ref{w72}. In the first two figures the red lines denotes the height of $h_i$, the blue lines denotes $\ell_i(x)$.}
\end{figure}
\end{proof}
\subsubsection{Lipschitz-constants for the function $\widetilde{F}(a,b,c)$}\label{w60}
Putting together Facts \ref{w72} and \ref{w47}
from now we may always assume that 
\begin{equation}
\label{w46}
\frac{1}{3}\leq a\leq b\leq 1.
\end{equation}
 Since the gradient of the renormalized plane remains the same these assumptions holds when instead of $\widetilde{F}$ we consider $\widetilde{F}\circ g$ (g was defined in \eqref{w93}).
For a fix $(a,b)$
 satisfying \eqref{w46} consider the plane $S(a,b,c)$. This plane can intersect the unit cube $[0,1]^3$
 in eight different ways depending on $c$. The different ways of intersections yields different formulas for $\widetilde{F}(a,b,c)$.

 We estimate the Lipschitz constant from above  by the sup-norm of the gradient of $\widetilde{F}$, which can be calculated using elementary calculus. Table \ref{w71} contains the results of the calculation, and some figures to visualize  the different cases.
\begin{table}[]
\begin{tabular}{llllll}
\hline
  Name & Figure1 												  & Figure2 													 & Conditions                                                                             & $\widetilde{F}(a,b,c)$          & Lipsch.                    \\ \hline
{\footnotesize 0} &		    												  & \includegraphics[height=1.1cm]{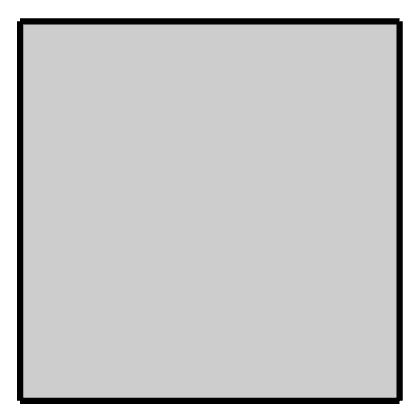} & \begin{tabular}[c]{@{}l@{}}$c\geq1$ \\ or\\ $a+b+c\leq 0$\end{tabular}                 & 0                               & 0                     \\ \hline
{\footnotesize 1} &	\includegraphics[height=1.2cm]{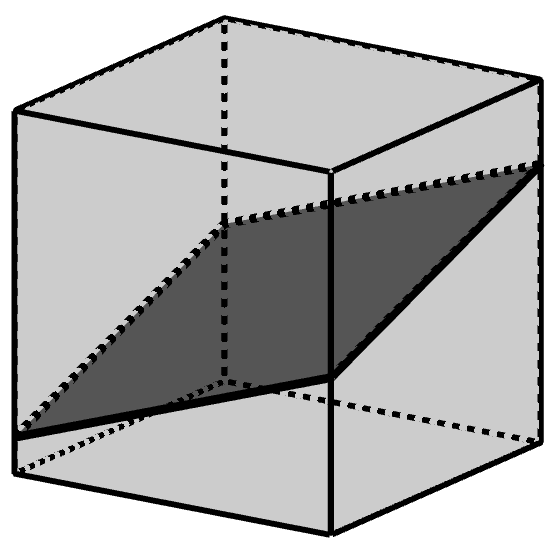} & \includegraphics[height=1.1cm]{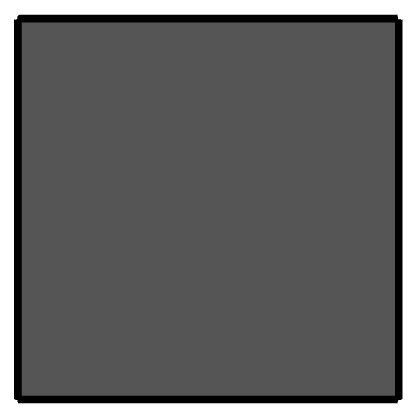} & \begin{tabular}[c]{@{}l@{}}$c \geq 0$ \\ and\\ $a+b+c \leq 1$\end{tabular}             & 1                               & 0                     \\ \hline
{\footnotesize 2} &	\includegraphics[height=1.2cm]{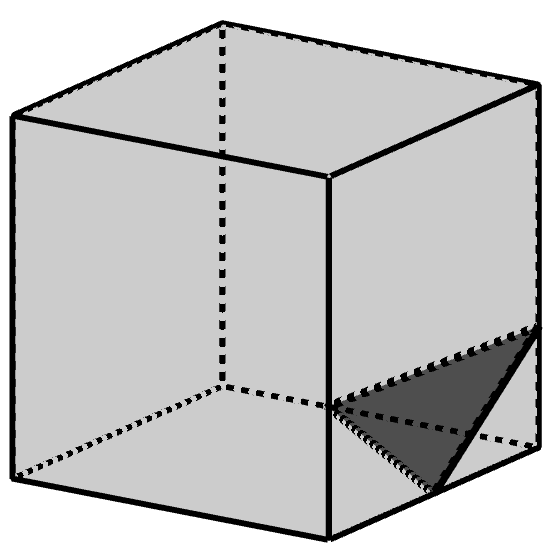} & \includegraphics[height=1.1cm]{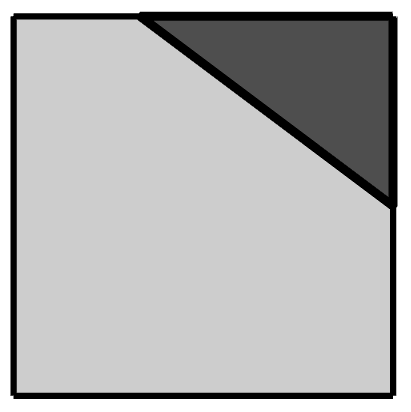} & $-(a+b) \leq c \leq b$                                                                 & $\frac{(a+b+c)^2}{2 a b}$       & $\frac{3\sqrt{3}}{2}$ \\ \hline
{\footnotesize 3} &	\includegraphics[height=1.2cm]{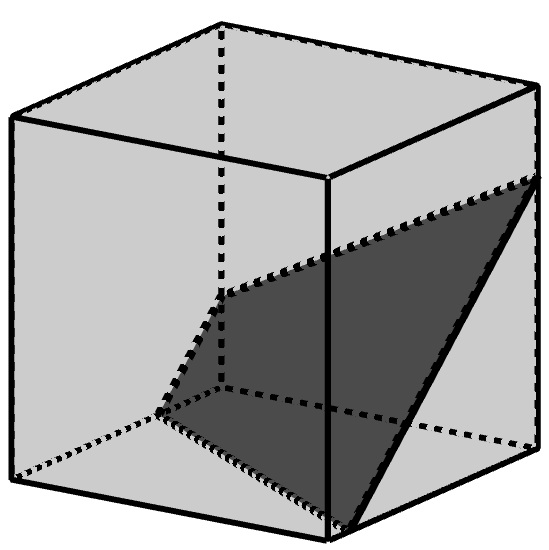} & \includegraphics[height=1.1cm]{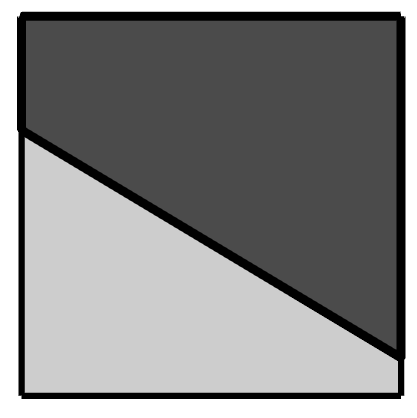} & $-b\leq c\leq -a$                                                                      & $\frac{a+2b+2c}{2b}$            & $\frac{3\sqrt{3}}{2}$ \\ \hline
{\footnotesize 4} &	\includegraphics[height=1.2cm]{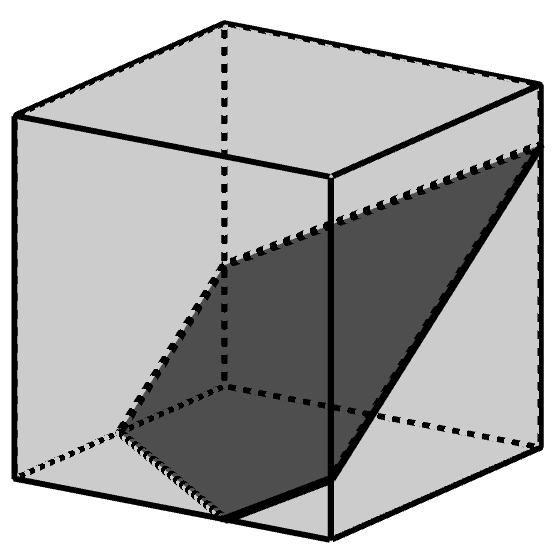} & \includegraphics[height=1.1cm]{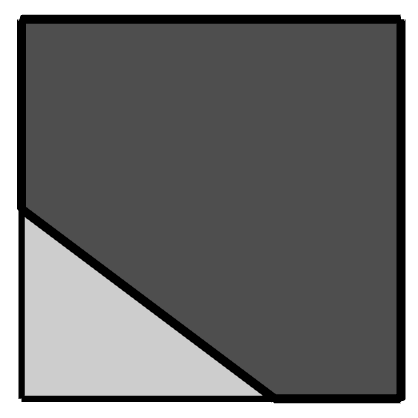} & \begin{tabular}[c]{@{}l@{}}$-a\leq c$ \\ and\\ $c\leq \min\left\{0,1-(a+b)\right\}$\end{tabular}  & $1-\frac{c^2}{2ab}$             & $\frac{3\sqrt{3}}{2}$ \\ \hline
{\footnotesize 5} &	\includegraphics[height=1.2cm]{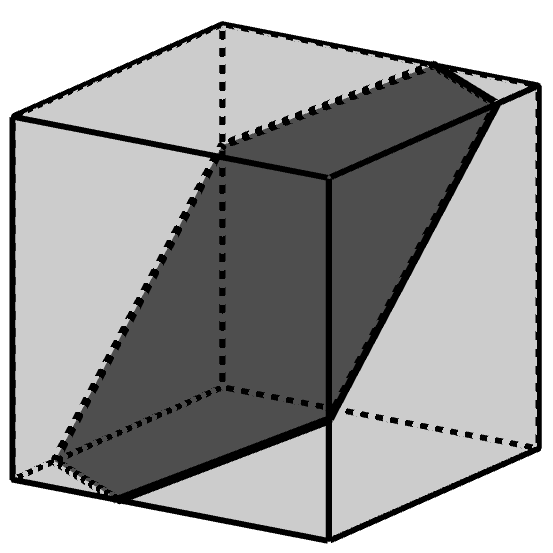} & \includegraphics[height=1.1cm]{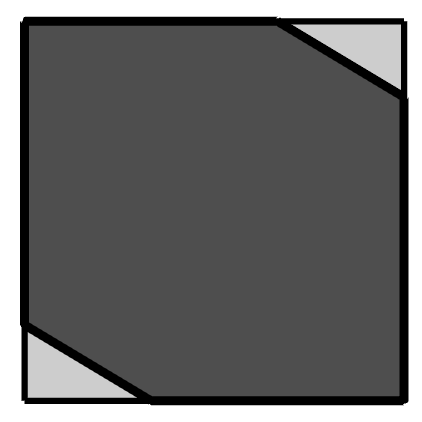} & \begin{tabular}[c]{@{}l@{}}$1\leq a+b$\\ and\\ $1-(a+b)\leq c\leq 0$\end{tabular}      & $1-\frac{c^2+(a+b+c-1)^2}{2ab}$ & $\frac{131}{54}$      \\ \hline
{\footnotesize 6} &	\includegraphics[height=1.2cm]{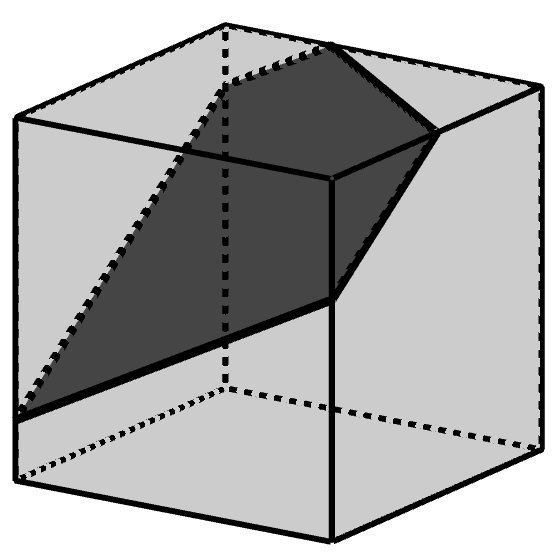} & \includegraphics[height=1.1cm]{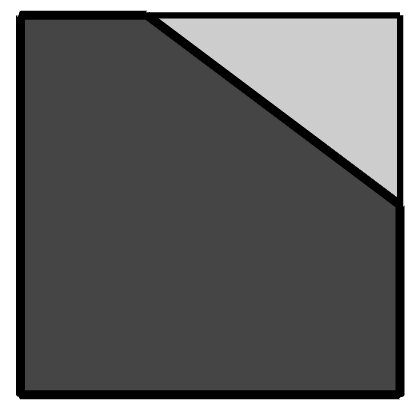} & \begin{tabular}[c]{@{}l@{}}$\max\left\{0,1-(a+b)\right\}\leq c$ \\ and\\ $c\leq 1-b$\end{tabular} & $1-\frac{(a+b+c-1)^2}{2ab}$     & $\frac{3\sqrt{3}}{2}$ \\ \hline
{\footnotesize 7} &	\includegraphics[height=1.2cm]{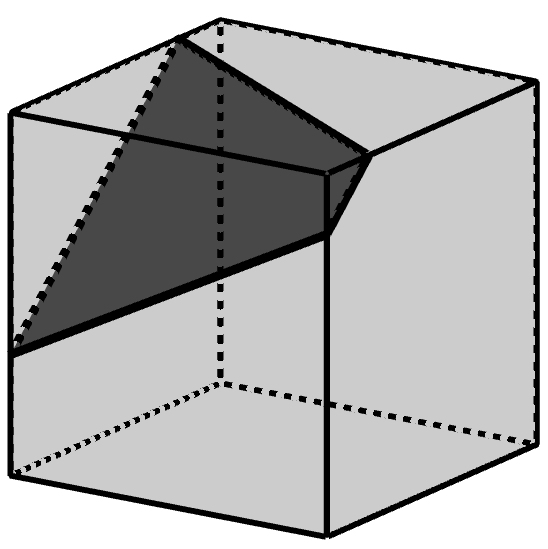} & \includegraphics[height=1.1cm]{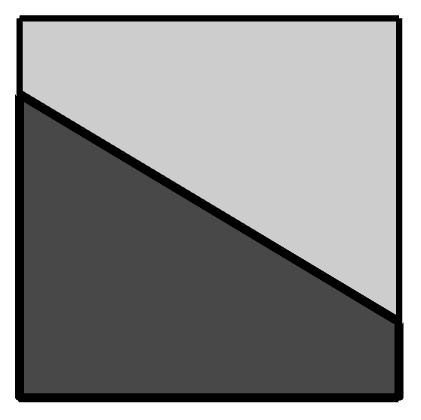} & $1-b\leq c\leq 1-a$                                                                    & $\frac{2-2c-a}{2b}$             & $\frac{3\sqrt{3}}{2}$ \\ \hline
{\footnotesize 8} &	\includegraphics[height=1.2cm]{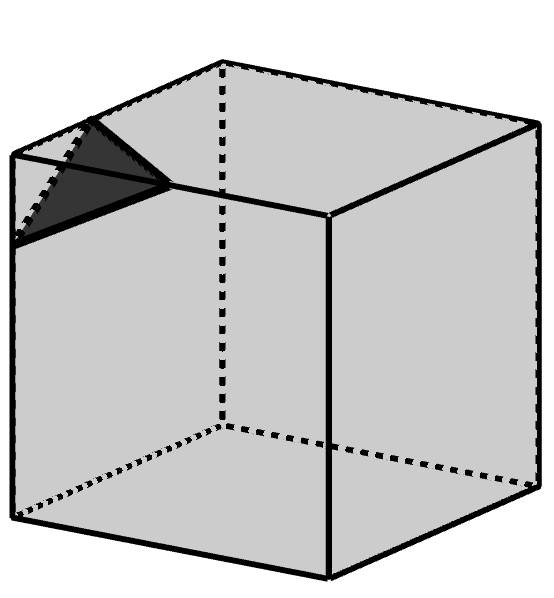} & \includegraphics[height=1.1cm]{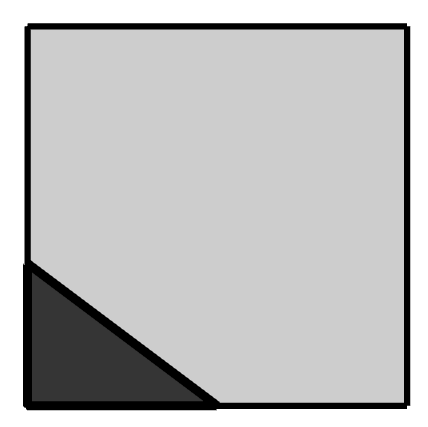} & $1-a\leq c\leq 1$                                                                      & $\frac{(1-c)^2}{2ab}$           & $\frac{3\sqrt{3}}{2}$ \\ \hline
\end{tabular}
\caption{The first column contains the name (used in Figure \ref{w70}) of the case. The second column shows the way of the intersection of the plane and the unit cube, the third column shows the projection of the above. In the fourth -- conditions --  column, we describe the region of $(a,b,c)$ under which we are in the given case. The fifth column contains the value of the function $\widetilde{F}$ in the given case, and the last, Lipsch., column contains the upper bound for the Lipschitz constant in the given case.}
\label{w71}
\end{table}

\begin{figure}[!htb]
    \centering
    \includegraphics[width=\textwidth]{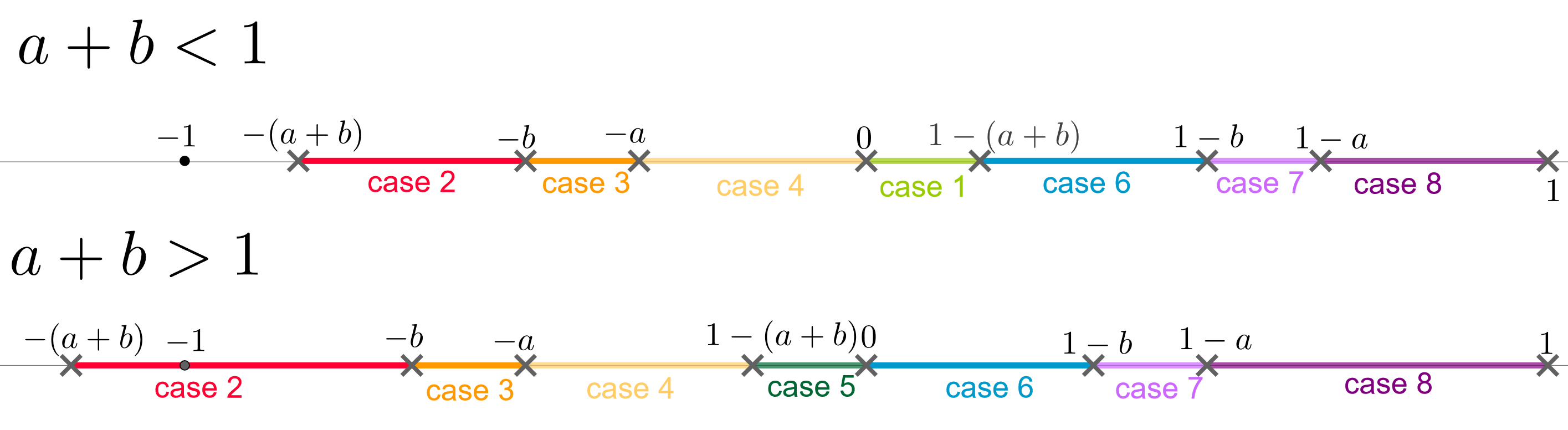}
    \caption{Figure for Table \ref{w71}. The two figures show that for a fixed $a$ and $b$ what case does the different values of $c$ give. The cases are named after the first column of Table \ref{w71}, e.g. when $-b\leq c \leq -a$ we are in case 3. In the upper figure we consider the case, when $a+b<1$, in the lower, when $a+b>1$.}
    \label{w70}
\end{figure}

Since according to the Table all the Lipschitz-constants are less then 4, 
\begin{equation}\label{w62}
    \left\|\widetilde{F}(a,b,c)-\widetilde{F}(\widehat{a}, \widehat{b}, \widehat{c})\right\|<4\left\|(a,b,c)-(\widehat{a}, \widehat{b}, \widehat{c})\right\|.
\end{equation}
Using that 
\begin{equation}
    A:= \nabla_{a,b,c} g(a,b,c,u,v,w)=\begin{bmatrix}
1 & 0 & 0\\
0 & 1 & 0 \\
3u & 3v & 3
\end{bmatrix}
\end{equation}
and that $\sup_{(u,v,w)\in \mathcal{A}}\|A\|<4$, it follows, that
\begin{multline}\label{w64}
    \left\|\sum_{(u,v,w)\in\mathcal{A}}\widetilde{F}(g(a,b,c,u,v,w))-\widetilde{F}(g(\widehat{a},\widehat{b},\widehat{c},u,v,w))\right\| \\
    < \sum_{(u,v,w)\in\mathcal{A}} 4 \cdot \left\|g(a,b,c,u,v,w)-g(\widehat{a},\widehat{b},\widehat{c},u,v,w)\right\| \\
    \leq |\mathcal{A}| \cdot 4 \cdot 4 \left\|(a,b,c)-(\widehat{a}, \widehat{b}, \widehat{c})\right\| 
    \leq 7 \cdot 4 \cdot 4\cdot \left\|(a,b,c)-(\widehat{a}, \widehat{b}, \widehat{c})\right\|.
\end{multline}
\begin{multline}
    \left\|\widetilde{H}(a,b,c)-\widetilde{H}(\widehat{a}, \widehat{b}, \widehat{c})\right\|\leq \frac{5}{9}\left\|\widetilde{F}(a,b,c)-\widetilde{F}(\widehat{a}, \widehat{b}, \widehat{c})\right\|\\
    +\frac{1}{9}\left\|\sum_{(u,v,w)\in\mathcal{A}}\widetilde{F}(g(a,b,c,u,v,w))-\widetilde{F}(g(\widehat{a},\widehat{b},\widehat{c},u,v,w))\right\| \\
    < \frac{44}{3} \left\|(a,b,c)-(\widehat{a}, \widehat{b}, \widehat{c})\right\|
    <15 \left\|(a,b,c)-(\widehat{a}, \widehat{b}, \widehat{c})\right\|,
\end{multline}
where the second inequality follows from the combination of 
\eqref{w62} and \eqref{w64}.
Hence, 
\begin{equation}\label{w45}
    \widetilde{H}(\widehat{a}, \widehat{b}, \widehat{c})>\widetilde{H}(a,b,c)-15 \left\|(a,b,c)-(\widehat{a}, \widehat{b}, \widehat{c})\right\|.
\end{equation}
\subsubsection{Numerical calculations using Wolfram Mathematica}\label{w50}
First we construct a grid $G$ 
depending on a given positive number $\widehat{d}$.
  Let $a_0=\frac{1}{3}$ and if $a_n$ is defined and $a_n+\widehat{d}\leq 1$, let $a_{n+1}=a_n+\widehat{d}$, otherwise $I=n$. For a given $i\in\left\{0,\dots,I\right\}$, let $b_{i,0}=a_i$ and if $b_{i,n}+\widehat{d}\leq 1$ then we define $b_{i, n+1}= b_{i,n}+\widehat{d}$, otherwise $J_i=n$. For a given $i\in\left\{0,\dots,I\right\}$, $j \in \left\{0,\dots, J_i\right\}$, let $c_{i,j,0}=\frac{2}{3}-(a+b) \leq 0$ and if $c_{i,j,n}$ is defined and $c_{i,j,n}+\widehat{d}\leq \frac{1}{3}$ then we put $c_{i,j,n+1}=c_{i,j,n}+\widehat{d}$ otherwise $K_{i,j}=n$. Now if we choose a point $(\widehat{a}, \widehat{b}, \widehat{c})$ such that $ \frac{1}{3}\leq \widehat{a}\leq 1$, $\widehat{a} \leq \widehat{b}\leq 1$ and $\frac{2}{3}-(\widehat{a}+\widehat{b})\leq \widehat{c}\leq \frac{1}{3}$, then there exists a grid point $(a_i, b_{i,j}, c_{i,j,k})\in G$, $i\in \left\{0\dots I\right\}, j \in \left\{0,\dots, J_i\right\}, k\in \left\{0, \dots, K_{i,j}\right\}$ such that $\|(\widehat{a}, \widehat{b}, \widehat{c})-(a_i, b_{i,j}, c_{i,j,k})\|<\sqrt{3}\widehat{d}$. Observe that 
  \eqref{w45}, yields the following implication: 
  If 
  $\widetilde{H}(a,b,c)\geq 15\cdot \sqrt{3}\cdot \widehat{d}$ holds for every $(a,b,c)\in G$
  then 
  we have
  $\widetilde{H}(a,b,c)\geq 0$ for every $(a,b,c)$ . 
  We choose $\widehat{d}=\frac{1}{500}$ and with Wolfram Mathematica we evaluate $\widetilde{H}$ at every grid point. 
  We obtain that
  the minimum of the results is $\frac{62509}{1125000}$. This is  greater than $15\cdot \sqrt{3}\cdot \frac{1}{500}$.
\section{Coin-tossing integer self-similar sets on the line}\label{y20}

We consider the IFS $\mathcal{F}=\{f_i\}_{i=0}^{M-1}$ defined in Example \ref{y28}. We would like to invoke the notation introduced in Definition \ref{y14}. To do so without loss of generality we may assume that the distinct elements of $\mathcal{F}$ are the first $m$ elements of $\mathcal{F}$. Using them we form a new IFS 
\begin{equation}
\label{y92}
\mathcal{S}:=
\left\{
    S_i(x):=\frac{1}{L}x+t_i
 \right\}_{i=0}^{m-1}.
\end{equation}
Moreover, also without loss of generality we may assume, that 
\begin{equation}
\label{y93}
L\geq 2,\ L\in\mathbb{N}, \text{ and }
t_0,\dots , t_m\in\mathbb{Z},
0= t_0< t_1 < \cdots < t_{m-1},\ L-1|t_{m-1}.
\end{equation}
Recall from \ref{u80} in Definition \ref{y14} that $n_j= \#\left\{f_{i}\in \mathcal{F}: f_{i}=S_{j}\right\}$.
From now on we denote the natural number (see \eqref{y93})  $\widetilde{n}:=\frac{t_{m-1}}{L-1}$. 
In this case the interval
\begin{equation}\label{y11}
    I:= \left[0,\widetilde{n}\cdot L\right]
\end{equation}
satisfies \eqref{y35}.
   \begin{lemma}[\cite{ruiz2009dimension}]\label{y90}
    \
  \begin{enumerate}
      \item For all $\ell \in[M], k\in [ N]$ there exist $i\in [N] $ and
$b\in [L] $ such that
$S_\ell (J^k)=J^i_{b}$.
\item $S_\ell (J^k)=J_{b}^{i }$ if and only if $S_\ell (J^k_{\underline{a}})=J_{b\underline{a}}^{i }$ for all $k\geq 0$, $\underline{a}\in [ L] ^k $.
\item If $S _{\ell }^{-1 }\left( J _{b\underline{a}}^{i } \right)\ne J _{\underline{ a}}^{k } $ holds for all
$k\in [ N] $ then $\nu \left( S _{\ell }^{ -1} (J _{b\underline{a}}^{i }) \right)=0$.
  \end{enumerate}
   \end{lemma}

\begin{corollary}
 For all $\underline{\ell}\in[m]^n$ and $k \in [N]$ there exist an $i \in [N]$ and $\underline{a} \in [L]^n$ such that 
$$
S_{\underline{\ell}}\left(J^k\right)=J^{i}_{\underline{a}}.
$$
\end{corollary}
\begin{proof}
 It follows from the first and second part of the previous Lemma (\ref{y90}) and mathematical induction.
\end{proof}
 It is easy to see, that the matrices are well-defined. Namely, for a given $\ell$ and $k$, for distinct $i_1$ and $i_2$ satisfying $S_{i_1}(J^k)=J^{\ell}_{a}= S_{i_2}(J^k)$, we have $t_{i_1}=t_{i_2}$, which is a contradiction, since by definition the translations differ.

Recall that the matrices $\left\{A_a\right\}_{a \in [L]}$ were defined in \eqref{y89}. Let 
\begin{equation}\label{y88}
    A:=\sum_{a \in [ L] }A_a .
\end{equation}
 The meaning of $A_{\underline{a}}(\ell, k)$ is simply the number of indices $(i_1, \dots, i_n)\in [M]^n$ such that $f_{i_1, \dots, i_n}(J^{k})=J^{\ell}_{\underline{a}}$, which is stated in the following lemma.
\begin{lemma}\label{u83}
For for any $n\in \N$, $n\geq 1$, $\underline{a} \in [L]^n$ and $\ell, k\in [N]$
\begin{equation}
     A_{\underline{a}}(\ell, k)= \#\left\{(i_1, \dots, i_n) \in [M]^n: f_{i_1, \dots, i_n}(J^k)=J_{\underline{a}}^{\ell}\right\}
\end{equation}
\end{lemma}
\begin{proof}
 Follows easily from mathematical induction on $n$.
\end{proof}
 
In what follows we denote 
\begin{equation}\label{y8}
    ||A||:= \underline{e}^T \cdot A \cdot \underline{e}, \quad \underline{e}^T=(1, \dots, 1)\in \R^N.
\end{equation}
Note that we have non-negative matrices, hence the norm defined above in our case coincide with $\| A \|_{1,1}$.

\subsection{The probability space}\label{u69}

First of all we define coin tossing self-similar sets precisely. Throughout this section we follow the method of Falconer and Xiong \cite{falconer2014exact}, only we simplify it a little, since the construction we use is way more simple. 
First let $\widehat{\Omega}:=\left\{0,1\right\}^{[M]}$, be the subsets of $[M]$, and let $\widehat{\mathcal{A}}$ be the discrete sigma-algebra on it. For an $\omega=(\omega_1, \dots, \omega_M)\in \widehat{\Omega}$ and $k=\#\left\{\ell: \omega_{\ell}=1\right\}$:
\begin{equation}\label{y17}
    \widehat{\mathbb{P}}(\left\{(\omega_1, \dots, \omega_M)\right\}):=p^k(1-p)^{M-k}.
\end{equation}
It is easy to see that $\widehat{\mathbb{P}}$ is a probability measure on the measurable-space $(\widehat{\Omega}, \widehat{\mathcal{A}})$. On this space we define the random variable 
\begin{equation}\label{y16}
    X(\omega):=(X_1(\omega), \dots, X_M(\omega)), \text{ where } X_k(\omega)=\omega_k.
\end{equation}
For the M-ary tree $\mathcal{T}$, we define
\begin{equation}\label{u75}
    (\Omega, \mathcal{A}, \mathbb{P}):= \bigotimes\limits_{{\tt i}\in\mathcal{T}}(\Omega_{\tt i}, \mathcal{F}_{\tt i}, \mathcal{P}_{\tt i}) \text{, where } (\Omega_{\tt i}, \mathcal{A}_{\tt i}, \mathcal{P}_{\tt i})=(\widehat{\Omega}, \widehat{\mathcal{A}}, \widehat{\mathbb{P}}).
\end{equation}
For each ${\tt i}\in \mathcal{T}$, we define the 

\begin{equation}
    \pi_{\tt i}: \Omega\to \Omega_{\tt i}, \quad \pi_{\tt i}(\pmb{\omega}):=\omega^{\tt i}.
\end{equation}
We also define
\begin{equation}
    X^{[{\tt i}]}:=X\circ \pi_{\tt i}, \text{ i.e. } X^{[{\tt i}]}_j=\left\{
\begin{array}{ll}
1 
,&
\hbox{if $\omega^{\tt i}_j=1$;}
\\
0
,&
\hbox{otherwise.}
\end{array} \right.
\end{equation}
Hence, $X^{[{\tt i}]}$ are i.i.d random variables with the same distribution as of $X$.
For an ${\tt i}=(i_1, \dots, i_n) \in \mathcal{T}$, we define 
\begin{equation}
    X^{{\tt i}}(\pmb{\omega}):=X^{[ i_1]}(\pmb{\omega})\cdots X^{[ i_n]}(\pmb{\omega}),
\end{equation}
and the level $n$-th and the eventual survival set:
\begin{equation}\label{u76}
    \mathcal{E}_n(\pmb{\omega}):= \left\{{\tt i} \in [M]^n: X^{{\tt i}}(\pmb{\omega})=1\right\}, \, \mathcal{E}_{\infty}(\pmb{\omega}):=\left\{{\tt i}\in \Sigma^{(M)}: {\tt i}|_n \in \mathcal{E}_{n}(\pmb{\omega}) \; \forall n \in \N\right\}.
\end{equation}
We are given the deterministic self-similar IFS $\mathcal{F}$ and $\mathcal{S}$ as in Definition \ref{y14} on the line. Put 
\begin{equation}
    E_{n}(\pmb{\omega}):=\bigcup\limits_{{\tt i}\in \mathcal{E}_{n}(\pmb{\omega})}I_{{\tt i}},
\end{equation}
where $I_{{\tt i}}:= f_{{\tt i}}(I)$ (recall, that $I$ was defined in \eqref{y11}).
The Coin tossing integer self-similar set on the line corresponding to probability $p$ and the IFS $\mathcal{F}$ is defined as
\begin{equation}
    \Lambda_{\mathcal{F}}(p)=\Lambda_{\mathcal{F}}(p,\pmb{\omega}):=\bigcap\limits_{n=1}^{\infty}E_{n}(\pmb{\omega}).
\end{equation}

\subsection{Proof of Theorem \ref{y26}}

The proof uses the method introduced by  Dekking and Simon \cite{SD2005}.

We say that the interval $J$,
\begin{itemize}
    \item  is of type $V$ if $J=f_{\tt i}\left(J^V\right)$ for some ${\tt i}\in \mathcal{E}_n$ for some $n$;
    \item  is of type $V$ with multiplicity $s\geq 0$ if  $$\#\left\{{\tt i}\in \mathcal{E}_n:J=f_{{\tt i}}\left(J^V\right)\right\}=s.$$
\end{itemize}

For $n\geq 0$; $U, V \in [N]$ and $\underline{a} \in [L]^n$, let
\begin{equation*}
    S^{U,V}\left(\underline{a}\right):=\left\{{\tt i} \in [M]^n:
    {\tt i} \in \mathcal{E}_n\text{ and } f_{{\tt i}}\left(J^V\right)=J^U_{\underline{a}}\right\}.
\end{equation*}
Thus, $J _{\underline{a}}^{U }$ is of type $V$ with multiplicity $\# S^{U,V}\left(\underline{a}\right)$.
For a $V\in [N]$ let  $\{S^{U,V}_U\left(\underline{a}\right)\}_{U\in [N]}$ be independent random variables such that the distribution of $S^{U,V}_{U}\left(\underline{a}\right)$ is equal to the distribution of $S^{U,V}\left(\underline{a}\right)$, and let 
\begin{equation*}
    S^V\left(\underline{a}\right):= \bigcup\limits_{U\in [N]}S^{U,V}_{U}\left(\underline{a}\right).
\end{equation*}
The following lemma illuminates the meaning of the previously introduced random variables. 
\begin{lemma}\label{y5}
For any $U, V \in [N]$ and $\underline{a}\in [L]^n$:
\begin{equation}
    \mathbb{E}\left(\#S^{U,V}\left(\underline{a}\right)\right)=p^nA_{ \underline{a}}\left(U,V\right)
\end{equation}
and
\begin{equation}
    \mathbb{E}\left(\#S^{V}\left(\underline{a}\right)\right)=p^n CS_{\underline{a},V}
\end{equation}
\end{lemma}
\begin{proof}
By Lemma \ref{u83} in the deterministic case the first part follows with $p=1$, and for every level $n$ cylinder the probability of retention is $p^n$, hence the assertion follows. For the second part:
\begin{equation}
    \mathbb{E}\left(\#S^{V}\left(\underline{a}\right)\right)=\sum_{U \in [N]}\mathbb{E}\left(\#S^{U,V}\left(\underline{a}\right)\right)=p^n\sum_{U\in[N]}A_{\underline{a}}\left(U,V\right)
\end{equation}
\end{proof}
\begin{lemma}\label{y4}
There exists a $U\in [N]$ and $\underline{a}\in [L]^n$ such that \begin{equation}
    \mathbb{P}\left(\left\{\forall \,{V}\in [N]:\#S^{U,V}\left(\underline{a}\right)>0\right\}  \right)>0
\end{equation}
\end{lemma}
\begin{proof}
The lemma follows from the second condition of the theorem, namely: under that condition there exist a $U\in [N]$ and $\underline{a}\in [L]^n$ for some $n$ such that $A_{\underline{a}}\left(U, V\right)>0$ for all $V\in [N]$.
The events $\left\{\#S^{U,V}\left(\underline{a}\right)>0\right\}_{V=0}^{N-1}$ are not exclusive and each has positive probability, hence the probability that all of them happens simultaneously is also positive.
\end{proof}
Fix $\widetilde{U}$ and $\widetilde{\underline{b}}_k\in[L]^k$   in a way that lemma \ref{y4} holds for $U=\widetilde{U}$ and $\underline{a}= \widetilde{\underline{b}}_k$. Let
\begin{equation}\label{u64}
\gamma_{V}:=p\cdot \min_{a \in [L]} CS_{a,V} \quad \text{and} \quad \gamma:= \min_{V\in [N]}\gamma_V.
\end{equation}
Note that $\gamma >1$ is equivalent to the the first condition of Theorem \ref{y26}. That is
\begin{equation}
\label{u63}
\gamma >1 \Longleftrightarrow
p\cdot CS_{a,j}>1 \text { for all } a\in [L]
\text{ and } j\in [N].
\end{equation}

Let $H^0$ be the event in Lemma \ref{y4}: that the interval $J^{\widetilde{U}}_{\widetilde{\underline{b}}_k}$ is of every type, i.e. 
\begin{equation}
    H^0:=\left\{\forall\, V\in [N]: S^{\widetilde{U}, V}\left(\widetilde{\underline{b}}_k\right)>0 \right\}.
\end{equation}
Then by Lemma \ref{y4}, 
\begin{equation}
    p_0:= \mathbb{P}\left(H^0\right)>0.
\end{equation}
This means that with positive probability we can find
\begin{equation}\label{y2}
    {\tt i}^{0,0}, \dots, {\tt i}^{0,N-1} \in [M]^k \text{ such that } f_{{\tt i}^{0,V}}\left(J^V\right)=J^{\widetilde{U}}_{\widetilde{\underline{b}}_k},
\end{equation}
meaning that we can find indices that make $J^{\widetilde{U}}_{\widetilde{\underline{b}}_k}$ of every type.

In what follows we define several processes counting the multiplicity of different types of the intervals of the form $J^V_{\underline{a}_n}$ for $\underline{a}_n=a_1, \dots, a_n \in [L]^n$ and $V\in [N]$ for the different values of $n\in \N$. We apply large deviation theory to prove that these processes simultaneously do not die out with positive probability, implying that the random attractor contains an interval with positive probability. Lastly a standard argument reveals that this is a 0-1 event conditioned on non-extinction.

We say that ${\tt i}\in [M]^*$ makes $J^U_{\underline{a}}$ ($U \in [N], a \in [L]^{|{\tt i}|}$) a type $V$ ($V\in [N]$) interval if the following two holds:
\begin{itemize}
    \item ${\tt i} \in \mathcal{E}_{|{\tt i}|}$ and
    \item $f_{\tt i}(J^V)=J^U_{a}$.
\end{itemize}
First  we  define a process, that collects the different set of indices $\{{\tt i}_V\in [M]^*\}_{V\in [N]}$, such that the elements of the set makes an interval of all type. Start the process with the zeroth level $N$-tuple: 
\begin{equation}
\mathfrak{T}^0\left(\emptyset\right):=\left\{\left({\tt i}^{0,0}, \dots, {\tt i}^{0,N-1}\right)\right\}.
\end{equation}
and for $\underline{c}_n \in [L]^n$ the level-$n$ $N$-tuple of  $\underline{c}_n$ is:
\begin{equation}
    \mathfrak{T}^n\left(\underline{c}_n\right)=\left\{\left({\tt i}^{n,0}_0, \dots, {\tt i}^{n,N-1}_0\right), \dots,\left({\tt i}^{n,0}_{j-1}, \dots, {\tt i}^{n,N-1}_{j-1}\right) \right\},
\end{equation}
in a way that the following three conditions hold:
\begin{itemize}
    \item ${\tt i}^{n,V}\in [M]^{n+k}$,
    \item $f_{{\tt i}^{n,V}}\left(J^V\right)=J^{\widetilde{U}}_{\underline{\widetilde{b}}_k\underline{c}_n}$, meaning that ${\tt i}^{n,V}$ makes  $J^{\widetilde{U}}_{\underline{\widetilde{b}}_k\underline{c}_n}$ be of  type $V$.
    \item For all $\ell_1, \ell_2 \in [j]$ and for all $V \in [N]$: ${\tt i}_{\ell_{1}}^{n,V}\neq {\tt i}_{\ell_{2}}^{n,V}$, meaning that all elements appear only once.
\end{itemize}
Observe that the distribution of $\# \mathfrak{T}^n\left(\underline{a}_n\right)|H^0$ is the same as the distribution of $\min\limits_{U\in [N]}\#S^U\left(\underline{a}_n\right)$. Now given $H^0$ consider $\#\mathfrak{T}^n\left(\underline{a}_n\right)$.
By condition (1) of the Theorem, we have $\gamma >1$
(see \eqref{u63}).
Choose $1<\rho<\gamma$, and define the events
\begin{equation}
    H_n:=\left\{\forall \underline{a}_n\in [L]^n: \#\mathfrak{T}^n\left(\underline{a}_n\right)>\rho^n \right\}.
\end{equation}
Recall that $1<\gamma= \min_{V \in [N]}\min_{a \in [L]}\mathbb{E}(\#S^V(a))$, by Lemma \ref{y5} and \eqref{u64}.
As usual let $\overline{A}$ denote the complement of the event $A$.


\begin{lemma}\label{y1}
There exists a $0<\delta<1$ such that for all $n\in\N$ 
\begin{equation}
\mathbb{P}\left(\overline{H}_{n+1}\left|\right.H_n, H_0\right)\leq L^{n+1}\cdot N \cdot \delta^{\rho^{n}}.
\end{equation}
\end{lemma}
\begin{fact}[Azuma-Hoeffding inequality]\label{y0}
Let $Z^U_{0}\left(k\right), \dots, Z^{U}_{\ell}\left(k\right)$ be i.i.d random variables distributed according to $\#S^{U}(k)$, then
\begin{equation}
    \mathbb{P}\left(Z^U_{0}\left(k\right)+\dots+Z^{U}_{\ell}\left(k\right)\leq \ell \cdot \rho\right)\leq \delta^{\ell}, \text{ for some }0<\delta<1.
\end{equation}
\end{fact}
\begin{proof}[Proof of Lemma \ref{y1}]
\begin{multline}
\mathbb{P}\left(\overline{H}_{n+1}|H_{n}, H_0\right) \\
\leq \sum_{\underline{a}_n \in [L^n]}\sum_{a \in [L]}\mathbb{P}\left(\#\mathfrak{T}^n\left(\underline{a}_n a\right) <\rho^{n+1}\left|\right.\forall \, \underline{c}_n\in [L]^n: \#\mathfrak{T}^n\left(\underline{c}_n\right)\geq \rho^n\right)
\end{multline}
For a fixed $\underline{a}_n\in [L]^n$, $a \in [L]$
\begin{multline}
   \mathbb{P}\left(\#\mathfrak{T}^n\left(\underline{a}_n a\right)<\rho^{n+1}\left|\right.\forall \underline{c}_n\in [L]^n: \#\mathfrak{T}^n\left(\underline{c}_n\right)\geq \rho^n\right) \\
   \leq \sum_{U \in [N]}\mathbb{P}\left(\#S^{U}\left(\underline{a}_n a\right)<\rho^{n+1}\left|\right.\min_{V\in [N]}\#S^{V}\left(\underline{a}_n\right)\geq \rho^n\right).
\end{multline}
For any fixed $U$ we can use Fact \ref{y0} to upper-bound the last probability. This is because conditioned on we have at least $\rho^n$ level $n$ $N$-lets, $\#S^U\left(\underline{a}_n a\right)$ is at least the sum of at least $\rho^n$ independent random variables, distributed according to $\#S^U\left(a\right)$. 
The reason that the random variables in this sum are independent is as follows: 
By the construction, given $\mathcal{E}_n$, the events that we retain or discard different cylinders on level $n+1$ are independent, hence the random variables in the sum are also indeed independent. 
\begin{multline}\label{u99}
\mathbb{P}\left(\#S^{U}\left(\underline{a}_n a\right)<\rho^{n+1}\left|\right.\min_{V\in [N]}\#S^{V}\left(\underline{a}_n\right)\geq \rho^n\right)\leq \delta\left(U, a\right)^{\rho^{n}} \\ 
\text{ for some }0<\delta\left(U,a\right)<1,
\end{multline}
for any particular $U\in [N]$,  $\underline{a}_n\in [L]^n$ and $a \in [L]$. Choose $\delta:=\max\limits_{U\in[N]}\max\limits_{a \in [L]} \delta\left(U,a\right)$, in this way we get that:
\begin{equation}\label{u73}
    \mathbb{P}\left(\overline{H_{n+1}}\left|\right.H_n\right)\leq \sum_{\underline{a}_n\in [L^n]}\sum_{a\in [L]}\sum_{U\in [N]} \delta^{\rho^{n}}=L^{n+1}\cdot N \cdot \delta^{\rho^{n}}
\end{equation}
\end{proof}
\begin{lemma}
For every $n \in \N$:
\begin{equation}
    \mathbb{P}\left(\forall\, V\in [N],\, \forall\, \underline{a}_n\in [L]^n:\:\#S^V\left(\underline{a}_n\right)\geq \rho^n, \right)>0
\end{equation}
\end{lemma}
\begin{proof}
This is an adaptation of Dekking and Simon (\cite{SD2005}) Lemma 1. For $\underline{a}_n \in [L]^n$:
\begin{multline}
    \underline{e}^TA_{\underline{a}_n}\geq p^n \underline{e}^T A_{\underline{a}_n}=[CS_{\underline{a}_n,0}, \dots, CS_{\underline{a}_n,N-1}] 
    \geq [\rho^n, \dots, \rho^n].
\end{multline}
This means that in the deterministic setup we have enough indices for the event to happen for any $\underline{a}_n$. Hence it follows that $\#S^V\left(\underline{a}_n\right)$ happens with positive probability for each $\underline{a}_n$ and $V$, and as the events are not exclusive the intersection of them also has positive probability. 
\end{proof}
It is easy to see, that if we can prove that 
\begin{equation} \label{u98}
    \mathbb{P}\left(\forall\, n, \,\forall\, \underline{a}_n: \: \#\mathfrak{T}^n\left(\underline{a}_n\right)>0\right)>0
\end{equation}
then we prove that $\Lambda_{\mathcal{F}}\left(p\right)$ contains an interval with positive probability. This is because \eqref{u98} means that with positive probability for any $n$ we retain something in every sub-interval $J^{\widetilde{U}}_{\underline{\widetilde{b}}_k\underline{a}_n}$ (for all $\underline{a}_n\in [L]^n$) of $J^{\widetilde{U}}_{\underline{\widetilde{b}}_k}$. Then \eqref{u98} holds, since 
\begin{multline}
    \mathbb{P}\left(H_n, n>r\right)=\mathbb{P}\left(H_0\right)\cdot \mathbb{P}\left(H_r\right)\prod \limits_{n=r}^{\infty}\mathbb{P}\left(H_{n+1}\left|\right.H_n\right) \\
    \geq p_0\cdot \mathbb{P}\left(H_{r}\right)\prod \limits_{n=r}^{\infty}\left(1-L^{n+1}\cdot N\cdot \delta^{\rho^{n}}\right).
\end{multline}
$p_0$ is positive by Lemma \ref{y4}, $\mathbb{P}\left(H_r\right)$ is positive by Lemma \ref{u98}, and we can choose $r$ such that the last expression is positive. 
Hence 
\begin{equation} \label{u96}
    \mathbb{P}\left(\Lambda_{\mathcal{F}}\left(p\right) \text{ contains an interval}\right)>0.
\end{equation}
\begin{lemma}\label{u95}
Let $\theta$ be a possible property of $\Lambda_{\mathcal{F}}(p)$. Let $\mathfrak{A}$ denote the event that $\theta$ holds for $\Lambda_{\mathcal{F}}(p)$. Assume that $\mathfrak{A}$ \begin{enumerate}
    \item  happens almost surely if the process dies out,
    \item  happens if and only if 
$\theta$ holds for every intersection of the set with level $n$ cylinders, namely for every $n$ and $i_1, \dots, i_n$: 
$\theta$ holds for  $\Lambda_{\mathcal{F}}\left(p\right)\cap I_{i_1, \dots, i_n}$ and
\item is not a sure event.
\end{enumerate}
Then conditioned on non-extinction $\mathfrak{A}$ almost surely does not happen.
\end{lemma}
\begin{proof}
The proof is based on a standard argument (similar to the one in \cite[page 471]{Falconer_Grimmet}) using statistical self-similarity. Let
$\mathfrak{C}:=\left\{\#\mathcal{E}_{\infty}>0\right\}$ denote the event that the process does not die out.
By the third assumption of the lemma
\begin{equation}
    \mathbb{P}\left(\mathfrak{A}\right)=\varepsilon<1.
\end{equation}
From the theory of branching processes we know that conditioned on non-extinction the number of retained level n cylinders tends to infinity almost surely, i.e. $\mathbb{P}\left(\#\mathcal{E}_{n}<\widetilde{M}^n\left|\right.\mathfrak{C}\right)\to 0$ as $n\to \infty$ for all $\widetilde{M}$.
Now, by the assumption of the Lemma: 
\begin{equation}
\begin{split}
& \mathbb{P}\left(\mathfrak{A}\left|\right.\mathfrak{C}\right) \\
& =\mathbb{P}\left(\mathfrak{A}\left|\right.\#\mathcal{E}_{n}\geq \widetilde{M}\right)\cdot \mathbb{P}\left(\#\mathcal{E}_{n}\geq \widetilde{M}\left|\right.\mathfrak{C}\right) \\
   & +\mathbb{P}\left(\mathfrak{A}\left|\right.\#\mathcal{E}_{n}< \widetilde{M}\right)\cdot \mathbb{P}\left(\#\mathcal{E}_{n}< \widetilde{M}\left|\right.\mathfrak{C}\right) \\
 & \leq \varepsilon^{\widetilde{M}}+\mathbb{P}\left(\#\mathcal{E}_{n}< \widetilde{M}\left|\right.\mathfrak{C}\right).
\end{split}
\end{equation}
The second inequality follows from the second condition of the lemma $$\mathfrak{A}
=\{\forall n,\,\forall i_1, \dots, i_n
\in \mathcal{E}_n: \theta
\text{ holds for }
\Lambda_{\mathcal{F}}(p)
\cap I_{i_1, \dots, i_n}\}$$ and statistical self-similarity. 
Namely from the fact that conditioned on $\mathcal{E}_n$ in all the retained level $n$-cylinders we have independent 
rescaled copies of $\Lambda_{\mathcal{F}}(p)$.
First let $n\to \infty$, then $\widetilde{M}\to \infty$ gives that 
\begin{equation}
    \mathbb{P}\left(\mathfrak{A}\left|\right.\mathfrak{C}\right)=0.
\end{equation} 
\end{proof}
\begin{proof}[Proof of Theorem \ref{y26}]
Follows from \eqref{u96} and Lemma \ref{u95}, using that does not containing an interval is property that satisfies the assumptions of Lemma \ref{u96}. 

\end{proof}
\color{black}

\subsection{Proof of Theorem \ref{y21}}
\begin{proof}[Proof of Theorem \ref{y21}]
Parts of the following proof resembles \cite[Proof of Thorem 1 (b)]{SD2005}.
Let $a\in [L]$ such that the spectral radius of $p \cdot A_a$ is smaller than 1. 
Solely for this section let $\underline{a}_n$ denote the $n$-long vector consisting only of $a$ i.e.
\begin{equation}
    \overline{a}^n:=(a, \dots, a).
\end{equation}
It is a well known fact that if the spectral radius of an $N\times N$ matrix $B$ is smaller than 1, then 
\begin{equation}
    \lim_{n\to \infty}\left\|B^n\right\|=0.
\end{equation}
This implies by the assumptions of the theorem, that $\lim\limits_{n \to \infty}\left\|(p\cdot A_a)^n)\right\|=\lim\limits_{n \to \infty}\left\|p^n\cdot A_{\overline{a}^n}\right\|=0$.
By the sub-multiplicativity  of the $\|.\|$ matrix norm defined in \eqref{y8}, it follows that for any $\underline{c}_k=(c_1, \dots, c_k)\in [L]^j$:
\begin{equation}\label{y7}
    \lim_{n \to \infty} \left\|p^{n+k}A_{\underline{c}_k\overline{a}^n}\right\|\leq \lim_{n \to \infty} \left\|p^{k}A_{\underline{c}_k}\right\|\cdot\left\|p^n A_{\overline{a}^n}\right\|= \left\|p^k A_{\underline{c}_k}\right\|\lim_{n \to \infty}\left\|p^n A_{\overline{a}^n}\right\|=0.
\end{equation}
Let $\underline{c}_{k}\in [L]^k$ be given and $Z_n$ denote the number of level $k+n$ cylinders intersecting
$\cup_{j \in [N]}J^j_{\underline{c}_k\underline{a}_n}$. 
From Lemma \ref{y5} we know that 
\begin{equation*}
    \mathbb{E}(Z_n)= \|p^{k+n}A_{\underline{c}_k\overline{a}^n}\|,
\end{equation*}
and hence from \eqref{y7}, it follows that 
\begin{equation}\label{u82}
\lim_{n\to \infty}\mathbb{E}(Z_n)=0,
\end{equation}
thus by Markov's inequality:
\begin{equation*}
    \mathbb{P}(Z_{n}\geq 1)\leq \mathbb{E}(Z_n)\to 0 \text{ as } n\to \infty.
\end{equation*}
in this way the points
\begin{equation*}
    \bigcup_{j\in [N]}\bigcap\limits_{n\to \infty}J^j_{\underline{c}_k\overline{a}^n}
\end{equation*}
are not contained in $\Lambda_{\mathcal{F}}(p)$ with probability one.
By varying $\underline{c}_k$ we get a countable dense set which is not contained in $\Lambda_{\mathcal{F}}(p)$ with probability one, hence it can not contain an interval.
\end{proof}
\subsection{Proof of Theorem \ref{y25}}
\begin{prop}\label{u94}
Under the conditions of Theorem \ref{y25} there exists a set $K$ of positive $\mathcal{L}eb_1$ measure such that 
\begin{equation}
    \mathbb{P}\left(x \in
    \Lambda_{\mathcal{F}}\left(p\right)\right)>0 
    \quad  
    \text{for } \mathcal{L}eb_1\text{-a.e. } x \in K. 
\end{equation}
\end{prop}
First we prove Theorem \ref{y25} using Proposition \ref{u94} and Lemma \ref{u95}, then we prove Proposition \ref{u94}.
\begin{prop}\label{u92}
$\mathbb{P}\left(\mathcal{L}eb_1\left(\Lambda_{\mathcal{F}}\left(p\right)\right)>0\right)>0$.
\end{prop}

\begin{proof}[Proof of Proposition \ref{u92} assuming Proposition \ref{u94}]
Since \newline
$\mathbb{P}\left(\mathcal{L}eb_1\left(\Lambda_{\mathcal{F}}\left(p\right)\right)>0\right)>0$ if and only if $\mathbb{E}\left(\mathcal{L}eb_1\left(\Lambda_{\mathcal{F}}\left(p\right)\right)\right)>0$ we will prove the second.
\begin{equation}
\begin{split}
    \mathbb{E}\left(\mathcal{L}eb_1\left(\Lambda_{\mathcal{F}}\left(p\right)\right)\right)
    & =\int\limits_{\Omega}\mathcal{L}eb_1\left(\Lambda_{\mathcal{F}}\left(p, \pmb{\omega}\right)\right)d\mathbb{P}\left(\pmb{\omega}\right) \\
    & \geq 
    \int\limits_{\Omega} \int\limits_{K} \ind{\left\{x\in \Lambda_{\mathcal{F}}\left(p, \pmb{\omega}\right)\right\}}d\mathcal{L}eb_1\left(x\right)d\mathbb{P}\left(\pmb{\omega}\right) \\
    & =
    \int\limits_{K}\int\limits_{\Omega}\ind{\left\{x\in \Lambda_{\mathcal{F}}\left(p, \pmb{\omega}\right)\right\}} d\mathbb{P}\left(\pmb{\omega}\right)d\mathcal{L}eb_1\left(x\right) \\
    & =
    \int\limits_{K}\mathbb{P}\left(x\in \Lambda_{\mathcal{F}}\left(p\right)\right)d\mathcal{L}eb_1\left(x\right)>0,
    \end{split}
\end{equation}
where the last inequality follows form Proposition \ref{u94}.
\end{proof}
\begin{proof}[Proof of Theorem \ref{y25} assuming Proposition \ref{u94}]
By a similar argument as in the proof of Theorem \ref{y26} using Proposition \ref{u92} and Lemma \ref{u95}.
\end{proof}
\color{black}

The proof of Theorem \ref{y25} assuming Proposition \ref{u94} is very similar to the one presented in \cite[Proof of Theorem 2., page 140]{SMS2009} hence we will not repeat this proof. In what follows we prove Proposition \ref{u94}, using the same  --branching process  method -- as in \cite{SMS2009}, but with a modified process.
\color{black}
\begin{proof}[Proof of Proposition \ref{u94}]
We will use the first part of the proof of Theorem \ref{y26}, namely the part until \eqref{y2}.
 We choose
 $\widetilde{U}$ and $\widetilde{\underline{b}}_k\in[L]^k$   in a way that lemma \ref{y4} holds for $U=\widetilde{U}$ and $\underline{a}= \widetilde{\underline{b}}_k$.
Let 
\begin{equation}\label{u91}
    K:=J^{\widetilde{U}}_{\underline{\widetilde{b}}_k},
\end{equation}
 Let $U\sim \text{Uniform}\left(K\right)$ and  $\mathfrak{P}$ and $\mathfrak{E}$ denote the corresponding distribution and expectation respectively. In what follows we will prove that 
\begin{equation}
    \mathfrak{P}\left(\mathbb{P}\left(U \in K\right)>0\right)=1,
\end{equation}
using the theory of branching processes in random environment. We begin with defining the process. Namely,
for a given $\mathbf{a}=\left(a_1, \dots, a_n, \dots\right)\in \Sigma^{\left(L\right)}$
let 
\begin{equation}
\widetilde{\mathfrak{T}}^0\left(\mathbf{a}\right):={\left({\tt i}^{0,0}, \dots, {\tt i}^{0,N-1}\right)}, 
\end{equation}
where ${\tt i}^{0,0}, \dots, {\tt i}^{0,N-1}$ was defined in \eqref{y2}. Lemma \ref{y4} guarantees that such ${\tt i}^{0,0}, \dots, {\tt i}^{0,N-1}$ exists since the second condition of Theorem \ref{y25} is stronger than the second assumption of Theorem \ref{y26}. Recursively if we have 
\begin{equation}
    \widetilde{\mathfrak{T}}^{t-1}\left(\mathbf{a}\right)=\left\{\left({\tt i}^{t-1,0}_1, \dots, {\tt i}^{t-1,N-1}_1\right), \dots, \left({\tt i}^{t-1,0}_s, \dots, {\tt i}^{t-1,N-1}_s\right)\right\},
\end{equation}
then $\left({\tt i}^{t,0}, \dots, {\tt i}^{t,N-1}\right)$ with elements ${\tt i}^{j, V}\in [M]^t$ is in  $\widetilde{\mathfrak{T}}^{t}\left(\mathbf{a}\right)$ if and only if, for some $\Delta$ defined later
the following holds:
\begin{itemize}
    \item $\left|{\tt i}^{t,V}\right|=\Delta\cdot t+k$ 
    \item ${\tt i}^{t,V}|_{\Delta\cdot (t-1)+k}\in \widetilde{\mathfrak{T}}^{t-1}\left(\mathbf{a}\right)$
    \item $f_{{\tt i}^{t,V}}\left(J^V\right)=J^{\widetilde{U}}_{\widetilde{b}_k\mathbf{a}|_{t\cdot \Delta}}$
    \item If ${\tt i}^{t,V}$ is in an element of $\widetilde{\mathfrak{T}}^{t}$, then it is not contained in another element of $\widetilde{\mathfrak{T}}^{t}$.
\end{itemize}
The elements of $\widetilde{\mathfrak{T}}^n\left(\overline{\theta}\right)$ are
 denoted by $\left({\tt i}^{n,0}, \dots, {\tt i}^{n,N-1}\right)$,
 are called level-$n$ $N$-tuples.
Now we define the environment:
\begin{equation}\label{u59}
    \overline{\theta}=\left(\theta_0, \dots, \theta_n, \dots\right)\text{, where } \theta_k=\left(a_{k\Delta+1}, \dots, a_{\left(k+1\right)\Delta}\right),
\end{equation}
and the branching process in random environment is the following: 
\begin{equation*}
    \begin{split}
    & \mathcal{Z}_0\left(\overline{\theta}\right):=\#\widetilde{\mathfrak{T}}^0\left(\mathbf{a}\right)=1, \\
    & \mathcal{Z}_n\left(\overline{\theta}\right):=\#\widetilde{\mathfrak{T}}^n\left(\mathbf{a}\right).
    \end{split}
\end{equation*}
The above defined $\mathcal{Z}_n\left(\overline{\theta}\right)$ process is indeed a branching process in random (i.i.d. hence stationary and ergodic) environment. This is because 
\begin{equation}
    \mathcal{Z}_{n+1}\left(\overline{\theta}\right)= \sum\limits_{i=1}^{\mathcal{Z}_n\left(\overline{\theta}\right)}X_{n,i}\left(\overline{\theta}\right),
\end{equation}
where $X_{n,i}\left(\overline{\theta}\right)$ is the number of level-$n+1$ $N$-tuples coming from the $i$-th level-$n$ $N$-tuple in $\widetilde{\mathfrak{T}}^n\left(\mathbf{a}\right)$. The random variables $\left\{X_{n,i}\left(\overline{\theta}\right)\right\}_{i=1}^{\mathcal{Z}^n\left(\overline{\theta}\right)}$ are independent, because what happens in different retained cylinders are independent of each other.

If $\mathcal{Z}_n\left(\overline{\theta}\right)$ does not die out for a given $\overline{\theta}=\left(a_1, \dots, a_{k\Delta+1}, \dots\right)$, then conditioned on $H^0$ the point $x$ which has $L$-adic expansion \newline $\widetilde{\underline{b}}_k,a_1, \dots, a_{k\Delta+1}, \dots$ shifted with the left endpoint of the interval $J^{\widehat{U}}=\left[\widehat{u}_1, \widehat{u}_2\right]$ i.e.   $x=\widehat{u}_1+\sum_{j=1}^{k}\widehat{b}_{k_j}L^{j-1}+\sum_{j=1}^{\infty}a_jL^{j-1+k}$ is contained in $\Lambda_{\mathcal{F}}\left(p\right)$. This is because if the process does not die out then on every level the cylinder containing $x$ is retained, since it is of a retained type.

Denote the probability that the interval $J^{V}_a$ is of every type with $q\left(a, V\right)$:
\begin{multline}\label{u77}
    q\left(k,V\right):=\mathbb{P}\left(\forall\, U\in [N]\;\exists\, i \in \mathcal{E}_1:f_i\left(J^U\right)=J^V_a\right), \text{ and } \\ q:=\min_{a\in [L]}\max_{V \in [N]}q\left(a,V\right).
\end{multline}
To continue the proof of Proposition \ref{u94} we need the following fact
\begin{fact}
q>0.
\end{fact} 
\begin{proof}[The proof of the Fact]
Every matrix $A_a$ has a positive row $A_{a, V}$ by the second condition in the Theorem \ref{y25}. 
\begin{multline}
\mathbb{P}\left(\{\forall U: \: \#S^{V,U}\left(a\right)>0\}\right)
=\mathbb{P}\left(\prod\limits_{U\in [N]}\#S^{V,U}\left(a\right)>0\right)>0
\\
\Longleftrightarrow
\mathbb{E}\left(\prod\limits_{U\in [N]}\#S^{V,U}\left(a\right)\right)>0,
\end{multline}
which is by independence and Lemma \ref{y5} equivalent to 
$$\prod\limits_{U\in [N]}\mathbb{E}\left(\#S^{V,U}\left(a\right)\right)=\prod\limits_{U\in [N]}p\cdot A_a\left(V,U\right)>0.$$ 
The last inequality holds by the second condition Theorem \ref{y25}.
\end{proof}
Let $V\in[N]$ be arbitrary such that 
$V=\max\limits_{W\in [N]}q\left(a,W\right)$. Then we define
\begin{equation}
    U\left(a\right):=V.
\end{equation}
Using the Theory of branching processes in random environment (more precisely \cite[Theorem 3]{bp_renv}), under the following two conditions $\mathfrak{P}$-almost surely the process ${\mathcal{Z}_n\left(\overline{\theta}\right)}_{n\geq 1}$ does not die out with positive probability.
\begin{description}
    \item[C1] There exists a $c>0$ such that for all $\overline{\theta}$: $\mathbb{P}\left(\mathcal{Z}_1\left(\overline{\theta}\right)>0\right)>c$;
    \item[C2] \label{u72} $\frac{1}{L^{\Delta}}\sum\limits_{\substack{\left(a_1, \dots, a_{\Delta}\right)\\ \in [L]^{\Delta}}}\log\left[\mathbb{E}\left(\mathcal{Z}_1\left(\overline{\theta} \right)|\theta_{0}=\left(a_0, \dots, a_{\Delta}\right)\right)\right]>0$,
  \end{description}
where $\Delta $ is natural number which is to be defined later, but was introduced  as the 
length of the elements of the environmental 
random variable (see \eqref{u59}).
To see that
Condition {\bf{C1}} holds, note that   by the definition of $q$, for some $\varepsilon\in\left(0,1\right)$ we have
\begin{equation}
\mathbb{P}\left(\mathcal{Z}_1\left(\overline{\theta}\right)>0\right)>p_0\cdot q^{\Delta}\left(1-\varepsilon\right).
\end{equation}
The proof of the second condition {\bf {C2}} is a little bit trickier. 
Recall, that
\begin{equation}
     g_j=  \left(\prod_{a\in [L]}CS_{a,j}\right)^{^1/_L}
\end{equation}
and denote
\begin{equation}
    \widehat{\Gamma}=\min_{j \in [N]} g_j.
\end{equation}
By the first assumption of Theorem \ref{y25} $p\cdot \widehat{\Gamma}>1$.


\begin{lemma}
For any $n\in \N$ and $V\in [N]$:
\begin{equation}
    \mathfrak{E}\left[\log\left(\mathbb{E}\left(\#S^V\left(a_1,\dots, a_n\right)\right)\right)\right]\geq n\cdot \log \left[\widehat{\Gamma}\cdot p\right].
\end{equation}
\end{lemma}

\begin{proof}
The proof is by mathematical induction, namely, for $n=1$:
\begin{multline}
    \mathfrak{E}\left[\log\left(\mathbb{E}\left(\#S^V\left(a\right)\right)\right)\right]
    = \frac{1}{L}\sum_{a \in [L]}\log\left(\mathbb{E}\left(\#S^V\left(a\right)\right) \right)\\
    =\frac{1}{L}\log\left(\prod_{i\in [L]}p \cdot CS_{i, V}\right)
    =\log\left(p\cdot g_{V}\right)
    \geq \log\left(p \cdot \widehat{\Gamma}\right).
\end{multline}
Assume that $\mathfrak{E}\left[\log\left(\mathbb{E}\left(\#S^V\left(a_1,\dots, a_n\right)\right)\right)\right]\geq n\cdot \log \left[\widehat{\Gamma}\cdot p\right]$. From $$\mathbb{E}\left(\#S^{U,V}\left(a\right)\right)=p\cdot A_{a}\left(U,V\right),$$it follows, that
\begin{multline}
    \mathbb{E}\left(\#S^V\left(a_1,\dots, a_n, a_{n+1}\right)\right)
    = \sum_{U \in [N]} \mathbb{E}\left(\#S^{U,V}\left(a_{n+1}\right)\right)\cdot  \mathbb{E}\left(\#S^U\left(a_1,\dots, a_n\right)\right) \\
    = \sum_{U \in [N]} \left[\frac{A_{a_{n+1}}\left(U,V\right)}{CS_{a_{n+1}, V}}\right]\cdot  \left[p\cdot {CS_{a_{n+1}, V}}\mathbb{E}\left(\#S^U\left(a_1, \dots, a_n\right)\right)\right]
\end{multline}
by the concavity of the $\log$ function
\begin{multline}
    \mathbb{E}\left(\#S^V\left(a_1,\dots,a_n, a_{n+1}\right)\right)\geq 
    \log\left(p\cdot CS_{a_{n+1}, V}\right) \\+\sum_{U\in [N]}\frac{A_{a_{n+1}}\left(U,V\right)}{CS_{a_{n+1}, V}} \log\left[\mathbb{E}\left(\#S^U\left(a_1, \dots, a_n\right)\right)\right]
\end{multline}
Hence by independence and the linearity of the expectation, we get that
\begin{multline}
    \mathfrak{E}\left[\log\left(\mathbb{E}\left(\#S^V\left(a_1,\dots, a_n ,a_{n+1}\right)\right)\right)\right]
    =\mathfrak{E}\left[\log\left(p\cdot CS_{a_{n+1}, V}\right)\right] \\ + \sum_{U\in [N]}\mathfrak{E}\left[\frac{A_{a_{n+1}}\left(U,V\right)}{CS_{a_{n+1}, V}}\right]\mathfrak{E}\left[ \log\left(\mathbb{E}\left(\#S^U\left(a_1,\dots, a_{n}\right)\right)\right)\right].
\end{multline}
The first part of the sum is
\begin{equation}
\mathfrak{E}\left(\log\left(p\cdot CS_{a_{n+1}, V}\right)\right)= \log\left(g_{V}\cdot p\right)\geq \log\left(\widehat{\Gamma}\cdot p\right),
\end{equation}
and for the second, by the induction hypothesis $$\mathfrak{E}\left[ \log\left(\mathbb{E}\left(\#S^U\left(a_1,\dots, a_{n}\right)\right)\right)\right]\geq n\cdot \log\left(\widehat{\Gamma}\cdot p\right).$$ Hence
\begin{multline}
\sum_{U\in [N]}\mathfrak{E}\left[\frac{A_{a_{n+1}}\left(U,V\right)}{CS_{a_{n+1}, V}}\right]\mathfrak{E}\left[ \log\left(\mathbb{E}\left(\#S^U\left(a_1,\dots, a_{n}\right)\right)\right)\right] \\
= n\cdot \log\left(\widehat{\Gamma}\cdot p\right) \sum_{U\in [N]}\mathfrak{E}\left[\frac{A_{a_{n+1}}\left(U,V\right)}{CS_{a_{n+1}, V}}\right]
=n\cdot \log\left(\widehat{\Gamma}\cdot p\right)
\end{multline}
combining the two above gives the required inequality.
\end{proof}
It follows from
\begin{multline}
    \mathbb{E}\left(\mathcal{Z}_{1}\left(\overline{\theta}\right)|\theta_0=\left(a_1, \dots, a_{\Delta}\right)\right) \geq \mathbb{E}\left(\#S^{U\left(a_{\Delta}\right)}\left(a_1,\dots, a_{\Delta}\right)\right)q\left(a_{\Delta}, U\left(a_{\Delta}\right)\right) \\
    \geq q \cdot \mathbb{E}\left(\#S^{U\left(a_{\Delta}\right)}\left(a_1,\dots, a_{\Delta}\right)\right),
\end{multline}
that
\begin{multline}
\mathfrak{E}\left[\log\left(\mathbb{E}\left(\mathcal{Z}_1\left(\overline{\theta}
\right)|\theta_{0}=\left(a_1, \dots,
a_{\Delta}\right)\right)\right)\right]\geq
\mathfrak{E}\left(\log\left(\mathbb{E}\left(\#S^V\left(a_1, \dots,a_{\Delta} \right)\right)\right)\right) \\ +\log\left(q\right)
\geq \Delta \cdot \log\left(\widehat{\Gamma}\cdot p\right)+\log\left(q\right).
\end{multline}
Since $\log\left(\widehat{\Gamma}\cdot p\right)>0$ we can choose $\Delta$, such that 
\begin{equation}\label{u78}
    \Delta \cdot \log\left(\widehat{\Gamma}\cdot p\right)+\log\left(q\right)>0.
\end{equation}

\end{proof}
\subsection{Proof of Theorem \ref{y24}}
Throughout this section we always assume that the deterministic attractor, denoted by $\Lambda_{\mathcal{F}}$, has positive Lebesgue measure, $\mathcal{L}eb_1(\Lambda_{\mathcal{F}})>0$. Whenever it has zero $\mathcal{L}eb_1$-measure the random attractor also has $0$ measure hence the theorem trivially holds.
In this section we consider the deterministic system and first state a theorem (Theorem \ref{y50}) about it using that we prove Theorem \ref{y25} and then using similar methods as in \cite{barany2014dimension} and \cite{ruiz2009dimension} we prove Theorem \ref{y50}.
On the space $\Sigma^{(L)}$ we define the left shift operator, $\sigma:\Sigma^{(L)}\to \Sigma^{(L)}$:
\begin{equation}
    \sigma(\mathbf{a}):=a_2\dots a_n\dots,
\end{equation}
for $\mathbf{a}=a_1,\dots, a_n,\dots \in \Sigma^{(L)}$. Also we define the uniform measure
\begin{equation}
    \widetilde{\mathcal{L}}:=\left (\frac{1}{L}, \dots, \frac{1}{L}\right)^{\N},
\end{equation}
on $\Sigma^{(L)}$.

\begin{theorem}\label{y50}
There exist a $w< \log\frac{M}{L}$ such that 
\begin{equation}
  \lim_{n \to \infty}\frac{1}{n}\log\|A_{\mathbf{a}|n}\|=w, \quad \text{ for } \widetilde{\mathcal{L}} \text{ almost every }  \mathbf{a}\in \Sigma^{(L)}. 
\end{equation}
\end{theorem}
The proof of this theorem is presented at end of this section.

We denote by $\mathfrak{t}_{j}$ the translation that translates $[0,L]$ into $J^j$, $j \in [N]$. Namely, if $J^j=[k\cdot L, (k+1) \cdot L]$, then $\mathfrak{t}_{j}(x)=x+k\cdot L$. This way for $\mathbf{a}\in [L]^n, n\geq 0$ $\mathfrak{t}_{j}^{-1}(J^j_{\mathbf{a}})=J^0_{\mathbf{a}}$. Also let $\Xi_{j}: \Sigma^{(L)} \to J^j$,
\begin{equation}\label{y38}
    \Xi_{j}(\mathbf{a}):=\mathfrak{t}_{j}\left(\sum_{n=1}^{\infty}a_l L^{1-n}\right) \quad \text{and} \quad  \Xi(\mathbf{a}):=\Xi_0(\mathbf{a}).
\end{equation}
The composition of the natural projection from the symbolic space to the interval $[0,L]$ and of the translation of $[0,L]$ into the $j$-th interval, $J^j$. 
\begin{lemma}\label{y9}
If $x=\Xi_{\ell}(\mathbf{a})$ for some $\ell\in [N]$ and $\mathbf{a}\in \Sigma^{(L)}$, then
\begin{equation}
    \#\left\{(j_1, \dots, j_n)\in [M]^n: x\in I_{j_1, \dots, j_n}\right\}\leq \|A_{\mathbf{a}|n}\|.
\end{equation}
\end{lemma}
The proof of Lemma \ref{y9} is very similar to the proof of Lemma \ref{y76}, hence we will only present the proof of the later one.
\begin{proof}[Proof of Theorem \ref{y24} using Theorem \ref{y50} and Lemma \ref{y9}]
\end{proof}

It follows from Theorem \ref{y50} and Egorov's theorem, that for any fixed $\varepsilon$ there exists a measurable subset $\widetilde{F}_{\varepsilon}$ such that 
\begin{itemize}
    \item $\widetilde{\mathcal{L}}(\widetilde{F}_{\varepsilon}^c)<\varepsilon$
    \item for $\mathbf{a} \in \widetilde{F}_{\varepsilon}$ the convergence $\lim_{n\to \infty}\frac{1}{n}\log\|A_{\mathbf{a}|n}\|=w$ is uniform.
\end{itemize}
Let 
$$F_{\varepsilon}:=\left ( \bigcup\limits_{\ell \in [ N]
}\mathfrak{t}_{\ell}(\widetilde{F}_{\varepsilon})\right
)\cap \Lambda_{\mathcal{F}}.$$
It is easy to see, that $\mathcal{L}eb_1(\Lambda_{\mathcal{F}} \setminus F_{\varepsilon})<N\cdot \varepsilon$.
Choose $v \in \left(w,\log\frac{M}{L}\right)$, and $\delta>0$ such that $v=\log \left(\frac{M-\delta}{L}\right)$. Because of the uniform convergence and of the fact that $v$ is strictly greater than $w$, we can choose $N'$ such that for any $\mathbf{a}\in \widetilde{F}_{\varepsilon}$ and $n > N'$:
\begin{equation}\label{y49}
    \frac{1}{n}\log \|A_{\mathbf{a}|n}\|<v.
\end{equation}
If $x \in F_{\varepsilon}$ and $x=\Xi_{\ell}(\mathbf{a})$ for some $\ell\in [N]$ and $\mathbf{a}\in \Sigma^{(L)}$ as in Lemma \ref{y9}, then
\begin{multline}
    \frac{1}{n}\log \#\left\{(j_1, \dots, j_n) \in [M]^n: x \in I_{j_1, \dots, j_n}\right\} \leq \frac{1}{n}\log \|A_{\mathbf{a}|n}\|<v.
\end{multline}
Hence
\begin{equation}
    \#\left\{(j_1, \dots, j_n) \in [M] ^n: x \in I_{j_1, \dots, j_n}\right\}<\left(\frac{M-\delta}{L}\right)^n.
\end{equation}
As in \ref{u76}, let 
\begin{equation}\label{u79}
\begin{split}
    & \mathcal{E}_{n}=\mathcal{E}_{n}(\pmb{\omega})=\left\{(j_1, \dots, j_n)\in [M]^n: I_{j_1, \dots, j_n}\text{ is selected}\right\},\\
    & \mathcal{E}_{n}(\pmb{\omega},x)=\left\{(j_1, \dots, j_n)\in \mathcal{E}_{n}: x \in I_{j_1, \dots, j_n}\right\} \text { and} \\
    & \mathcal{G}_{n}(U)=\left\{(j_1, \dots, j_n) \in [M]^n: U \cap [j_1, \dots, j_n]\neq \emptyset\right\}.
    \end{split}
\end{equation}
Fix now a $p<\frac{L}{M-\delta}$. With this choice of $p$
$$
\frac{M-\delta}{L}\cdot p<1.
$$
Recall that $\Pi^{(M)}$ denotes the natural projection from the symbolic space $\Sigma^{(M)}=[M]^{\N}$ to $\R$.
\begin{multline}\label{y48}
    \mathbb{E}(\#\mathcal{E}_{n}(x))= \sum_{\substack{(j_1, \dots, j_n)\in  \\ \mathcal{G}_{n}((\Pi^{(M)})^{-1}(x))}}\mathbb{P}(I_{j_1, \dots, j_n}\text{ is selected}) \\
    = \#\mathcal{G}_{n}((\Pi^{(M)})^{-1}(x))\cdot p^n
    < \left(\frac{M-\delta}{L}\cdot p\right)^n=:b^n.    
\end{multline}
Let 
\begin{equation}
    u:=u(\varepsilon):= \mathcal{L}eb_1(F_{\varepsilon}\cap \Lambda_{\mathcal{F}}).
\end{equation}
Recall, that our assumption was that $\mathcal{L}eb_1(\Lambda_{\mathcal{F}})>0$.
Let 
$$\widehat{\mu}:=\frac{\mathcal{L}eb_{1}|_{F_{\varepsilon}\cap \Lambda_{\mathcal{F}}} \times \mathbb{P}}{u},$$
and consider the finite measure-space: $(F_{\varepsilon}\times \Omega, \mathcal{B}\times \mathcal{A}, \widehat{\mu})$ ($\mathcal{B}$ denotes the Borel sigma-algebra and recall: ($\Omega, \mathcal{A}$) was defined in \ref{u75}) and let $\Lambda_{\mathcal{F}}(p,\pmb{\omega})$ denote the random Cantor set corresponding to the realization $\pmb{\omega}$. Then by the previous calculation,
\begin{equation}
    \int\limits_{x \in F_{\varepsilon}} \int\limits_{\pmb{\omega} \in \Omega} \#\mathcal{E}_{n}(\pmb{\omega}, x)d\widehat{\mu}(\pmb{\omega}, x)\leq b^n.
\end{equation}

\begin{lemma}
For $\mathbb{P}$ almost every $\pmb{\omega}\in \Omega$ we have
$$
\mathcal{L}eb_1(\Lambda_{\mathcal{F}}(p, \pmb{\omega}) \cap F_{\varepsilon})=0.
$$
\end{lemma}
\begin{proof}
Assume that there exists an $\widehat{\varepsilon}>0$ such that
\begin{equation}
\label{u29}
\mathbb{P}(\mathcal{L}eb_1(\Lambda_{\mathcal{F}}(p, \pmb{\omega})\cap F_{\varepsilon})>\widehat{\varepsilon})>\widehat{\varepsilon}.
\end{equation} 
The $n$-th approximation of the random attractor is denoted by
$$
\Lambda_{n,\mathcal{F}}(p, \pmb{\omega}):=\bigcup_{(j_1, \dots, j_n)\in \mathcal{E}_{n}(\pmb{\omega})}I_{j_1, \dots, j_n}.
$$
Choose $n>N'$ (such that \eqref{y49} holds).
By the assumption \eqref{u29} $\mathbb{P}(\pmb{\omega} \in \Omega: \mathcal{L}eb_1(\Lambda_{n,\mathcal{F}}(p, \pmb{\omega})\cap F_{\varepsilon})>\widehat{\varepsilon})>\widehat{\varepsilon}$ also holds. Let
\begin{equation}
    H_{n}:= \left\{\pmb{\omega} \in \Omega: \mathcal{L}eb_1(\Lambda_{n,\mathcal{F}}(p, \pmb{\omega})\cap F_{\varepsilon})>\widehat{\varepsilon}\right\}.
\end{equation}
Then 
$\mathbb{P}(H_n)>\widehat{\varepsilon}$. For an $\pmb{\omega}\in H_{n}$, by the definition of $\Lambda_{n,\mathcal{F}}(p, \pmb{\omega})$
$$
F_{\varepsilon}\cap \Lambda_{n,\mathcal{F}}(p, \pmb{\omega})= \bigcup_{\substack{(j_1, \dots, j_n) \\ \in \mathcal{E}_{n}(\pmb{\omega})}}I_{j_1, \dots, j_n} \cap F_{\varepsilon}.
$$
It follows from the assumption \eqref{u29} that
\begin{equation}
    \widehat{\varepsilon}<\mathcal{L}eb_1(\Lambda_{\mathcal{F}}(p, \pmb{\omega})\cap F_{\varepsilon}) \leq \int_{F_{\varepsilon}}\sum_{\substack{(j_1, \dots, j_n) \\ \in \mathcal{E}_{n}(\pmb{\omega})}}\ind_{I_{j_1, \dots, j_n}}(x)dx =\int_{F_{\varepsilon}}\#\mathcal{E}_{n}(\pmb{\omega}, x) dx
\end{equation} 
Since $n>N'$ by \eqref{y48}, $\mathbb{E}(\#\mathcal{E}_{n}(x))<b^n$, and since $b<1$, $b^n$ can be arbitrarily small, which contradicts the above calculation, which shows that  $\mathbb{E}(\#\mathcal{E}_{n}(x))>\widehat{\varepsilon}$.
\end{proof}
\begin{proof}[Proof of Theorem \ref{y24}]
Let $\widetilde{\varepsilon}=\frac{1}{k}$ for $k \in \N$. 
Then there exists a set $F_{\frac{1}{k}}$ 
such that 
$\mathcal{L}eb_1(\Lambda_{\mathcal{F}}\setminus F_{\frac{1}{k}})<N \cdot \frac{1}{k}$ and $\Omega_{k}\subset \Omega$ with $\mathbb{P}(\Omega_{k})=1$ such that for any $\pmb{\omega} \in \Omega_k$: $\mathcal{L}eb_1(\Lambda_{\mathcal{F}}(p, \pmb{\omega})\cap F_{\frac{1}{k}})=0$. 
It follows that $\mathcal{L}eb_1(\Lambda_{\mathcal{F}}(p, \pmb{\omega}))<\frac{N}{k}$ holds 
for any $\pmb{\omega} \in \Omega_{k}$.
Let $\widetilde{\Omega}=\bigcap\limits_{k=1}^{\infty}\Omega_{k}$ then $\mathbb{P}(\widetilde{\Omega})=1$ and for any $\pmb{\omega} \in \widetilde{\Omega}$
we get
$\mathcal{L}eb_1(\Lambda_{\mathcal{F}}(p, \pmb{\omega}))=0$.  
\end{proof}

The rest of this section is devoted to the proof of  Theorem \ref{y50}. It is a simple consequence of a combination of a number of theorems due to Ruiz \cite{ruiz2009dimension} and jointly to Bárány and Rams \cite{barany2014dimension}. First we state these assertions and then we provide the proof of Theorem \ref{y50}.

\begin{fact}\label{y85}
For any $\ell\in [N] $, $b\in [L] $ and $\underline{a}\in [L] ^n$, $n\geq 0$:
$$\nu(J^{\ell}_{b\underline{a}})=\sum_{k\in [N] }\frac{1}{M}A_{b}(\ell, k)\nu(J_{\underline{a}}^k).$$
\end{fact}
This fact can be proved by the same argument used in 
 \cite[bottom of page 354]{ruiz2009dimension}.

For an $\underline{a}\in [L] ^n$ we form the vectors:
\begin{equation}\label{y84}
    \pmb{\nu}(., \underline{a})=\left(\nu(J^0_{\underline{a}}), \dots, \nu(J^{N-1}_{\underline{a}})\right), \quad \pmb{\nu}(., \emptyset)=\left(\nu(J^0), \dots, \nu(J^{N-1})\right).
\end{equation}
Fact \ref{y85} implies that $\nu(J^{\ell}_{\underline{a}})=\frac{1}{M^n} \underline{e}_{\ell}^T\cdot A_{\underline{a}}\pmb{\nu}(., \emptyset)$, where $\underline{e}_{\ell}\in\mathbb{R}^N$ is the
$\ell $-th coordinate unit vector.
Let 
\begin{equation}\label{y83}
    \eta_j:= \nu_j \circ \mathfrak{t}_{j}, \quad \eta:= \sum_{j \in [ N] }\eta_j
\end{equation}
where $\nu_j= \nu_{|J^j}$.

It is easy to see, that $\eta$ is probability measure on $[0,L]$. We also define 
\begin{equation}\label{y82}
\widetilde{\eta}([a_1,\dots, a_n]):= \eta(J^0_{a_1,\dots, a_n}).
\end{equation}
Observe that 
\begin{equation}
    \eta(J^0_{\underline{a}})= \sum_{j\in [N] }\nu(J^k_{\underline{a}})= \frac{1}{M^n}\underline{e}^TA_{\underline{a}}\pmb{\nu}(., \emptyset).
\end{equation}

Now we briefly explain the intuition behind the method presented below.
Instead of studying the original attractor we consider the one which we get by intersecting the original attractor with the different intervals $J^0, \dots, J^N$, and translate all of the resulting intersection sets to $[0,1]$. The measure $\eta$ can be thought of as the natural measure of the modified system corresponding to the modified attractor. Also to be able to connect the matrices $A_{a}$ and their products to the behaviour of the system it is more convenient for us to work in the symbolic space $\Sigma^{(L)}$ containing the $L$-adic codes of the points of $[0,L]$ (the support of our new attractor).
\begin{lemma}\label{y81}
The matrix $A=\sum_{k\in [N] }A_k$ is primitive, meaning that there exists a $K> 0$ such that $A^K(i,j)> 0$ for every $i,j\in \left\{1,\dots, N\right\}$.
\end{lemma}                                                 \begin{proof}
It is proven in \cite[Section 4.4]{SBS_book}, but for the convenience of the reader we present the proof of this lemma.

Recall that $|I|=\widetilde{n}L$.
Let $d_{\ell}$ be the distance between $x_{\ell}$ and the nearest end-point of $J_{\ell}$, let $d'>0$ be the minimum of the $d_{\ell}$s. We choose $K$ such that $|f_{\mathbf{i}^{\ell}|_{K}}(I)|=\frac{\widetilde{n}}{L^{K-1}} < d'$. Using this and the fact that $x_{\ell} \in f_{\mathbf{i}^{\ell}|_{K}}(I)$ we obtain that $f_{\mathbf{i}^{\ell}|_{K}}(I) \subset J^{\ell}$. It follows that for any $u \in [N] $
$$
f_{\mathbf{i}^{\ell}|_{K}}(J^u)\subset J^{\ell}.
$$
Hence there exists an $\underline{a}^u\in [L] ^K$ such that $f_{\mathbf{i}^{\ell}|_{K}}(J^u)=J_{\underline{a}^u}^{\ell}$. It follows that $A_{\mathbf{i}^{\ell}|_{K}}(\ell, u)>0$ and thus for every $\ell$ and $u$ 
$$
A^K(\ell, u)= \sum_{i_1, \dots, i_K} A_{i_1, \dots, i_K}(\ell, u) >0.
$$
\end{proof} 
\begin{lemma}\label{y79}
$\widetilde{\eta}$ is $\sigma$-invariant and mixing.
\end{lemma}
\begin{proof}
The proof is a slightly modified version of the proof  in \cite[Lemma 3.4]{barany2014dimension}. In our case instead of using \cite[Lemma 3.1]{barany2014dimension} we use the previous lemma -- Lemma \ref{y81}.
\end{proof}
\begin{lemma}\label{y41}
$\text{dim}_H\widetilde{\eta}<1$.
\end{lemma}
\begin{proof}
The fact that $\text{dim}_H\nu<1$ holds was made explicit first in \cite[Theorem 44]{Torma} and 
the proof of this fact is also available in
\cite[Section 4.4.4]{SBS_book}.
On the other hand,
$\text{dim}_H\widetilde{\eta}=\text{dim}_H\nu$
 follows from the arguments in \cite[page 357]{ruiz2009dimension}.
\end{proof}
Recall that we denoted $[M]^{\N}$ by $\Sigma^{(M)}$. Moreover, $\Sigma^{(M)}$ is equipped with the metrics 
$$
\rho(\mathbf{i}, \mathbf{j}):=L^{-|\mathbf{i}\land \mathbf{j}|}.
$$
For any $U\subset \Sigma^{(M)}$ we have 
\begin{equation}\label{y40}
    \overline{\text{dim}}_{B}(U)= \lim_{n\to \infty}\frac{\log \# \mathcal{G}_n(U)}{n \log L},
\end{equation}
where $\mathcal{G}_n(U)$ was defined in \eqref{u79}.
Let $\tau: \Sigma^{(L)} \to \Sigma^{(M)}$,
\begin{equation}\label{y37}
    \tau(\mathbf{a}):=\left(\Pi^{(M)}\right)^{-1}\left(\bigcup_{\ell \in [N]  } \Xi_{\ell}(\mathbf{a})\right).
\end{equation}

Now we briefly explain the meaning of this function. Recall that $\Xi_{\ell}$ is the composition of a projection from the symbolic space $\Sigma^{(L)}$ to $[0,L]$ and a translation to the $\ell$-th interval, $J^\ell$. Hence what we do with $\tau$ is that we take an $L$-adic representation $\mathbf{a}\in \Sigma^{L}$, then consider the point $x$ with $L$-adic representation $\mathbf{a}$ and translate it with the left-endpoints of the intervals $J^1, \dots, J^N$ to get the points $x_1, \dots, x_{N}\in  \R$. Then consider the  symbolic space corresponding the iterated function system,  $\Sigma^{M}$. The resulting set consists of those points of $\Sigma^{M}$ which natural projection ($\Pi^{(M)}$) is one of $x_1, \dots, x_{N}$.

The relations between the different projections are indicated by the following commutative diagram.

 \begin{equation}\label{y64}
 \large{
    \xymatrix{ 
    \Sigma^{(M)}    \ar[d]_{\Pi^{(M)}(.)}  & 
    \Sigma^{(L)} \ar[l]_{\tau(.)} \ar[d]^{\Pi^{(L)}(.)}   \ar[ld]_{\Xi_{k}(.)}
                       \\ %
    J^k &
    [0,L] \ar[l]^{\mathfrak{t}_{k}(.)} }}%
\end{equation}


In the following section our goal is to  upper bound the upper box dimension of the set $\tau(\mathbf{a})$ for a large set of $\mathbf{a}\in \Sigma^{(L)}$. 
Recall that for any $\mathbf{a}\in \Sigma^{(L)}$: 
$$
\overline{\text{dim}}_B(\tau(\mathbf{a}))=\limsup_{n\to \infty}\frac{\log\#\mathcal{G}_{n}(\tau(\mathbf{a}))}{n \log L}.
$$
In what follows we  give an upper bound on  $\#\mathcal{G}_{n}(\tau(\mathbf{a}))$.

\begin{lemma} \label{y76}
For any $\mathbf{a}=(a_1,\dots, a_n,\dots)$ we have
$\#\mathcal{G}_n(\tau(\mathbf{a}))\leq ||A_{a_1,\dots, a_n}||$.
\end{lemma}

\begin{proof}
Fix $\mathbf{a}\in \sigma^{(M)}$.
Recall that the definition of $\mathcal{G}_{n}(\tau(\mathbf{a}))$ is the following:
$$\mathcal{G}_{n}(\tau(\mathbf{a}))=\left\{(j_1,\dots,j_n)\in [M]^n: [j_1\dots j_n]\cap \tau(\mathbf{a})\neq \emptyset\right\}.$$
Hence, for a $(j_1,\dots, j_n)\in [M]^n$ we have
$(j_1, \dots, j_n) \in \mathcal{G}_{n}(\tau(\mathbf{a}))$
if there exists an $\mathbf{i} \in \Sigma^{(M)}$ with $\mathbf{i}|n=j_1, \dots, j_n$ and an $\ell\in [ N]  $ such that $\lim_{k\to \infty}f_{\mathbf{i}_{|k}}(0)=\Xi_{\ell}(\mathbf{a})\in J^{\ell}_{a_1,\dots, a_n}$. It follows that $\Xi_{\ell}(\mathbf{a})\in f_{j_1\dots j_n}(\Lambda) \subset f_{j_1\dots j_n}(\cup_{k\in [ N]  }J^k)$. Therefore there exists a $k$ such that $\Xi_{\ell}(\mathbf{a})\in f_{j_1\dots j_n}(J^k)$.
Since for any interval $J^{k}$ there exists $\ell$ and $\underline{b}=(b_1, \dots, b_n)$ such that $f_{j_1, \dots, j_n}(J^k)=J^{\ell}_{\underline{b}}$,  it follows from 
$J^{\ell}_{\mathbf{a}_{|n}}\cap f_{j_1\dots j_n}(J^k)\neq \emptyset$, that $J^{\ell}_{\mathbf{a}_{|n}}=f_{j_1\dots j_n}(J^k)$.
Let
$$
U_{n}^{k, \ell}(\mathbf{a}):= \#\left\{(j_1, \dots, j_n)\in [M]^n: f_{j_1, \dots, j_n}(J^k)=J_{{\mathbf{a}}_{|n}}^{\ell}\right\}.
$$
It is easy to see from the definition of $A_{a_1, \dots, a_n}$ and 
the argument presented above
$U_{n}^{k, \ell}(\mathbf{a})$ and the above that
\begin{equation*}
\begin{split}
    & \#\left\{(j_1, \dots, j_n) \in [M]^n: [j_1\dots j_n] \cap \tau(\mathbf{a})\neq \emptyset\right\} 
     \\
    & \leq \#\left\{(j_1, \dots, j_n)\in [M]^n: \exists k, \exists \ell \, f_{j_1, \dots, j_n}(J^k)=J^{\ell}_{a_1, \dots, a_n}\right\} 
     \\
    & \leq \#\bigcup_{k\in [ N]  }\bigcup_{\ell \in [ N]  }\left\{ (j_1, \dots, j_n)\in [M]^n: f_{j_1, \dots, j_n}(J^k)=J^{\ell}_{a_1, \dots, a_n}\right\}
     \\
    & \leq \sum_{k\in [ N]  }\sum_{\ell \in [ N]  } U_{n}^{k, \ell}(\mathbf{a})
    =
    ||A_{a_1, \dots, a_n}||
    \end{split}
\end{equation*}
\end{proof}
It follows from Lemma \ref{y76}, that 
\begin{equation*}
    \overline{\text{dim}}_{B}(\tau(\mathbf{a})) \leq \limsup_{n\to \infty}\frac{\log(||A_{a_1, \dots, a_n}||) }{n \log(L)}
\end{equation*}
%
Let us define the pressure function $P(t)$ as 
\begin{equation*}
    P(t):= \lim_{n \to \infty} \frac{1}{n \log (L)}\log \sum_{\substack{(a_1,\dots, a_n) \\ \in [ L]  ^n}}(||A_{a_1,\dots, a_n}||)^t.
\end{equation*}
Since $\#[L]  ^n=L^n$, it is easy to see, that $P(0)= 1$. Also, $\sum\limits_{(a_1,\dots, a_n)\in [ L]  ^n}(||A_{a_1,\dots, a_n}||)=||A^n||$ and from the meaning of the matrices $A$ and $A^n$, it follows that $||A^n||=M^n$, hence $P(1)=\frac{n\log M}{n \log L}=s.$
\begin{lemma} \label{y70}
The function $t \to P(t)$
\begin{itemize}
    \item exists for $t \in \R$,
    \item is monotone increasing, convex, continuous,
    \item is continuously differentiable for $t>0$.
\end{itemize}
\end{lemma}

\begin{proof}
Since the matrix $A$ is primitive (see Lemma \ref{y81}), the proof is similar to the proof of Lemma 4.3 in \cite{barany2014dimension}.
\end{proof}

The following lemma is also a slightly modified version of lemma 4.5. in\cite{barany2014dimension}.

\begin{lemma} \label{y69}
There exists a unique ergodic, left-shift invariant Gibbs measure $\mu_{1}$ on $\Sigma^{(L)}$ such that, there exists a $C>0$ that for any $(a_1, \dots, a_n) \in [ L]  ^n$ we have
\begin{enumerate}
    \item  $C^{-1} \leq \frac{\mu_{1}([a_1, \dots, a_n])}{||A_{a_1, \dots, a_n}||M^{-n}}\leq C$;
    \item $\text{dim}_{H}\mu_{1}=-P'(1)+s$;
    \item $\lim_{n \to \infty}\frac{\log ||A_{a_1, \dots, a_n}||}{n \log (L)}$.
\end{enumerate}
\end{lemma}
From lemma \ref{y79} we know that $\widetilde{\eta}$ is also an ergodic measure on $\Sigma^{(L)}$, hence it follows, that $\widetilde{\eta}=\mu_1$. 

Hence by Lemma \ref{y41} $\text{dim}_{H}\mu_1=\text{dim}_{H}\widetilde{\eta}<1$ thus by the second part of Lemma \ref{y69}
$$
P'(1)>s-1.
$$
From the convexity and continuity of the pressure function it follows that
\begin{equation} \label{y68}
    \inf_{t>0}(P(t)-(s-1)t)<1,
\end{equation}
see Figure \ref{y45} for a visual explanation.
\begin{figure}[h!]
    \centering
    \includegraphics[width=0.3\textwidth]{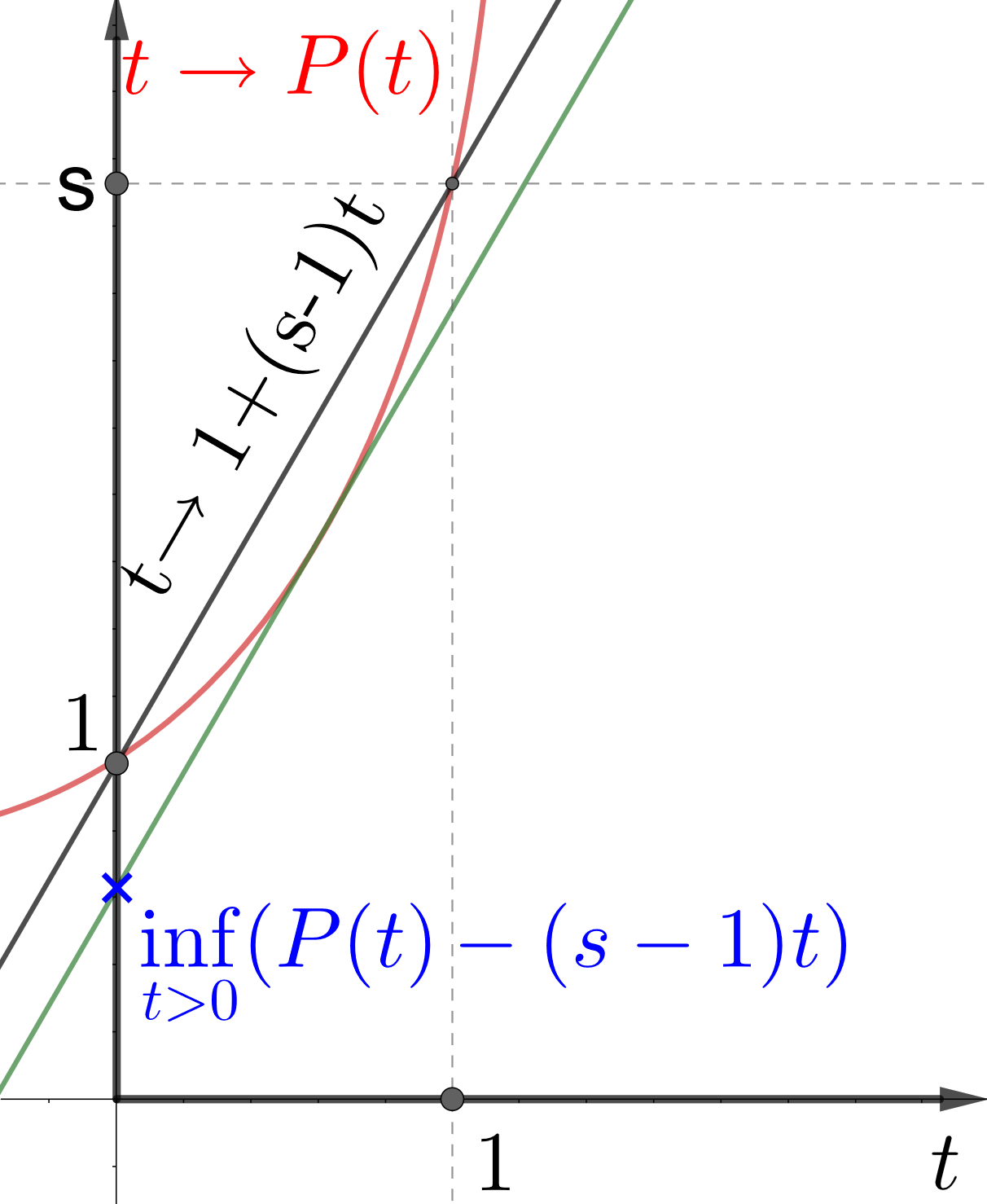}
    \caption{Explanation of \eqref{y68}.}
    \label{y45}
\end{figure}
\begin{lemma} \label{y67}
\begin{multline}
\text{dim}_{H}\left\{\mathbf{a}\in \Sigma^{(L)}: \lim_{n \to \infty} \frac{1}{n \log L} \log ||A_{\mathbf{a}_|n}||\geq s-1\right\} \\
\leq \inf_{t>0}(P(t)-(s-1)t).
\end{multline}
\end{lemma}
\begin{proof}
The proof is the same as the proof of Lemma 4.7 in \cite{barany2014dimension}.
\end{proof}

From Lemma \ref{y67} and equation \eqref{y68} it follows that
\begin{equation*}
    \text{dim}_{H}\left\{\mathbf{a}\in \Sigma^{(L)}: \lim_{n \to \infty} \frac{1}{n \log L} \log ||A_{\mathbf{a}\left|\right.n}||\geq s-1\right\}<1 
\end{equation*}

\begin{proof}[Proof of Theorem \ref{y50}]

Let $\varphi_{n}:\Sigma^{(L)}\to \R$,
\begin{equation*}
    \varphi_{n}(\mathbf{a})=\log ||A_{\mathbf{a}|n}||.
\end{equation*}
From the properties of the 1-norm, it is easy to see, that 
\begin{equation*}
    ||A_{a_1, \dots, a_n, a_{n+1},\dots, a_{n+m}}|| \leq ||A_{a_1, \dots, a_n}||\cdot ||A_{n+1}\dots A_{n+m}||,
\end{equation*}
hence, $\varphi_{n+m}(\mathbf{a})\leq \varphi_{n}(\mathbf{a})+\varphi_{m}(\sigma^{n}\mathbf{a})$, i.e. the function $\varphi$ is subadditive.
Recall, that  $\mathcal{\widetilde{L}}=(\frac{1}{L}, \dots, \frac{1}{L})^{\N}$ denotes the uniform distribution measure on $\Sigma^{(L)}$. $\mathcal{\widetilde{L}}$ is $\sigma$-invariant and ergodic. Also clearly $\varphi_{1}$ is an $L^1(\mathcal{\widetilde{L}})$ function. Hence we can use the Subadditive ergodic theorem (\cite{walters2000introduction}, Theorem 10.1), which states that there exists a $\varphi$ such that 
\begin{enumerate}
    \item $\varphi \in L^1(\mathcal{\widetilde{L}})$,
    \item $\varphi \circ \sigma = \varphi$ almost everywhere,
    \item $\lim_{n\to \infty}\frac{1}{n}\varphi_{n}=\varphi$ almost everywhere,
    \item $\lim_{n \to \infty}\frac{1}{n}\int \varphi_{n}d\mathcal{\widetilde{L}}=\int \varphi d\mathcal{\widetilde{L}}$
\end{enumerate}
Since $\mathcal{\widetilde{L}}$ is ergodic measure and $\varphi(\mathbf{a})=\varphi(\sigma \mathbf{a})$ for $\widetilde{\mathcal{L}}$ almost every $\mathbf{a} \in \Sigma^{(L)}$ there exists a $w$ such that $\varphi(\mathbf{a})=w$ for $\mathcal{\widetilde{L}}$ almost every $\mathbf{a}\in\Sigma^{(L)}$.  
\end{proof}

 \section{Appendix}
 \subsection{Random Sierpi\'nski carpet}
\begin{figure}[!htb]
\centering
\minipage{0.6\textwidth}
\minipage{0.3\textwidth}
\raggedright
  \includegraphics[width=\linewidth]{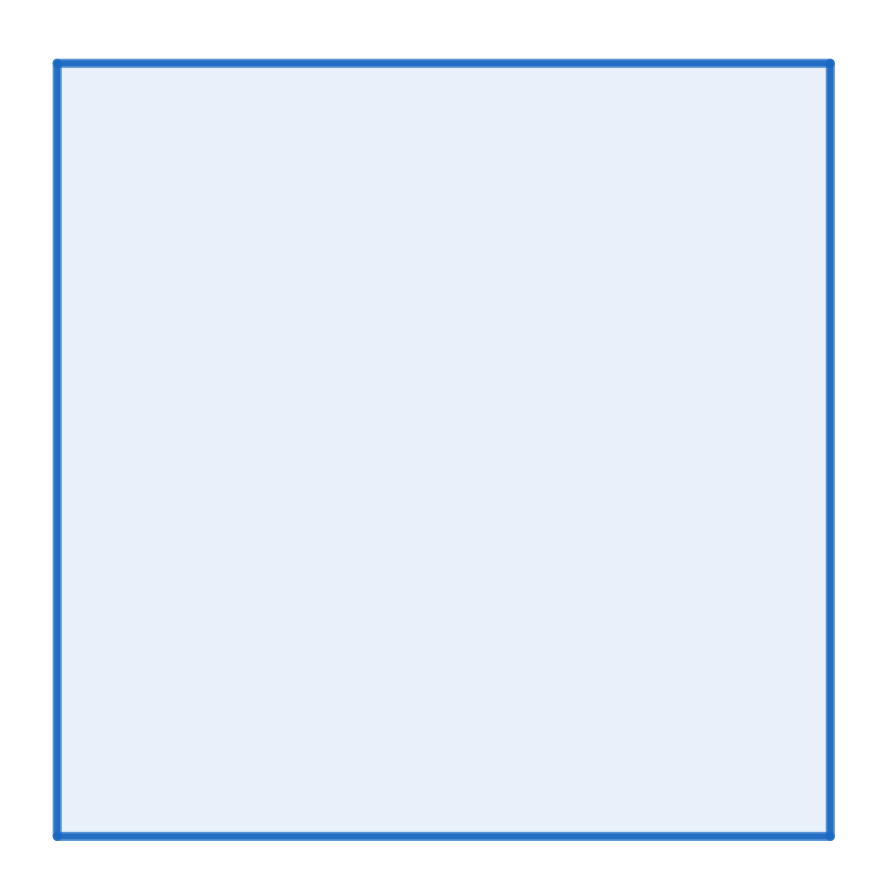}

\endminipage\hfill
\minipage{0.3\textwidth}
\centering
  \includegraphics[width=\linewidth]{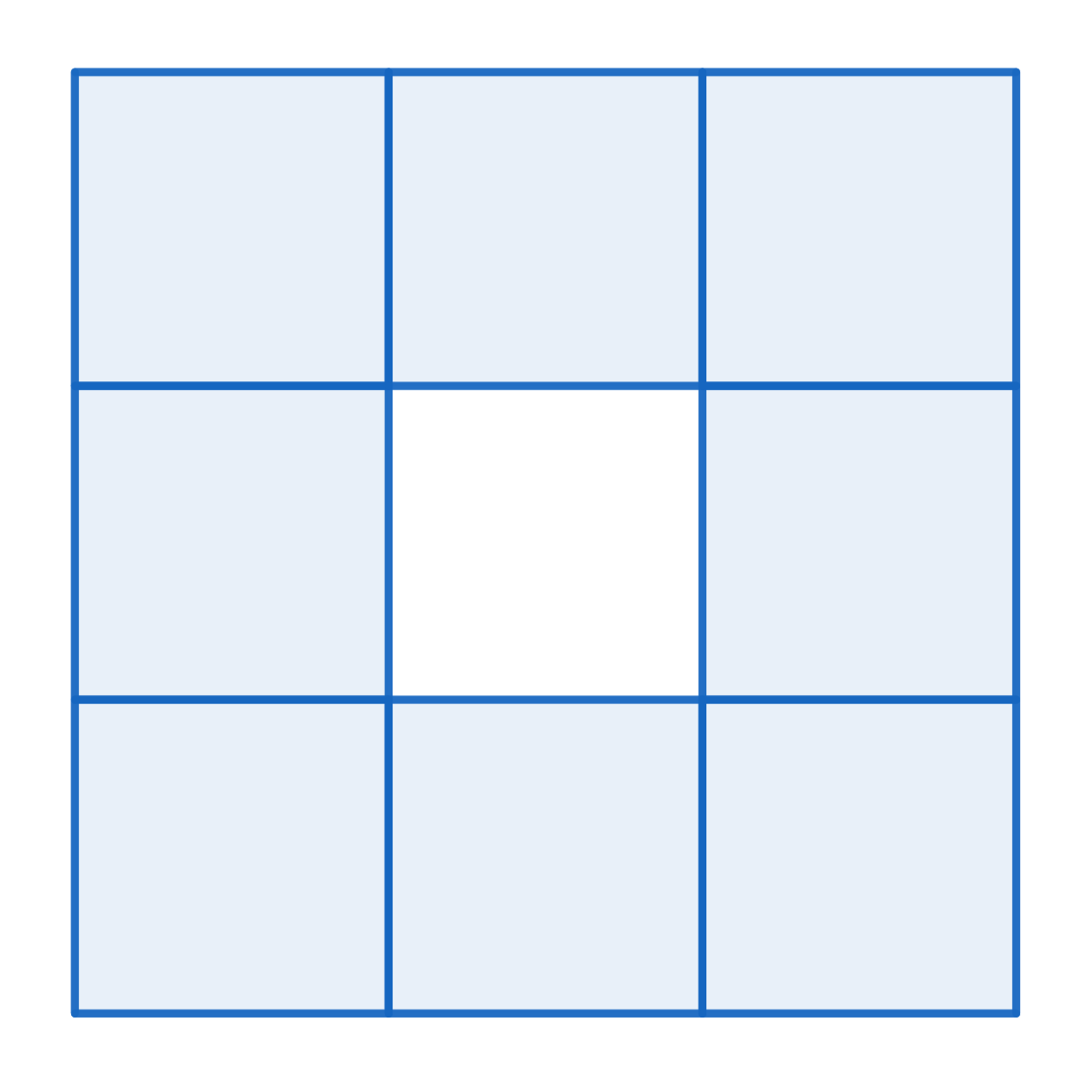}

\endminipage\hfill
\minipage{0.3\textwidth}
\raggedleft
  \includegraphics[width=\linewidth]{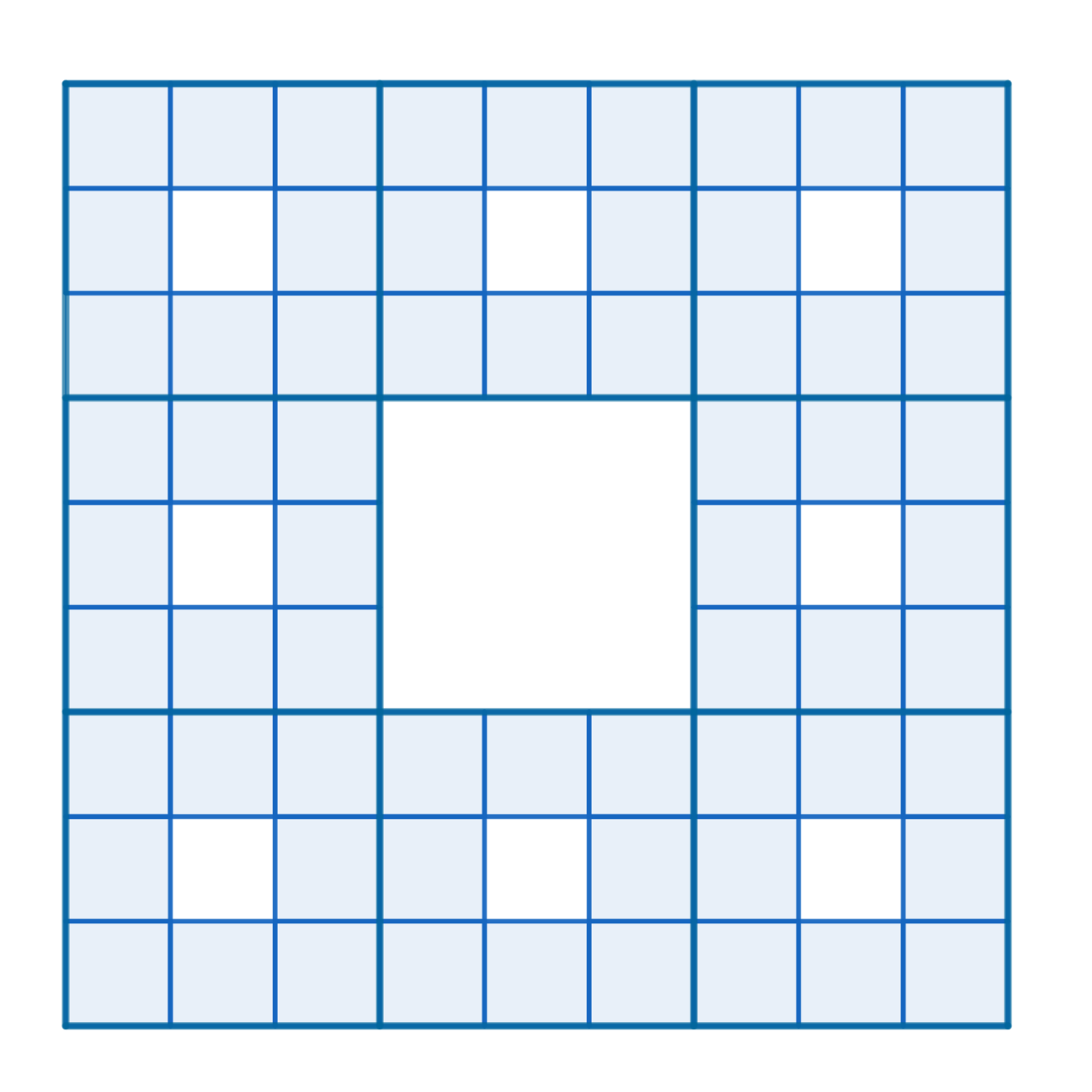}

\endminipage\hfill
\endminipage
\caption{The first three level approximation of the  Sierpi\'nski carpet.}
\label{u19}
\end{figure}

\begin{figure}[ht]
    \centering
    \includegraphics[width=0.5\textwidth]{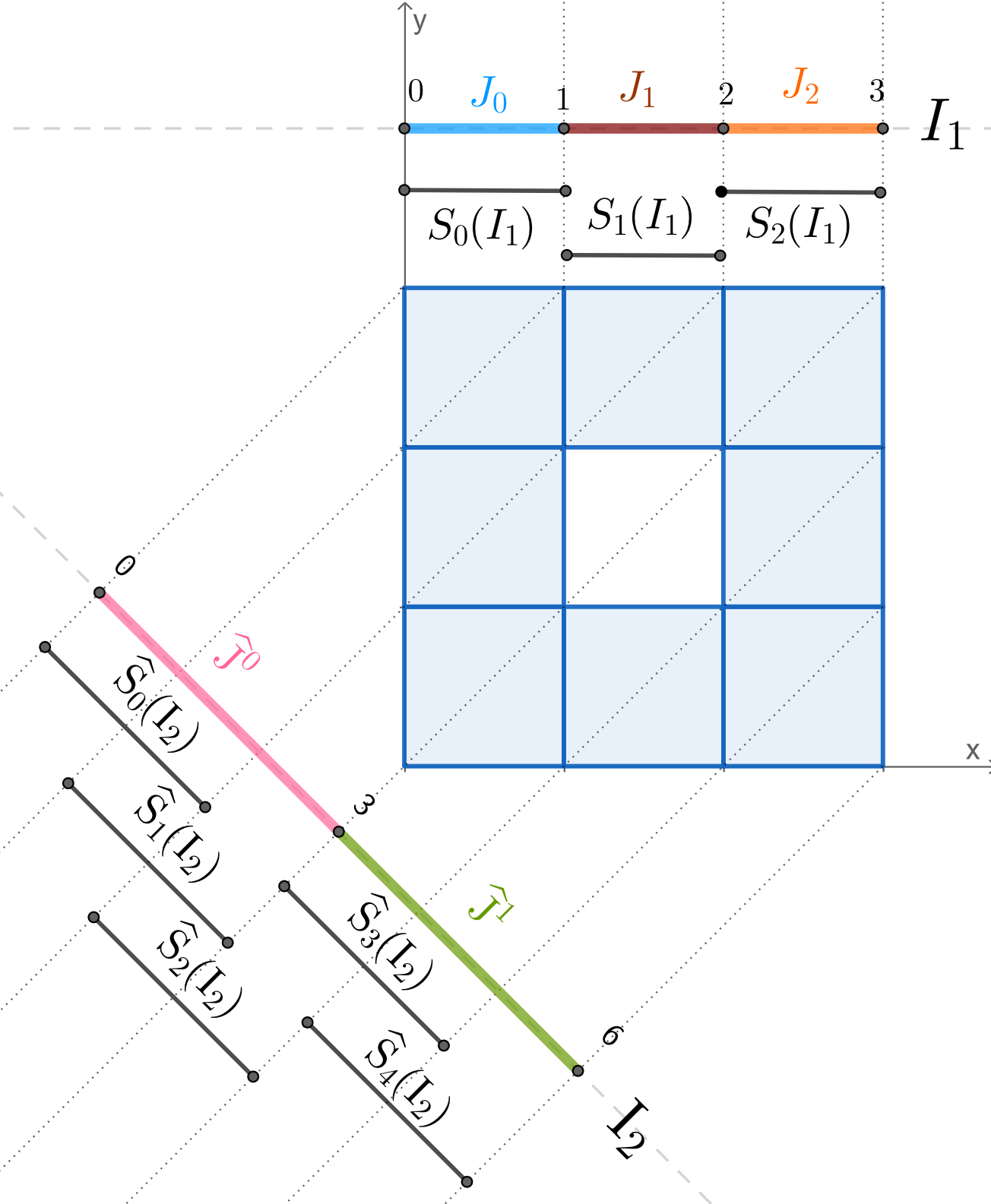}
    \caption{Illustration of the $\widehat{\text{proj}}_{(1,-1)}$ and the $\widehat{\text{proj}}_{(1,0)}$ projection of the first level of the Sierpi\'nski carpet.}
    \label{u39}
\end{figure}
The (deterministic) Sierpi\'nski carpet is the attractor (see Figure \ref{u50}) of the following self-similar IFS in $\R^2$:
\begin{equation}\label{u48}
\mathcal{F}:
=\left\{f_{i}(\underline{x})=\frac{1}{3}(\underline{x})+t_i\right\}_{i=0}^{8},
\end{equation}
where $\left\{t_i\right\}_{i=0}^{8}$ is an enumeration of the following set
\begin{equation}\label{u35}
\left\{0,\frac{1}{3},\frac{2}{3}\right\}^2 \setminus \left\{\left(\frac{1}{3}, \frac{1}{3}\right)\right\}
\end{equation}
We obtain the random Sierpi\'nski carpet by applying the random construction introduced in Definition \ref{y36} for the deterministic IFS above. We denote the random Sierpi\'nski carpet with parameter $p$ by $\mathcal{S}_p$.

Let $\widehat{\text{proj}}_{(\alpha,\beta)}$ ($\alpha,  \beta\in \R$) denote the following projection to $\R$:
\begin{equation}\label{u42}
\widehat{\text{proj}}_{(\alpha,\beta)}(x,y)=\alpha\cdot x+\beta \cdot y,
\end{equation}
which is the orthogonal projection to the line of tangent $\beta/\alpha$ rescaled. In what follows we consider the $(1,-1)$ and the $(1,0)$ projections of the random Sierpi\'nski carpet as it is schematically illustrated in Figure \ref{u39}. The behaviour of the projections of $\mathcal{S}_p$ is quite well known (see for example \cite{SR2015}, \cite{SV2018}). In particular, about the $\widehat{\text{proj}}_{(1,0)}$ projection we know almost everything from  \cite{Dekking1990}. However, by Theorems \ref{y26}-\ref{y24} we can further clarify the picture, as we explain below. 
\begin{example}[$\widehat{\text{proj}}_{(1,0)}(\mathcal{S}_p)$]\label{u41}
The projected and rescaled IFS is the following (see Definition \ref{y14}):
$S_i(x)= \frac{1}{3}x+t_i$, $t_i \in \left\{0,1,2\right\}$, with $\left(n_0,n_1,n_2\right)=\left(3,2,3\right)$, and basic types (see \eqref{u67}) $J^0=\left[0,1\right]$, $J^1=\left[1,2\right]$, $J^2=\left[2,3\right]$. Hence, the corresponding matrices (see \eqref{y89}) are,
\begin{equation}
D_0=
\begin{bmatrix}
3 & 0 & 0\\
2 & 0 & 0 \\
3 & 0 & 0
\end{bmatrix}, \quad
D_1=
\begin{bmatrix}
0 & 3 & 0\\
0 & 2 & 0 \\
0 & 3 & 0
\end{bmatrix}, \quad
D_2=
\begin{bmatrix}
0 & 0 & 3\\
0 & 0 & 2 \\
0 & 0 & 3
\end{bmatrix}.
\end{equation}
In this case the 
spectral radiuses of $p\cdot D_{i}$, $i\in \{0,1,2\}$ are $3p,2p,3p$ respectively. 
\end{example}

\begin{example}[$\widehat{\text{proj}}_{(1,-1)}
(\mathcal{S}_p)$]\label{u34}
The projected and rescaled IFS is the following (see Definition \ref{y14}):
$\widehat{S}_i(x)= \frac{1}{3}x+t_i$, $t_i \in \left\{0,1,2,3,4\right\}$, with $\left(n_0,n_1,n_2, n_3, n_4\right)=\left(1,2,2,2,1\right)$, and basic types (see \eqref{u67}) $\widehat{J}^0=\left[0,3\right]$, $\widehat{J}^1=\left[3,6\right]$. Hence, the corresponding matrices (see \eqref{y89}) are,
\begin{equation}
C_0=
\begin{bmatrix}
1 & 0 \\
2 & 2  \\

\end{bmatrix}, \quad
C_1=
\begin{bmatrix}
2 & 1 \\
1 & 2  \\

\end{bmatrix}, \quad
C_2=
\begin{bmatrix}
2 & 2 \\
0 & 1  \\
\end{bmatrix}.
\end{equation}
The spectral radiuses of $p\cdot C_i$ ($i \in \{0,1,2\}$) are $2 p, 3 p ,2p $ respectively.
\end{example}

As we mentioned already the $\widehat{\text{proj}_{(1,0)}}$ projection of $\mathcal{S}_p$ is heavily studied by \cite{Dekking1990}. Their results are illustrated in Figure \ref{u18}. What we can add to this is as follows: \begin{enumerate}
    \item by Theorem \ref{y21} we know that for $p<\frac{1}{2}$ the projection $\widehat{\text{proj}_{(1,0)}}$ does not contain any intervals almost surely.
\end{enumerate}
It is clear from the results in \cite{Dekking1990} that for $p>\frac{1}{2}$ the projection $\widehat{\text{proj}_{(1,0)}}(\mathcal{S}_p)$ contains an interval almost surely conditioned on non-extinction. 

\begin{figure}[ht]
    \centering
    \includegraphics[width=\textwidth]{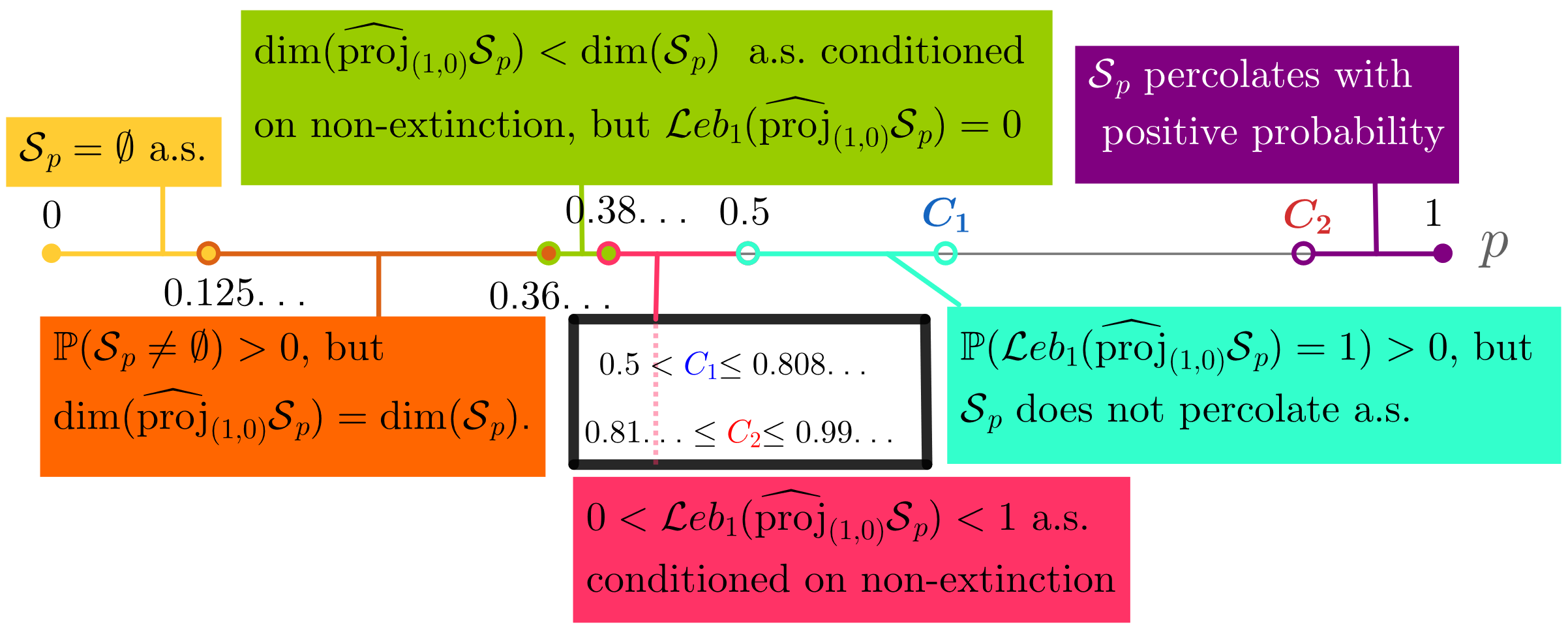}
    \caption{Illustration of the results of Dekking and Meester \cite{Dekking1990} on the $\widehat{\text{proj}}_{(1,0)}$ phases of the projection of the random Sierpi\'nski carpet.
    It follows from our results that the content of the red rectangle can be replaced with "
    $\widehat{\text{proj}}_{(1,0)}(\mathcal{S}_p)$ has positive Lebesgue measure a.s. conditioned on non-extinction but empty interior a.s.hat".
    } 
    \label{u18}
\end{figure}
Moreover,  it is  stated in \cite{SR2015},
 that this, last result holds for every direction, especially for the $(-1, 1)$ direction too. About the $\widehat{\text{proj}}_{(-1,1)}$ projection we further know from \cite{SV2018} that there exists a $p'>\frac{3}{8}$, such that for $p<p'$, $\text{dim}_{H}(\mathcal{S}_p)>1$ almost surely conditioned on non-extinction but the $\widehat{\text{proj}}_{(-1,1)}$ projection does not contain any intervals almost surely. Using Theorems \ref{y26}-\ref{y24} we can extend these results as follows:
\begin{enumerate}
    \item there exists a $p''$ such that for $p<p''$,
    $\text{dim}_{H}(\mathcal{S}_p)>1$ almost surely conditioned
    on non-extinction but the 
    projection 
    $\widehat{\text{proj}}_{(-1,1)}(\mathcal{S}_p)$
    has not only empty
    interior but also zero Lebesgue
    measure.
    \item For $p>0.38\dots=\frac{1}{18^{^1/_3}}$ (note that this is the same value that occurred in the other projection $\widehat{\text{proj}}_{(1,0)}$ in Figure \ref{u18}) the projection
    $\widehat{\text{proj}}_{(-1,1)}(\mathcal{S}_p)$
    has positive Lebesgue measure almost surely conditioned on non extinction.
    \item For $p<\frac{1}{2}$ the projection
    $\widehat{\text{proj}}_{(-1,1)}(\mathcal{S}_p)$
    does not contain an interval almost surely. 
\end{enumerate}

The Menger sponge is the $3$-dimensional analogue of the Sierpi\'nski carpet and the projections we studied in this chapter are analogous to those three dimensional ones which we denoted by $\text{proj}_{(1,0,0)}$ and $\text{proj}_{(1,1,1)}$. Nonetheless unlike in the case of the Menger sponge (where the $p$ parameter intervals were very different in case of the two projections) here the two projections seemingly shows a very similar behavior for a given $p$.

\bibliographystyle{plain}
\bibliography{references}
\end{document}